\documentclass[11pt]{amsart}
\usepackage{mathtools,amsmath,amssymb}
\usepackage{amsthm}
\usepackage{graphicx}
\usepackage[dvipsnames]{xcolor}
\usepackage{tikz-cd}
\usepackage[OT2,T1]{fontenc} 
\usepackage{bbm}
\usepackage{stmaryrd}
\usepackage[hidelinks]{hyperref}
\usepackage{musicography}
\usepackage{tikz}
\usepackage{geometry}
\usepackage{bbm}
\usepackage[backend=biber,style=alphabetic,sorting=nyt,maxnames=99]{biblatex}
\renewbibmacro{in:}{}
\usepackage{accents}
\usepackage{listings}

\allowdisplaybreaks

\mathtoolsset{showonlyrefs,showmanualtags}

\DeclareFieldFormat{postnote}{#1}
\DeclareFieldFormat{multipostnote}{#1}

\DeclareFieldFormat[article]{citetitle}{\mkbibemph{#1}}
\DeclareFieldFormat[article]{title}{\mkbibemph{#1}}
\DeclareFieldFormat[article]{volume}{\mkbibbold{#1}}
\DeclareFieldFormat[article]{journaltitle}{#1}

\DeclareFieldFormat[inbook]{citetitle}{\mkbibemph{#1}}
\DeclareFieldFormat[inbook]{title}{\mkbibemph{#1}}

\DeclareFieldFormat[inproceedings]{citetitle}{\mkbibemph{#1}}
\DeclareFieldFormat[inproceedings]{title}{\mkbibemph{#1}}
\DeclareFieldFormat[inproceedings]{booktitle}{#1}

\DeclareFieldFormat[incollection]{citetitle}{\mkbibemph{#1}}
\DeclareFieldFormat[incollection]{title}{\mkbibemph{#1}}
\DeclareFieldFormat[incollection]{booktitle}{#1}

\defbibheading{bibliography}{}

\theoremstyle{plain}
\newtheorem{theorem}{Theorem}
\numberwithin{theorem}{section}

\newtheorem{lemma}[theorem]{Lemma}
\newtheorem{prop}[theorem]{Proposition}

\newtheorem{conjecture}[theorem]{Conjecture}

\theoremstyle{definition}

\newtheorem{remark}[theorem]{Remark}
\newtheorem*{remark*}{Remark}

\newtheorem*{claim*}{Claim}

\newtheoremstyle{named}{}{}{\itshape}{}{\bfseries}{.}{.5em}{\thmnote{#3's }#1}
\theoremstyle{named}

\DeclarePairedDelimiter\abs{\lvert}{\rvert}%
\DeclarePairedDelimiter\ceil{\lceil}{\rceil}

\newcommand{\pdiv}{\mid\!\mid}

\DeclareMathOperator{\fin}{fin}

\DeclareMathOperator{\im}{Im}

\DeclareMathOperator{\mo}{mod}
\DeclareMathOperator{\mt}{MT}

\DeclareMathOperator{\rat}{rat}
\DeclareMathOperator{\re}{Re}

\def\C{\mathbb{C}}
\def\E{\mathbb{E}}

\def\Q{\mathbb{Q}}
\def\R{\mathbb{R}}

\def\Z{\mathbb{Z}}

\title{An asymptotic formula for a cubic moment of $GL_2$ $L$-functions}
\author{Catinca Mujdei}
\thanks{\textit{Affiliation}. University College London, London, UK}
\thanks{\textit{E-mail address}. \url{catinca.mujdei.23@ucl.ac.uk}}
\thanks{The author was supported by the Engineering and Physical Sciences Research Council [EP/S021590/1], via the EPSRC Centre for Doctoral Training in Geometry and Number Theory (The London School of Geometry and Number Theory), University College London.}

\begin{document}

\maketitle

\begin{abstract}
    We obtain an asymptotic formula for a cubic moment of self-dual $GL_2$ $L$-functions studied by Petrow and Young, with a small extra averaging over characters. Our asymptotic consists of the main term conjectured by Conrey--Farmer--Keating--Rubinstein--Snaith and a power-saving error term whose strength depends on the amount of extra averaging.
\end{abstract}

\section{Introduction}
In their seminal paper \cite{CFKRS}, Conrey--Farmer--Keating--Rubinstein--Snaith (CFKRS) formulated a heuristic ``recipe'' for computing the main terms of integral moments of families of $L$-functions. In the recipe, the $L$-functions are evaluated at small shifts away from the central point $s=1/2$, which reveals interesting symmetries of the family in the conjectured main terms.

The main result of the present work confirms the CFKRS conjecture for some cubic moments of families of automorphic $L$-functions previously studied by Conrey--Iwaniec and Petrow--Young, with an additional small average over Dirichlet characters. These cubic moments play a key role in the analytic theory of $L$-functions because of their role in Weyl-subconvexity estimates \cite[Section 6]{Mi}. The main difficulty overcome in the present paper is the identification and cancellation of all the degenerate off-diagonal main terms that the CFKRS recipe implicates in these cubic moments, provably, with a power-saving error term. These degenerate terms have often been a sticking point in the literature surrounding this cubic moment (see below for more discussion).

\subsection{Background}
Given $\kappa\in2\Z_{\geq1}$ and $q\in\Z_{\geq1}$, denote by $\mathcal{F}_\kappa(q)$ an orthonormal basis of Hecke eigenforms of weight $\kappa$ and level $q$. Let $\chi$ be a primitive Dirichlet character modulo $q$. Conrey and Iwaniec \cite{CI} proved for quadratic $\chi$, $\kappa\geq12$ and $q$ odd, square-free the Lindel\"of-on-average estimate
\begin{align*}
    \sum_{f\in\mathcal{F}_\kappa(q)}L(1/2,f\otimes\chi)^3\ll_{\kappa,\varepsilon}q^{1+\varepsilon}.
\end{align*}
Their result was later extended to all $\kappa\geq2$ by Petrow \cite{Pe}, who obtained in particular a formula roughly of the form
\begin{align}
    &\sum_{f\in \mathcal{F}_\kappa(q)}w_fL(1/2,f\otimes\chi)^3 = \text{MT}_h(\chi)+\text{DM}_h(\chi) + O_{\kappa,\varepsilon}(q^{-1/3+\varepsilon}), \label{eq:motohashi_summarized} 
\end{align}
where the $w_f=q^{-1+o(1)}$ are the natural harmonic weights facilitating the Petersson trace formula. The main term $\text{MT}_h(\chi)$ matches the CFKRS conjecture, and $\text{DM}_h(\chi)$ is a fourth moment of Dirichlet $L$-functions. The identity \eqref{eq:motohashi_summarized} is reminiscent of some formulas first derived by Motohashi \cite[Chapter 4]{M}; see also \cite[Sections 4.5.4--4.5.5]{MV}. The only bound currently available for the dual moment is $\text{DM}_h(\chi)\ll(\log q)^4$ by a large-sieve estimate, while $\text{MT}_h(\chi)\sim\frac{8}{3}(\log q)^3$. So unfortunately \eqref{eq:motohashi_summarized} does not prove the CFKRS conjecture.

The non-holomorphic analog of \eqref{eq:motohashi_summarized} was studied in \cite{PeY,PeYfourth} by Petrow and Young for general (not necessarily quadratic) primitive characters $\chi$ modulo $q$. Let $T\in\R_{\geq1}$. Write $\mathcal{H}_{it}(m,\overline{\chi}^2)$ for the set of Hecke-normalized Hecke--Maass newforms of level $m\mid q$, central character $\overline{\chi}^2$ and spectral parameter $t$. Petrow and Young showed that, crudely,
\begin{align}
    &\sum_{\abs{t}\leq T}\sum_{m\mid q}\sum_{f\in\mathcal{H}_{it}(m,\overline{\chi}^2)}w_{f}L(1/2,f\otimes\chi)^3 + \text{(Eisenstein series contribution)} \label{eq:cubefree_summarized} \\
    &= \text{MT}_{nh}(\chi) + \text{DM}_{nh}(\chi) + O_\varepsilon(T^Bq^\varepsilon) 
\end{align}
for some constant $B>0$. The harmonic weights $w_{f}$ are analogous to those in \eqref{eq:motohashi_summarized}, and now facilitate the Kuznetsov trace formula.

Morally, the dual moment on the right-hand side of \eqref{eq:cubefree_summarized} is of the form
\begin{align}
    \text{DM}_{nh}(\chi) \approx \frac{1}{q}\sum_{\psi\,(\mo q)}\abs{L(1/2,\psi)}^4g(\chi,\psi), \label{eq:dual_moment_nh}
\end{align}
featuring double character sums $g(\chi,\psi)$ defined by
\begin{align*}
    \sum_{t,u\,(\mo q)}\chi(t)\overline{\chi}(t+1)\overline{\chi}(u)\chi(u+1)\psi(tu-1)
\end{align*} 
when $c(\psi)=q$. The latter can be bounded using machinery from algebraic geometry, so after taking absolute values in the dual moment $\text{DM}_{nh}(\chi)$ one can again bound it using a large-sieve estimate. However, the authors did not show that $\text{MT}_{nh}(\chi)$ matches the CFKRS main term, as they only needed an upper bound  for it to prove their main theorems. To obtain an asymptotic formula in \eqref{eq:cubefree_summarized}, one would like to show sufficient cancellation among the terms $\lvert L(1/2,\psi)\rvert^4g(\chi,\psi)$ as $\psi$ ranges among Dirichlet characters modulo $q$, using equidistribution properties of $g(\chi,\psi)$. Unfortunately however, this seems out of reach of current methods.

In the present work, we evaluate an average of the cubic moment in \eqref{eq:cubefree_summarized} over characters, and isolate the main term conjectured by CFKRS from the diagonal of the trace formula, up to a power-saving error term. The resulting asymptotic formula is stated in Theorem \ref{thm:main} below. As far as we are aware, this is the first instance where the conjectured main term for the cubic moment of Petrow--Young (albeit taken over a slightly enlarged family) is fully recovered and shown to be larger than the dual sum.

According to very general principles of harmonic analysis, one expects that the dual moment should shrink when the initial family grows. On the other hand, the main difficulty in this paper is to prove cancellation of the remaining part of the diagonal and certain off-diagonal terms. Indeed, in several works generalizing Petrow--Young, the collection of the main terms poses a significant obstacle, see e.g. \cite{BFW,F,GHLN,Kw2,Kw3,Nel}. In \cite{BFW} and \cite{Nel}, Balkanova--Frolenkov--Wu and Nelson used the period integral approach to generalize to number fields, and were able to identify the degenerate main terms and compute upper bounds for them, but not to extract the conjectured main term (see e.g. \cite[Remark 15.2]{Nel}). Kwan \cite{Kw2,Kw3} obtained an exact reciprocity formula relating the cubic moment of untwisted/twisted $GL_2$ $L$-functions to the dual fourth moment of $\zeta(\cdot)$/Dirichlet $L$-functions, and was able to match the degenerate terms to the conjectured main terms of both moments. However, Kwan's work only concerns cusp forms of level $1$, whereas we focus on the ramified case. Moreover, Kwan does not prove an asymptotic formula for either the third or fourth moments. Our present result is thus an instance in which the classical approach might have some advantage over the period integral approach.

We also mention the work of Frolenkov \cite{F}, who established an asymptotic formula for a cubic moment of $L$-functions attached to holomorphic cusp forms in the weight aspect. His formula features the CFKRS main term along with a dual moment, and relies on classical inputs such as a twisted second moment formula of the same $L$-functions due to Kuznetsov. Rather than deploying the approximate functional equation, Frolenkov shifts the evaluation point from $s=1/2$ into the region of absolute convergence, and at a later stage applies analytic continuation. A similar strategy is pursued by Ganguly--Humphries--Lin--Nunes in \cite{GHLN}, who combine this analytic continuation method with the Kuznetsov and Petersson formulas to obtain a spectral reciprocity formula relating a $GL_2$ moment of $GL_3\times GL_2$ Rankin--Selberg $L$-functions to a sum of two main terms and a dual moment. The latter is a $GL_1$ moment of $GL_4\times GL_1$ Rankin--Selberg $L$-functions, and interestingly contains the double character sum $g(\chi,\psi)$.

\subsection{Statement of the result}\label{sec:statement}
Let $p$ be an odd prime and $a\in\Z_{\geq1}$. Let $\chi$ be a primitive Dirichlet character modulo $p^a$. For $m\mid p^a$, write $\mathcal{H}_{it,\text{Eis}}(m,\overline{\chi}^2)$ for the set of newform Eisenstein series as defined in \cite[Section 2.2]{PeY}. Define also the test function 
\begin{align*}
    h_0(t) \coloneqq \exp(-(t/T)^2)\frac{t^2+\frac{1}{4}}{T^2}.
\end{align*}
Lastly, denote by $w_{f,\ell},w_{E,\ell}$ the harmonic weights that facilitate the Kuznetsov trace formula later on; in fact, we have 
\begin{align*}
    w_{f,\ell} = q^{-1}(q(1+\abs{t}))^{o(1)} \quad \text{ and } \quad w_{E,\ell} \gg q^{-1}(q(1+\abs{t}))^{-\varepsilon}
\end{align*}
according to \cite[Section 2.4]{PeY}. We consider the cubic moment
\begin{align}
    \mathcal{M}_{nh}(\chi,\pmb{\alpha}) &\coloneqq \sum_{t}h_0(t)\sum_{\ell m=p^a}\sum_{f\in\mathcal{H}_{it}(m,\overline{\chi}^2)}w_{f,\ell}\prod_{1\leq j\leq 3}L(1/2+\alpha_j,f\otimes\chi) \label{eq:moment} \\
    &\quad\,\, + \frac{1}{4\pi}\int_{\R}h_0(t)\sum_{\ell m=p^a}\sum_{E\in\mathcal{H}_{it,\text{Eis}}(m,\overline{\chi}^2)}w_{E,\ell}\prod_{1\leq j\leq 3}L(1/2+\alpha_j,E\otimes\chi) \,dt \\
\end{align}
for shifts $\pmb{\alpha}=(\alpha_1,\alpha_2,\alpha_3)\in\C^3$.

Now take $b\in\Z_{\geq1}$ with $b<a$. Given a Dirichlet character $\phi$ modulo $p^b$, note that $\chi\phi$ has conductor $p^a$. We will study $\mathcal{M}_{nh}(\chi,\pmb{\alpha})$ through its average
\begin{align}
    \E_{\phi\,(\mo p^b)}[\mathcal{M}_{nh}(\chi\phi,\pmb{\alpha})] \coloneqq \frac{1}{\varphi(p^b)}\sum_{\phi\,(\mo p^b)}\mathcal{M}_{nh}(\chi\phi,\pmb{\alpha}). \label{eq:informal_explanation}
\end{align} 
The aim of this paper is to obtain an asymptotic formula for $\E_{\phi\,(\mo p^b)}[\mathcal{M}_{nh}(\chi\phi,\pmb{\alpha})]$. In view of the $\log(\mathrm{conductor})/\log(\abs{\mathrm{family}})$ heuristic for moments of $L$-functions, the most analytically interesting case is the one with the narrowest possible family of automorphic forms, i.e. with a value of $b$ that is as small as possible. We are therefore content to assume $b\leq a/3$, a technical choice that simplifies some of our arguments (as explained in Remark \ref{rmk:remark_intro}(iii)).

Our main result is the following.

\begin{theorem}\label{thm:main}
    Let $p$ be an odd prime, $a,b\in\Z_{\geq1}$ with $b\leq a/3$, $T=(p^a)^{o(1)}$, and $\re\alpha_j\ll \frac{1}{\log p^a}$ and $\im\alpha_j=O(1)$ for all $1\leq j\leq 3$. We have
    \begin{align}
        \E_{\phi\,(\mo p^b)}[\mathcal{M}_{nh}(\chi\phi,\pmb{\alpha})] = \sum_{\pmb{\sigma}\in\{\pm1\}^3}\frac{1+\sigma}{2}\text{MT}_{nh}(\pmb{\alpha},\pmb{\sigma}) + O_\varepsilon(p^{-b+a\varepsilon}), \label{eq:claim_asymptotic}
    \end{align}
    where
    \begin{align*}
        &\mt_{nh}(\pmb{\alpha},\pmb{\sigma}) \\
        &\coloneqq \frac{p^{a\sigma\alpha}}{2p^{a\alpha}}\zeta^{(p)}(1+\sigma_1\alpha_1+\sigma_2\alpha_2)\zeta^{(p)}(1+\sigma_1\alpha_1+\sigma_3\alpha_3)\zeta^{(p)}(1+\sigma_2\alpha_2+\sigma_3\alpha_3) \\
        &\quad\,\, \cdot \frac{1}{4\pi}\int_{\R}\tanh(\pi t)th_0(t)\sum_{\delta\in\{0,1\}}\prod_{1\leq j\leq 3}\frac{\Gamma_\R(1/2+\delta+\sigma_j\alpha_j+it)\Gamma_\R(1/2+\delta+\sigma_j\alpha_j-it)}{\Gamma_\R(1/2+\delta+\alpha_j+it)\Gamma_\R(1/2+\delta+\alpha_j-it)}\,dt
    \end{align*}
    and
    \begin{align*}
        \sigma &\coloneqq \sigma_1\sigma_2\sigma_3, \\
        \alpha &\coloneqq \alpha_1+\alpha_2+\alpha_3, \\
        \sigma\alpha &\coloneqq \sigma_1\alpha_1+\sigma_2\alpha_2+\sigma_3\alpha_3.
    \end{align*}
    In particular, as $p^a\rightarrow\infty$ one has 
    \begin{align}
        \sum_{\pmb{\sigma}\in\{\pm1\}^3}\frac{1+\sigma}{2}\mt_{nh}(\pmb{0},\pmb{\sigma}) \sim \frac{8}{6}(1-p^{-1})^3(\log p^a)^3\cdot \frac{1}{4\pi}\int_{\R}\tanh(\pi t)th_0(t)\,dt. \label{eq:explicit_asymptotic}
    \end{align}
\end{theorem}
$\ $

\begin{remark}\label{rmk:remark_intro}
    \leavevmode
    \vspace{-13pt}
    \begin{enumerate}
        \item[(i)] Note that \eqref{eq:claim_asymptotic} is not a depth-aspect result, although it does contain one. It is truly an asymptotic formula in terms of $p^a$ so long as $b$ is chosen such that $p^{-b+a\varepsilon}=o((\log p^a)^3)$ for sufficiently small $\varepsilon>0$. If we fix $a,b$ and let $p\rightarrow\infty$, then $\varepsilon\leq b/a$ suffices. If we fix $p$ and let $a\rightarrow\infty$, then we require some dependence of $b$ on $a$, e.g. $b\geq a/D$ for some fixed $D\geq3$, in which case $\varepsilon\leq 1/D$ suffices.
        
        \item[(ii)] As already remarked, one key idea in proving Theorem \ref{thm:main} is that the sum over $\phi\,(\mo p^{b})$ shortens the dual sum in \eqref{eq:cubefree_summarized}, namely from modulus $p^a$ to modulus $p^{a-b}$. We roughly obtain
        \begin{align}
            &\E_{\phi\,(\mo p^b)}[\mathcal{M}_{nh}(\chi\phi,\pmb{\alpha})] \\
            &\approx \sum_{\pmb{\sigma}\in\{\pm1\}^3}\frac{1+\sigma}{2}\text{MT}_{nh}(\pmb{\alpha},\pmb{\sigma}) + \frac{1}{p^{a-b}}\sum_{\psi\,(\mo p^{a-b})}\abs{L(1/2,\psi)}^4g(\chi,\psi) + O_\varepsilon(p^{a(-\frac{1}{3}+\varepsilon)}). \label{eq:averaged_moment_summarized}
        \end{align}
        As a consequence, we save a factor $p^{b}$ in large-sieve estimates on the dual moment, thus resulting in an error term $O_\varepsilon(p^{-b+a\varepsilon})$ in \eqref{eq:claim_asymptotic}.
        The shortening of the dual sum turns out to be particularly convenient for us, as algebro-geometric estimates on $g(\chi,\psi)$ are required only when $b=0$. We also remark that \eqref{eq:averaged_moment_summarized} is an oversimplified expression for the dual moment for the sake of exposition. The precise formula involves additional parameters and multiple Mellin inversion integrals, and is obtained by combining the results in Sections \ref{sec:evaluation_of_gauss_sums}--\ref{sec:evaluation_of_generating_functions} and \ref{sec:completing_the_proof}.

        The critical part of this work, however, is the matching of certain diagonal and off-diagonal terms in Section \ref{sec:cancellation}. Such a cancellation result was previously obtained in \cite[Section 3]{Pe} only for the case of holomorphic cusp forms, but without showing that the main term asymptotically dominates the dual moment. In any case, our present (generalized) proof for the non-holomorphic case is new, and moreover the cancellation arises from different terms in the calculation than in the holomorphic case. The argument involves Mellin--Barnes integral representations for the Fourier transforms of the Bessel transforms in the Kuznetsov formula, functional equations for $\Gamma(s)$, followed by some intricate combinatorics. One interesting aspect of the main terms that can be inferred from our proofs is that, as $\pmb{\alpha}\rightarrow\pmb{0}$, the diagonal terms are matched against those off-diagonal terms that come from the opposite-sign case of the Kuznetsov formula, whereas the off-diagonal terms of the same-sign case are absorbed by the error term $O_\varepsilon(p^{a(-\frac{1}{3}+\varepsilon)})$.
        
        \item[(iii)] We explain the source of the technical hypothesis $b\leq a/3$. The computations in Section \ref{sec:cancellation} (and those in Section \ref{sec:diag}) involve bounding contour integrals of the form
        \begin{align}
            \int_{(c)}p^{as}f(s)\frac{\Gamma_\R(1/2+\delta+s+it)\Gamma_\R(1/2+\delta+s-it)}{\Gamma_\R(1/2+\delta+\alpha+it)\Gamma_\R(1/2+\delta+\alpha-it)}\,ds \label{eq:abstract_contour_integral}
        \end{align}
        for some small shift $\alpha\in\C$, $-1/2<c<0$, $\delta\in\{0,1\}$, $t\in\R$, and $f(s)$ some function holomorphic and exponentially decaying in a vertical strip containing $s=-1/2,c$. If $\delta=0$, the integrand has poles on the $(-1/2)$-vertical line, so in order to avoid dealing with complicated residual terms, we refrain from shifting the contour to the left past $s=-1/2$. Choosing $c=-1/3$, we can bound \eqref{eq:abstract_contour_integral} as $\ll_\varepsilon p^{a(-1/3+\varepsilon)}$. The hypothesis $b\leq a/3$ permits us to dominate this term by $O_\varepsilon(p^{-b+a\varepsilon})$ in \eqref{eq:claim_asymptotic}, since the latter error term is the best we can obtain when applying large-sieve estimates to the dual moment in Section \ref{sec:completing_the_proof}. To relax the hypothesis $b\leq a/3$, it would suffice to compute the contribution to \eqref{eq:abstract_contour_integral} from residues on the $(-1/2)$-line, and show that they cancel out, as we would expect based on the CFKRS conjecture.
        
        \item[(iv)] We assume $T=(p^a)^{o(1)}$ so that we can suppress $T$-dependence in most of our computations. Our strategy likely also works with some modifications if this condition on $T$ is relaxed. Likewise, one could also study the small spectral window $t\in(T,T+\Delta]$ as in \cite[Theorem 1.2]{PeY}, but we chose not to pursue this.
         
        \item[(v)] The CFKRS conjectured asymptotic formula for $\E_{\phi\,(\mo p^b)}[\mathcal{M}_{nh}(\chi\phi,\pmb{\alpha})]$ has the same main term, and an error term $O_\varepsilon(p^{a(-1+\varepsilon)})$; see Conjecture \ref{conj:cfkrs}. We could have also used the framework of multiple Dirichlet series \cite{DGH} to produce a conjectural asymptotic formula for the moment of $L$-functions against which to compare our main theorem.

        \item[(vi)] The strategy of averaging over $\phi\,(\mo p^b)$ can also be applied to the cubic moment of holomorphic forms in \cite{Pe}. We give the statement and sketch the proof of such a result in Appendix \ref{sec:proof_hol}.

        \item[(vii)] The assumption that $\chi$ has prime-power conductor $p^a$ is for convenience, and our method of proof likely also works for composite conductors.
        
        \item[(viii)] From the outset of this paper, we assumed that proving
        \begin{align*}
            \sum_{\psi\,(\mo q)}\abs{L(1/2,\psi)}^4g(\chi,\psi) = o(q)
        \end{align*}
        on the right-hand side of \eqref{eq:cubefree_summarized} is hard. If however such a result were proved in the future, a non-averaged version of Theorem \ref{thm:main} (i.e. with $b=0$) could be obtained with minimal extra work, as our extraction of $\text{MT}_{nh}(\cdot)$ and the cancellation of remaining diagonal and off-diagonal terms remains valid.
        
        \item[(ix)] One can also view our choice of ``neighborhood'' of $f\otimes\chi$ in terms of local representation theory. Let $\mathcal{H}_{it}$ be the set of normalized Hecke--Maass newforms of any level, trivial central character, and spectral parameter $t$. Given $f\in\mathcal{H}_{it}$, denote by $\pi_f\simeq\pi_{f,\infty}\otimes\bigotimes_p'\pi_{f,p}$ the associated automorphic representation. For non-quadratic $\chi$, the set of principal series representations
        \begin{align*}
            \mathcal{F}_{\chi,p} \coloneqq \{\pi(\chi\abs{\cdot}^{it/\log p},\chi^{-1}\abs{\cdot}^{-it/\log p})\,:\,t\in[0,2\pi)\}
        \end{align*}
        forms a connected component of the unitary dual of $PGL_2(\Q_p)$ with respect to the Fell topology. Fixing $\chi$ non-quadratic with $c(\chi)=p^a$ and $t$, we then have
        \begin{align}
            \{f\otimes\chi\,:\,f\in\mathcal{H}_{it}(m,\overline{\chi}^2),\,m\mid p^a\} = \{f\in\mathcal{H}_{it}\,:\,\pi_{f,p}\in\mathcal{F}_{\chi,p},\,\pi_{f,\tilde{p}}\text{ unramified }\forall\,\tilde{p}\neq p\}. \label{eq:identity_local}
        \end{align}
        Defining the neighborhoods $\chi[b] \coloneqq \{\chi_1\in\Q_p^{\times\wedge}\,:\,c(\chi\chi_1^{-1})\leq b\}$, $1\leq b<a$, of $\chi$, one has
        \begin{align*}
            &\bigcup_{\phi\,(\mo p^b)}\{f\otimes\chi\phi\,:\,f\in\mathcal{H}_{it}(m,\overline{\chi\phi}^2),\,m\mid p^a\} \\
            &= \{f\in\mathcal{H}_{it}\,:\,\pi_{f,p}\in\mathcal{F}_{\chi_1,p},\,\pi_{f,\tilde{p}}\text{ unramified }\forall\,\tilde{p}\neq p,\,\chi_1\in\chi[b]\}.
        \end{align*}
        Given this rephrasing, it would be interesting to try and compute an asymptotic formula for the cubic moment with $\mathcal{F}_{\chi,p}$ on the right-hand side of \eqref{eq:identity_local} replaced by other parts of the unitary dual. We hope to address this question in future work with $\mathcal{F}_{\chi,p}$ consisting of supercuspidal representations.

        \item[(x)] An asymptotic formula for $\E_{\phi\,(\mo p^b)}[\mathcal{M}_{nh}(\chi\phi,\pmb{\alpha})]$ with an explicit main term (rather than merely an upper bound) opens the door to several potential applications. On the one hand, establishing an asymptotic allows one to deploy Soundararajan's resonance method \cite{S} (and its refinements) to show existence of large central $L$-values within our family. On the other hand, combining our proof of Theorem \ref{thm:main} with mollification techniques could yield simultaneous non-vanishing results for central $L$-values in the spirit of Michel and VanderKam \cite{MVdk}. For example, it would be interesting to adapt our proof to compute an asymptotic formula for the (simplified) mixed cubic moment
        \begin{align*}
            \sum_{\phi}\sum_{t,m}\sum_{f\in\mathcal{H}_{it}(m,\overline{\chi\phi}^2)}L(1/2,f\otimes\chi\phi)L(1/2,f\otimes\chi\phi\otimes\eta_1)L(1/2,f\otimes\chi\phi\otimes\eta_2) + \text{(Eis.)}
        \end{align*}
        for some Dirichlet characters $\eta_1,\eta_2$, possibly with $\eta_1=\eta$, $\eta_2=\eta^2$ for $\eta$ a cubic character. In a similar vein, given a self-dual $GL_2$ newform $g$, our framework could be modified to approach the evaluation of moments of the form 
        \begin{align*}
            \sum_{\phi}\sum_{t,m}\sum_{f\in\mathcal{H}_{it}(m,\overline{\chi\phi}^2)}L(1/2,f\otimes\chi\phi)L(1/2,f\otimes\chi\phi\otimes g) + \text{(Eis.)}.
        \end{align*}
    \end{enumerate}
\end{remark}

\subsection{Overview of the proof}
We prove Theorem \ref{thm:main} throughout Sections \ref{sec:set_up}--\ref{sec:completing_the_proof}, and in Section \ref{sec:conjecture} we match our asymptotic formula against the conjecture from \cite{CFKRS}. A precise outline of the proof of Theorem \ref{thm:main} is as follows:
\begin{enumerate}
    \item[Section \ref{sec:set_up}:] We recall the definitions of the $L$-functions under consideration, and apply the approximate functional equation (resulting in a triple sum $\sum_{n_1,n_2,n_3\geq1}(\dotsc)$) and Kuznetsov trace formula (resulting in a diagonal and an off-diagonal sum $\sum_{\substack{c\geq1\,:\, c\equiv0\,(\mo p^a)}}(\dotsc)$) to $\mathcal{M}_{nh}(\chi,\pmb{\alpha})$.
    \item[Section \ref{sec:diag}:] We evaluate the diagonal term coming out of the trace formula, and express it as a sum of contour integration residues. We decompose this sum into the conjectured main term and other terms that will cancel out among each other or with certain off-diagonal terms.
    \item[Section \ref{sec:averaging_kloosterman_sums}:] We evaluate the average over $\phi\,(\mo p^b)$ of the exponential sums appearing in the off-diagonal terms of the trace formula, following the strategy of Hu \cite{H}. As a result, all terms with $a<\nu_p(c)<a+b$ in the off-diagonal contribution of the trace formula vanish. The surviving terms with $\nu_p(c)=a$ are crucial in our work, and were not previously analyzed in \cite{H}. These terms lead to the off-diagonal terms that cancel out with diagonal terms.
    \item[Section \ref{sec:poisson}:] We apply Poisson summation to each $n_j$ in the off-diagonal, resulting in a triple dual sum $\sum_{m_1,m_2,m_3\in\Z}(\dotsc)$.
    \item[Section \ref{sec:evaluation_of_gauss_sums}:] We evaluate the Gauss-type sums resulting from Poisson summation. 
    \item[Section \ref{sec:evaluation_of_generating_functions}:] We evaluate certain generating functions of the Gauss-type sums, which will be applied later in large-sieve estimates in Section \ref{sec:completing_the_proof}.
    \item[Section \ref{sec:cancellation}:] We obtain cancellation among the diagonal terms that were separated off in Section \ref{sec:diag} and the off-diagonal terms with $m_1m_2m_3=0$. This section is the heart of the proof.
    \item[Section \ref{sec:completing_the_proof}:] We complete the proof of Theorem \ref{thm:main} by bounding the remaining off-diagonal terms using stationary-phase results from \cite{PeY} and large-sieve estimates.
\end{enumerate}

\subsection{Notation}
Let $c,m,n\in\Z_{\geq1}$, and $p$ a prime. For $x\in\R$, let $e(x)\coloneqq\exp(2\pi i x)$ and $e_n(x)\coloneqq e(x/n)$. Write $U_p(n)\coloneqq1+p^n\Z\subset\Z$, and let $\nu_p:\Q\rightarrow\Z\cup\{\infty\}$ denote $p$-adic valuation. Given a Dirichlet character $\psi$ modulo $c$, denote its conductor by $c(\psi)$, and the primitive Dirichlet character inducing it by $\psi'$. Let $\tau(\psi)\coloneqq\sum_{x\,(\mo c)}\psi(x)e_{c}(x)$ denote the Gauss sum attached to $\psi$. Writing
$$\,\sideset{}{^*}\sum_{y\,(\mo c)}(\dotsc)$$
for summation over $y\in(\Z/c\Z)^\times$, we define the Sali\'e sum
\begin{align*}
    S_\psi(m,n;c) \coloneqq \,\sideset{}{^*}\sum_{y\,(\mo c)}\overline{\psi}(y)e_c(my+n\overline{y}),
\end{align*}
so that $S(m,n;c)\coloneqq S_{\chi_0}(m,n;c)$ is just the classical Kloosterman sum. Define also the Ramanujan sum $R_{c}(n)\coloneqq S(n,0;c)$. 

Consider shifts $\pmb{\alpha}=(\alpha_1,\alpha_2,\alpha_3)\in\C^3$ and permutations $\pmb{\sigma}=(\sigma_1,\sigma_2,\sigma_3)\in\{\pm1\}^3$. We will often make use of the ratios of $\Gamma$-functions
\begin{align*}
    \Gamma_{\rat}(\delta,t,\pmb{\alpha},\pmb{\sigma}) &\coloneqq \prod_{1\leq j\leq 3}\gamma_{\rat}(\delta,t,\sigma_j\alpha_j,\alpha_j), \\
    \gamma_{\rat}(\delta,t,u,v) &\coloneqq \frac{\Gamma_\R(1/2+\delta+u+it)\Gamma_\R(1/2+\delta+u-it)}{\Gamma_\R(1/2+\delta+v+it)\Gamma_\R(1/2+\delta+v-it)}.
\end{align*}
Given $r\in\Z_{\geq1}$, write $\zeta^{(r)}(s)\coloneqq\zeta(s)\prod_{p\mid r}(1-p^{-s})$ for the Riemann-zeta function defined via the usual Euler product, but with the local factors at primes dividing $r$ omitted.

We call a map $w:\R\rightarrow\R_{\geq0}$ a \textit{dyadic partition function} of $\Z_{\geq1}$ if $w$ is smooth, compactly supported on $[1/2,3]$, $w\vert_{[1,2]}\equiv1$, and 
\begin{align*}
    \sum_{k\geq0}w(n/2^k)=1\quad\text{ for all }n\in\Z_{\geq1}.
\end{align*}
We will often write $w_{2^k}(\cdot)\coloneqq w(\cdot/2^k)$.

For the sake of consistency, much of the remaining notation in this paper is similar to that in \cite{Pe} and \cite{PeY}.

\section*{Acknowledgements}
The author would like to thank her PhD advisor Ian Petrow for his valuable guidance on this project, and for his careful proofreading of the paper. The author also expresses gratitude towards Alberto Acosta Reche for helpful comments.
\section{Set-up}\label{sec:set_up}
We define the $L$-functions under consideration, and state the approximate functional equation and Kuznetsov trace formula. For a more detailed background, see \cite[Section 2]{PeY}.

\subsection{Automorphic forms and $L$-functions}\label{sec:automorphic_forms}Let $q\in\Z_{\geq1}$, $m\mid q$, and consider a primitive Dirichlet character $\chi$ modulo $q$. Let $u\in\mathcal{H}_{it}(m,\overline{\chi}^2)\cup\mathcal{H}_{it,\text{Eis}}(m,\overline{\chi}^2)$. Then $u\otimes\chi$ is a newform of level $q^2$ and trivial central character. For $\re s>1$, we define the $L$-series
\begin{align*}
    L(s,u\otimes\chi) \coloneqq \sum_{n\geq1}\frac{\lambda_u(n)\chi(n)}{n^s},
\end{align*}
as well as the local $L$-function at the infinite place
\begin{align}
    L_\infty(s,u\otimes\chi) &\coloneqq \Gamma_\R(\delta_{u\otimes\chi}+s+it)\Gamma_\R(\delta_{u\otimes\chi}+s-it), \label{eq:Linfty} \\
    \delta_{u\otimes\chi} &\coloneqq \frac{1-\lambda_u(-1)\chi(-1)}{2},
\end{align}
and the completed $L$-function $\Lambda(s,u\otimes\chi)\coloneqq q^sL_\infty(s,u\otimes\chi)L(s,u\otimes\chi)$. The functional equation $\Lambda(s,u\otimes\chi)=\epsilon(u\otimes\chi)\Lambda(1-s,u\otimes\chi)$ then holds, with root number $\epsilon(u\otimes\chi)=\lambda_u(-1)$.

\subsection{Approximate functional equation}
Keep the same notation as in Section \ref{sec:automorphic_forms}, and specify to $q=p^a$. Let $u\in\mathcal{H}_{it}(m,\overline{\chi}^2)\cup\mathcal{H}_{it,\text{Eis}}(m,\overline{\chi}^2)$. For $y\in\R_{>0}$, define the weight function (cf. \cite[(5.13)]{IK})
\begin{align*}
    &V_{1/2+\alpha_j}(y,t,\delta_{u\otimes\chi}) \coloneqq \frac{1}{2\pi i}\int_{(\sigma)}y^{-s}\frac{L_\infty(1/2+s,u\otimes\chi)}{L_\infty(1/2+\alpha_j,u\otimes\chi)}\frac{G(s-\alpha_j)}{s-\alpha_j}\,ds,\quad G(s)\coloneqq e^{2s^2},
\end{align*}
where $\sigma\in\R$ lies to the right of all poles of the integrand. The function $V_{1/2+\alpha_j}(y,t,\cdot)$ decays rapidly for $y\gg1+\abs{t}$ (see Lemma \ref{lemma:Vyt} below), so it makes sense to speak of a ``truncated'' Dirichlet series
\begin{align*}
    L_0(\alpha_j,u\otimes\chi) \coloneqq \sum_{n\geq1}\frac{\lambda_u(n)\chi(n)}{\sqrt{n}}V_{1/2+\alpha_j}\left(\frac{n}{p^a},t,\delta_{u\otimes\chi}\right).
\end{align*}

\begin{lemma}[Approximate functional equation]
For $m\mid p^a$ and $u\in\mathcal{H}_{it}(m,\overline{\chi}^2)\cup\mathcal{H}_{it,\text{Eis}}(m,\overline{\chi}^2)$, one has
\begin{align}
    &L(1/2+\alpha_j,u\otimes\chi) \\
    &= p^{-a\alpha_j}\bigg(L_0(\alpha_j,u\otimes\chi) + \lambda_u(-1)\frac{L_\infty(1/2-\alpha_j,u\otimes\chi)}{L_\infty(1/2+\alpha_j,u\otimes\chi)}L_0(-\alpha_j,u\otimes\chi)\bigg). \label{eq:AFE_IK}
\end{align}
\end{lemma}

\begin{proof}
    See \cite[Theorem 5.3]{IK}; the result follows after plugging in $f=u\otimes\chi$, $X=1$, $s=1/2+\alpha_j$, $q(f)=p^{2a}$, $\gamma(f,s)=L_\infty(s,u\otimes\chi)$, $\varepsilon(f)=\lambda_u(-1)$ into loc.\ cit.
\end{proof}

Now write $\alpha\coloneqq\alpha_1+\alpha_2+\alpha_3$. Also, for $\pmb{\sigma}\coloneqq(\sigma_1,\sigma_2,\sigma_3)\in\{\pm1\}^3$, write $\sigma\coloneqq\sigma_1\sigma_2\sigma_3$. Applying first \eqref{eq:AFE_IK} to $L(1/2+\alpha_j,u\otimes\chi)$ for each $1\leq j\leq 3$, and then the Hecke multiplicativity relation
\begin{align*}
    \lambda_u(n_2)\lambda_u(n_3) = \sum_{d\mid(n_2,n_3)}\lambda_u(n_2n_3/d^2)\overline{\chi}^2(d),
\end{align*}
we obtain
\begin{align}
    \prod_{1\leq j\leq 3}L(1/2+\alpha_j,u\otimes\chi) &= p^{-a\alpha}\sum_{\pmb{\sigma}\in\{\pm1\}^3}\sum_{\substack{d\geq1 \\ (d,p)=1}}\frac{1}{d}\sum_{n_1,n_2,n_3\geq1}\frac{\lambda_u(\sigma n_1)\overline{\lambda_u}(n_2n_3)\chi(n_1)\overline{\chi}(n_2n_3)}{\sqrt{n_1n_2n_3}} \\
    &\quad\quad\quad\quad\quad\quad\quad\quad \cdot V_{1/2+\sigma_1\alpha_1}\bigg(\frac{n_1}{p^a},\cdot\bigg)V_{1/2+\sigma_2\alpha_2}\bigg(\frac{dn_2}{p^a},\cdot\bigg) V_{1/2+\sigma_3\alpha_3}\bigg(\frac{dn_3}{p^a},\cdot\bigg) \\
    &\quad\quad\quad\quad\quad\quad\quad\quad\quad\quad\quad\quad\quad\quad\quad\quad\quad\quad\,\,\,\, \cdot \prod_{1\leq j\leq 3}\frac{L_\infty(1/2+\sigma_j\alpha_j,u\otimes\chi)}{L_\infty(1/2+\alpha_j,u\otimes\chi)}. \label{eq:cube_after_afe}
\end{align}
The conjugates in the last expression appear merely for convenience, as $\lambda_u(n)\chi(n)\in\R$ for all $n\in\Z_{\geq1}$.

We will later use the following growth properties of $V_{1/2+\alpha_j}(\cdot)$ and its derivatives.
\begin{lemma}\label{lemma:Vyt}
    \leavevmode
    \vspace{-13pt}
    \begin{enumerate}
        \item[(i)] For any given $y\in\R_{>0}$, $V_{1/2+\alpha_j}(y,t,\delta)$ is entire and even as a function of $t$. 
        
        \item[(ii)] Assume $t=(p^a)^{o(1)}$, $\re\alpha_j\ll \frac{1}{\log p^a}$ and $\im\alpha_j=O(1)$. Then, if $t\in\R$, we have 
        \begin{align}
            y^k(1/2+it)^l\frac{\partial^{k+l}}{\partial y^k\partial t^l}V_{1/2+\alpha_j}(y,t,\delta)\ll_{A,k,l}y^{-\re\alpha_j}\bigg(1+\frac{y}{1+\abs{t}}\bigg)^{-A}
            \label{eq:Vyt_t_real}
        \end{align}
    for any $A>0$. Moreover, for $t=-i/2+v$ with $v\in\R$, we have
    \begin{align}
        y^k\frac{\partial^k}{\partial y^k}V_{1/2+\alpha_j}(y,-i/2+v,\delta)\ll_{A,k}(y^{-\re\alpha_j}+1)\left(1+\frac{y}{1+\abs{v}}\right)^{-A} \label{eq:Vyt_t_shifted}
    \end{align}
    for any $A>0$.
    \end{enumerate}
\end{lemma}

\begin{remark}\label{rmk:afe_bounds}
    \leavevmode
    \vspace{-13pt}
    \begin{enumerate}
        \item[(i)] We omit proving Lemma \ref{lemma:Vyt}, as only a slight modification of the proof of \cite[Lemma 10.1]{PeY} is required to account for the shifts $\alpha_j$.
        \item[(ii)] The bounds \eqref{eq:Vyt_t_real} and \eqref{eq:Vyt_t_shifted} differ from those in \cite[Lemma 10.1]{PeY} due to the possible presence of a singularity of $V_{1/2+\alpha_j}(y,\cdot)$ at $y=0$ when $\alpha_j\neq0$. In practice we will have $y\geq\frac{1}{2p^a}$, and $p^{a\re\alpha_j}\ll1$ by assumption.
    \end{enumerate}
\end{remark}

\subsection{Kuznetsov trace formula}
Consider a test function $h(t,\cdot)$ holomorphic in $t$ in the strip $\abs{\im(t)}\leq1/2+\beta$, satisfying $h(t,\cdot)=h(-t,\cdot)$, and $\abs{h(t,\cdot)}\ll(1+\abs{t})^{-2-\beta}$ for some $\beta\in\R_{>0}$. We recall the bound
\begin{align}
    \abs{S_\psi(m,n;c)} \leq \tau(c)(m,n,c)^{1/2}c^{1/2}c(\psi)^{1/2} \label{eq:weil_bound}
\end{align} 
for Sali\'e sums from \cite[Chapter 9]{KL}. To state the relevant trace formula, we also consider the Bessel transforms
\begin{align*}
    &g_0(\cdot) \coloneqq \frac{1}{4\pi}\int_{\R}\tanh(\pi t)t h(t,\cdot)\,dt, \\
    & g^+(x,\cdot) \coloneqq \frac{i}{2}\int_{\R}\frac{J_{2it}(x)}{\cosh(\pi t)}t h(t,\cdot)\,dt, \\
    &g^-(x,\cdot) \coloneqq \frac{1}{\pi}\int_{\R}K_{2it}(x)\sinh(\pi t)t h(t,\cdot)\,dt.
\end{align*}

\begin{remark}
    In \cite[(2.14)]{PeY}, the integrals defining the latter Bessel transforms contain some normalization typos. The correct normalizations can be found in \cite[(1.1), (1.2), (1.23)]{HPY}.
\end{remark}

\begin{prop}[{\cite[Proposition 2.1]{PeY}}]\label{prop:kuznetsov}
    Suppose $\chi$ is primitive of conductor $p^a$, and not quadratic. 
    There exist positive weights $w_{f,\ell}\gg p^{-a}(p^a(1+\abs{t}))^{-\varepsilon}$ and $w_{E,\ell}\gg p^{-a}(p^a(1+\abs{t}))^{-\varepsilon}$ so that, for any $n_1,n_2\in\Z_{\geq1}$ with $(n_1n_2,p)=1$, one has
    \begin{align}
        \sum_{t}h(t,\cdot)\sum_{\ell m=p^a}\sum_{f\in\mathcal{H}_{it}(m,\overline{\chi}^2)}w_{f,\ell}\lambda_f(n_1)\overline{\lambda_f}(n_2)+\frac{1}{4\pi}\int_{\R}h(t,\cdot)\sum_{\ell m=p^a}\sum_{E\in \mathcal{H}_{it,\text{Eis}}(m,\overline{\chi}^2)}w_{E,\ell}\lambda_E(n_1)\overline{\lambda_E}(n_2)\,dt \\
        = \mathbbm{1}_{\{n_1=n_2\}}g_0(\cdot)+\sum_{\substack{c\geq1 \\ c\equiv0\,(\mo p^a)}}\frac{S_{\overline{\chi}^2}(n_1,n_2;c)}{c}g^+\bigg(\frac{4\pi\sqrt{n_1n_2}}{c},\cdot\bigg). \label{eq:kuznetsov_rhs}
    \end{align}
\end{prop}

The opposite-sign case of Proposition \ref{prop:kuznetsov}, i.e. when $n_1n_2<0$, states the same as \eqref{eq:kuznetsov_rhs} except that no diagonal term appears and $g^+(\frac{4\pi\sqrt{n_1n_2}}{c},\cdot)$ is replaced by $g^-(\frac{4\pi\sqrt{-n_1n_2}}{c},\cdot)$.

\subsection{Identifying diagonal and off-diagonal terms of $\mathcal{M}_{nh}(\cdot)$}
We first split
\begin{align*}
    \mathcal{M}_{nh}(\chi,\pmb{\alpha}) = \sum_{\pm}\mathcal{M}_{nh}^\pm(\chi,\pmb{\alpha})
\end{align*}
into the contribution from even and odd forms. After filling in \eqref{eq:cube_after_afe} into $\mathcal{M}^\pm_{nh}(\chi,\pmb{\alpha})$, we are ready to apply the Kuznetsov trace formula (for the indices $\sigma n_1$ and $n_2n_3$) with test function
\begin{align*}
    &h(t,n_1,n_2,n_3,d,\delta_\chi^\pm,\pmb{\alpha},\pmb{\sigma}) \\
    &\coloneqq h_0(t)V_{1/2+\sigma_1\alpha_1}\left(\frac{n_1}{p^a},t,\delta_\chi^\pm\right)V_{1/2+\sigma_2\alpha_2}\left(\frac{dn_2}{p^a},t,\delta_\chi^\pm\right)V_{1/2+\sigma_3\alpha_3}\left(\frac{dn_3}{p^a},t,\delta_\chi^\pm\right)\Gamma_{\rat}(\delta_\chi^\pm,t,\pmb{\alpha},\pmb{\sigma})
\end{align*}
to forms of parity $\delta_\chi^\pm\coloneqq\frac{1\mp\chi(-1)}{2}$.
We then have that
\begin{align}
    &\mathcal{M}_{nh}^\pm(\chi,\pmb{\alpha}) \\
    &=\sum_th_0(t)\sum_{\ell m=p^a}\sum_{f}\frac{1\pm\lambda_f(-1)}{2}\prod_{1\leq j\leq 3}L(1/2+\alpha_j,f\otimes\chi) \\
    &\quad\, + \frac{1}{4\pi}\int_\R h_0(t)\sum_{\ell m=p^a}\sum_{E}\frac{1\pm\lambda_E(-1)}{2}\prod_{1\leq j\leq 3}L(1/2+\alpha_j,E\otimes\chi)\, dt \\
    &= \frac{1}{4p^{a\alpha}}\sum_{\pmb{\sigma}\in\{\pm1\}^3}\bigg((1+\sigma\pm(1-\sigma))\mathcal{D}(\pm,\chi,\pmb{\alpha},\pmb{\sigma})+\sum_{\,\,\,\pm_\lambda}(1\pm_\lambda\sigma\pm(1\mp_\lambda\sigma))\mathcal{S}(\pm,\pm_\lambda,\chi,\pmb{\alpha},\pmb{\sigma})\bigg) \\
    &= \frac{1}{2p^{a\alpha}}\sum_{\pmb{\sigma}\in\{\pm1\}^3}(\pm1)^{\frac{1-\sigma}{2}}\bigg(\mathcal{D}(\pm,\chi,\pmb{\alpha},\pmb{\sigma})+\sum_{\,\,\,\pm_\lambda}(\pm_\lambda1)^{\frac{1\mp1}{2}}\mathcal{S}(\pm,\pm_\lambda,\chi,\pmb{\alpha},\pmb{\sigma})\bigg), \label{eq:avg_setup}
\end{align}
where
\begin{align*}
    \mathcal{D}(\pm,\chi,\pmb{\alpha},\pmb{\sigma}) &\coloneqq \sum_{\substack{d\geq1 \\ (d,p)=1}}\frac{1}{d}\sum_{\substack{n_1,n_2,n_3\geq1 \\ n_1=n_2n_3}}\frac{\chi(n_1)\overline{\chi}(n_2n_3)}{\sqrt{n_1n_2n_3}}g_0(n_2,n_3,d,\delta_\chi^\pm,\pmb{\alpha},\pmb{\sigma}) \\
    &\,= \sum_{\substack{d,n_2,n_3\geq1 \\ (dn_2n_3,p)=1}}\frac{g_0(n_2,n_3,d,\delta_\chi^\pm,\pmb{\alpha},\pmb{\sigma})}{dn_2n_3}, \\
    \mathcal{S}(\pm,\pm_\lambda,\chi,\pmb{\alpha},\pmb{\sigma}) &\coloneqq \sum_{\substack{d\geq1 \\ (d,p)=1}}\frac{1}{d}\sum_{n_1,n_2,n_3\geq1}\frac{\chi(n_1)\overline{\chi}(n_2n_3)}{\sqrt{n_1n_2n_3}}\sum_{\substack{c\geq1 \\ c\equiv0\,(\mo p^a)}}\frac{S_{\overline{\chi}^2}(\pm_\lambda n_1,n_2n_3;c)}{c} \\
    &\quad\quad\quad\quad\quad\quad\quad\quad\quad\quad\quad\quad\quad\quad\quad\quad\,\, \cdot g^{\pm_\lambda}\left(\frac{4\pi\sqrt{n_1n_2n_3}}{c},n_1,n_2,n_3,d,\delta_\chi^\pm,\pmb{\alpha},\pmb{\sigma}\right)
\end{align*}
denote the diagonal and off-diagonal contributions, respectively. We have obtained the following.
\begin{lemma}\label{lemma:summarizing_kuznetsov}
    One has
    \begin{align*}
        &\E_{\phi\,(\mo p^b)}[\mathcal{M}_{nh}(\chi\phi,\pmb{\alpha})] \\
        &= \frac{1}{2p^{a\alpha}}\sum_{\pm}\sum_{\pmb{\sigma}\in\{\pm1\}^3}(\pm1)^{\frac{1-\sigma}{2}}\bigg(\frac{1}{\varphi(p^b)}\sum_{\phi\,(\mo p^b)}\mathcal{D}(\pm,\chi\phi,\pmb{\alpha},\pmb{\sigma}) \\
        &\quad\quad\quad\quad\quad\quad\quad\quad\quad\quad\quad\quad\,\,\, +\sum_{\,\,\,\pm_\lambda}(\pm_\lambda1)^{\frac{1\mp1}{2}}\sum_{\substack{d\geq1 \\ (d,p)=1}}\sum_{n_1,n_2,n_3\geq1}\sum_{\substack{c\geq1 \\ c\equiv0\,(\mo p^a)}}\frac{X(\cdot)}{cd\sqrt{n_1n_2n_3}}\bigg),
    \end{align*}
    where 
    \begin{align}
        &X(\pm,\pm_\lambda,\chi,n_1,n_2,n_3,c,d,\pmb{\alpha},\pmb{\sigma}) \\
        &\coloneqq \frac{1}{\varphi(p^b)}\sum_{\phi\,(\mo p^b)}(\chi\phi)(n_1)(\overline{\chi\phi})(n_2n_3)S_{\overline{(\chi\phi)}^2}(\pm_\lambda n_1,n_2n_3;c)g^{\pm_\lambda}(\cdot,\delta_{\chi\phi}^\pm,\cdot).  \label{eq:X_def}
    \end{align}
\end{lemma}
We evaluate the diagonal contribution in Section \ref{sec:diag}, and evaluate $X(\cdot)$ in Section \ref{sec:averaging_kloosterman_sums}.

\begin{remark}
    We do not average the Kuznetsov trace formula over $\phi\,(\mo p^b)$, rather we evaluate this average after multiplying by an extra factor $(\chi\phi)(n_1)(\overline{\chi\phi})(n_2n_3)$ coming out of the approximate functional equation.
\end{remark}

We end this section by stating upper bounds for the Bessel transforms appearing in the off-diagonal of the trace formula.
\begin{lemma}\label{lemma:bounds_for_bessel_transforms}
    Suppose $T=(p^a)^{o(1)}$, and $\re\alpha_j\ll\frac{1}{\log p^a}$ and $\im\alpha_j=O(1)$ for all $1\leq j\leq 3$. Then, for $k\in\Z_{\geq0}$, one has 
    \begin{align*}
        \frac{\partial^k}{\partial x^k}g^\pm(x,\cdot) \ll_{\varepsilon,k,A}x(x^{-k}+x^k)T^{k+1}\bigg(1+\frac{\max\{n_1,dn_2,dn_3\}}{p^aT}\bigg)^{-A}\cdot\begin{dcases}
            1 &\text{ in the $(+)$-case} \\
            x^{-\varepsilon}T^\varepsilon &\text{ in the $(-)$-case}
        \end{dcases}.
    \end{align*}
\end{lemma}

\begin{proof}
    The result follows from combining the proofs of \cite[Lemma 10.2]{PeY} and \cite[Lemma 10.4]{PeY} with Lemma \ref{lemma:Vyt}.
\end{proof}
\section{Evaluating the diagonal contribution}
\label{sec:diag}
We work similarly to the holomorphic setting in \cite[Section 3]{Pe}. Define the function
\begin{align*}
    &\,\,L(s_1,s_2,s_3,\delta,t,\pmb{\alpha},\pmb{\sigma}) \\
    &\coloneqq p^{a(s_1+s_2+s_3)}(s_1+s_2)(s_1+s_3)(s_2+s_3)\zeta^{(p)}(1+s_1+s_2)\zeta^{(p)}(1+s_1+s_3)\zeta^{(p)}(1+s_2+s_3) \\
    &\quad\,\, \cdot \prod_{1\leq j\leq 3}\gamma_{\rat}(\delta,t,s_j,\alpha_j)G(s_j-\sigma_j\alpha_j),
\end{align*}
which is holomorphic on $\mathcal{A}\coloneqq\{(s_1,s_2,s_3)\in\C^3\,:\,\re s_j>-1/2-\delta\text{ for }1\leq j\leq 3\}$ as the poles of $\zeta^{(p)}(\cdot)$ are cancelled out by the factor $(s_1+s_2)(s_1+s_3)(s_2+s_3)$. We consider $t\in\R$, and recall that $t=(p^a)^{o(1)}$ as $a\rightarrow\infty$ or $p\rightarrow\infty$. Note that $L(\cdot)$ is invariant under simultaneous permutations of $(s_1,s_2,s_3)$, $(\alpha_1,\alpha_2,\alpha_3)$ and $(\sigma_1,\sigma_2,\sigma_3)$, but it is not invariant under simultaneous permutations of only $(\alpha_1,\alpha_2,\alpha_3)$ and $(\sigma_1,\sigma_2,\sigma_3)$.

\subsection{Combinatorial expression for $\mathcal{D}(\pm,\chi,\pmb{\alpha},\pmb{\sigma})$}
Plugging in the definitions, we obtain
\begin{align*}
    \mathcal{D}(\pm,\chi,\pmb{\alpha},\pmb{\sigma}) &= \frac{1}{4\pi}\sum_{\substack{d,n_2,n_3\geq1 \\ (dn_2n_3,p)=1}}\frac{1}{dn_2n_3}\int_{\R}\tanh(\pi t)t h(t,n_2n_3,n_2,n_3,d,\delta_\chi^\pm,\pmb{\alpha},\pmb{\sigma})\,dt \\
    &= \frac{1}{4\pi}\int_\R \tanh(\pi t)t h_0(t)f(\delta_\chi^\pm,t,\pmb{\alpha},\pmb{\sigma})\,dt,
\end{align*}
where
\begin{align*}
    &f(\delta,t,\pmb{\alpha},\pmb{\sigma}) \\
    &\coloneqq \bigg(\frac{1}{2\pi i}\bigg)^3\int_{(\frac{1}{6})}\int_{(\frac{1}{5})}\int_{(\frac{1}{4})}\frac{L(s_1,s_2,s_3,\delta,t,\pmb{\alpha},\pmb{\sigma})}{(s_1-\sigma_1\alpha_1)(s_2-\sigma_2\alpha_2)(s_3-\sigma_3\alpha_3)(s_1+s_2)(s_2+s_3)(s_3+s_1)}ds_1ds_2ds_3.
\end{align*}

We claim the following decomposition.
\begin{prop}\label{prop:diag_decomposition}
    If $t=(p^a)^{o(1)}$, and $\re\alpha_j\ll\frac{1}{\log p^a}$ for all $1\leq j\leq3$, then
    \begin{align}
        f(\delta,t,\pmb{\alpha},\pmb{\sigma}) &= f_{\text{MT}}(\delta,t,\pmb{\alpha},\pmb{\sigma}) \label{eq:f_final_expression} \\
        &\quad\, + N(\sigma_1\alpha_1,\sigma_2\alpha_2,\sigma_3\alpha_3,\delta,t,\pmb{\alpha},\pmb{\sigma}) \\
        &\quad\, + N(\sigma_1\alpha_1,\sigma_3\alpha_3,\sigma_2\alpha_2,\delta,t,(\alpha_1,\alpha_3,\alpha_2),(\sigma_1,\sigma_3,\sigma_2)) \\ 
        &\quad\, + N(\sigma_2\alpha_2,\sigma_3\alpha_3,\sigma_1\alpha_1,\delta,t,(\alpha_2,\alpha_3,\alpha_1),(\sigma_2,\sigma_3,\sigma_1)) \\
        &\quad\, + \frac{1}{2}\frac{L(0,0,0,\delta,t,\pmb{\alpha},\pmb{\sigma})}{(-\sigma_1\alpha_1)(-\sigma_2\alpha_2)(-\sigma_3\alpha_3)} + O_\varepsilon(p^{a(-\frac{1}{3}+\varepsilon)}),
    \end{align}
    where 
    \begin{align*}
        f_{\text{MT}}(\delta,t,\pmb{\alpha},\pmb{\sigma}) \coloneqq \frac{L(\sigma_1\alpha_1,\sigma_2\alpha_2,\sigma_3\alpha_3,\delta,t,\pmb{\alpha},\pmb{\sigma})}{(\sigma_1\alpha_1+\sigma_2\alpha_2)(\sigma_2\alpha_2+\sigma_3\alpha_3)(\sigma_3\alpha_3+\sigma_1\alpha_1)},
    \end{align*}
    and $N(\cdot)$ is defined in \eqref{eq:N_def} below.
\end{prop}

\begin{proof}
Within $\mathcal{A}$, the integrand of $f(\delta,t,\pmb{\alpha},\pmb{\sigma})$ has poles at $s_j=\sigma_j\alpha_j$ and at $s_j=-s_k$ for distinct $1\leq j,k\leq 3$. Shifting the $s_3$-line $(1/4)$ to $(-5/12)$, we pick up simple poles at $s_3\in\{\sigma_3\alpha_3,-s_1,-s_2\}$ with respective residues
\begin{align}
    &\bigg(\frac{1}{2\pi i}\bigg)^2\int_{(\frac{1}{6})}\int_{(\frac{1}{5})}\frac{L(s_1,s_2,\sigma_3\alpha_3,\delta,t,\pmb{\alpha},\pmb{\sigma})}{(s_1-\sigma_1\alpha_1)(s_2-\sigma_2\alpha_2)(s_1+s_2)(s_1+\sigma_3\alpha_3)(s_2+\sigma_3\alpha_3)}ds_1ds_2, \label{eq:integral_A}\\
    &\bigg(\frac{1}{2\pi i}\bigg)^2\int_{(\frac{1}{6})}\int_{(\frac{1}{5})}\frac{L(s_1,s_2,-s_1,\delta,t,\pmb{\alpha},\pmb{\sigma})}{(s_1-\sigma_1\alpha_1)(s_2-\sigma_2\alpha_2)(-s_1-\sigma_3\alpha_3)(-s_1+s_2)(s_1+s_2)}ds_1ds_2, \label{eq:integral_B}\\
    &\bigg(\frac{1}{2\pi i}\bigg)^2\int_{(\frac{1}{6})}\int_{(\frac{1}{5})}\frac{L(s_1,s_2,-s_2,\delta,t,\pmb{\alpha},\pmb{\sigma})}{(s_1-\sigma_1\alpha_1)(s_2-\sigma_2\alpha_2)(-s_2-\sigma_3\alpha_3)(s_1+s_2)(s_1-s_2)}ds_1ds_2. \label{eq:integral_C}
\end{align}
The remainder integral can be bounded as
\begin{align*}
    \bigg(\frac{1}{2\pi i}\bigg)^3\int_{(\frac{1}{24})}\int_{(\frac{1}{24})}\int_{(-\frac{5}{12})}(\dotsc)\,ds_1ds_2ds_3 \ll_\varepsilon p^{a(-\frac{1}{3}+\varepsilon)}
\end{align*}
after shifting the $s_1$- and $s_2$-lines to the left to $(1/24)$. Note that the latter bound is not optimal (as the contours can be shifted further to the left), but it is merely for convenience, as it will get absorbed by the upper bounds $O_\varepsilon(p^{-b+a\varepsilon})$ in Section \ref{sec:completing_the_proof} (since $1\leq b\leq a/3$). In each of the double integrals \eqref{eq:integral_A}--\eqref{eq:integral_C}, we further shift resp. the $s_2$-line $(1/5)$, the $s_2$-line $(1/5)$, and the $s_1$-line $(1/6)$ all to $(-5/12)$, picking up more residues at simple poles, i.e.
\begin{align*}
    \eqref{eq:integral_A} &= \frac{1}{2\pi i}\int_{(\frac{1}{6})}\bigg(\frac{L(s_1,\sigma_2\alpha_2,\sigma_3\alpha_3,\delta,t,\pmb{\alpha},\pmb{\sigma})}{(s_1-\sigma_1\alpha_1)(s_1+\sigma_2\alpha_2)(s_1+\sigma_3\alpha_3)(\sigma_2\alpha_2+\sigma_3\alpha_3)} \\
    &\quad\quad\quad\quad\quad\quad +\frac{L(s_1,-s_1,\sigma_3\alpha_3,\delta,t,\pmb{\alpha},\pmb{\sigma})}{(s_1-\sigma_1\alpha_1)(-s_1-\sigma_2\alpha_2)(s_1+\sigma_3\alpha_3)(-s_1+\sigma_3\alpha_3)} \\
    &\quad\quad\quad\quad\quad\quad +\frac{L(s_1,-\sigma_3\alpha_3,\sigma_3\alpha_3,\delta,t,\pmb{\alpha},\pmb{\sigma})}{(s_1-\sigma_1\alpha_1)(-\sigma_3\alpha_3-\sigma_2\alpha_2)(s_1-\sigma_3\alpha_3)(s_1+\sigma_3\alpha_3)}\bigg)ds_1 + O_\varepsilon(p^{a(-\frac{1}{3}+\varepsilon)}), \\
    \eqref{eq:integral_B} &= \frac{1}{2\pi i}\int_{(\frac{1}{6})}\bigg(\frac{L(s_1,\sigma_2\alpha_2,-s_1,\delta,t,\pmb{\alpha},\pmb{\sigma})}{(s_1-\sigma_1\alpha_1)(-s_1-\sigma_3\alpha_3)(-s_1+\sigma_2\alpha_2)(s_1+\sigma_2\alpha_2)} \\
    &\quad\quad\quad\quad\quad\quad +\frac{L(s_1,s_1,-s_1,\delta,t,\pmb{\alpha},\pmb{\sigma})}{(s_1-\sigma_1\alpha_1)(s_1-\sigma_2\alpha_2)(-s_1-\sigma_3\alpha_3)2s_1} \\
    &\quad\quad\quad\quad\quad\quad +\frac{L(s_1,-s_1,-s_1,\delta,t,\pmb{\alpha},\pmb{\sigma})}{(s_1-\sigma_1\alpha_1)(-s_1-\sigma_2\alpha_2)(-s_1-\sigma_3\alpha_3)(-2s_1)}\bigg)ds_1 + O_\varepsilon(p^{a(-\frac{1}{3}+\varepsilon)}), \\
    \eqref{eq:integral_C} &= \frac{1}{2\pi i}\int_{(\frac{1}{5})}\bigg(\frac{L(\sigma_1\alpha_1,s_2,-s_2,\delta,t,\pmb{\alpha},\pmb{\sigma})}{(s_2-\sigma_2\alpha_2)(-s_2-\sigma_3\alpha_3)(\sigma_1\alpha_1+s_2)(\sigma_1\alpha_1-s_2)} \\
    &\quad\quad\quad\quad\quad\quad +\frac{L(-s_2,s_2,-s_2,\delta,t,\pmb{\alpha},\pmb{\sigma})}{(-s_2-\sigma_1\alpha_1)(s_2-\sigma_2\alpha_2)(-s_2-\sigma_3\alpha_3)(-2s_2)}\bigg)ds_2 + O_\varepsilon(p^{a(-\frac{1}{3}+\varepsilon)}).
\end{align*}
Note that both the integral of the third term in the expression for \eqref{eq:integral_B} and the integral of the second term in the expression for \eqref{eq:integral_C} are sufficiently small to be absorbed into the respective error terms (after shifting the contours sufficiently to the right). As in \cite[(11)]{Pe}, we observe in $f(\delta,t,\pmb{\alpha},\pmb{\sigma})$ terms of the form
\begin{align*}
    M(\kappa,\lambda,\mu,\delta,t,\pmb{\alpha},\pmb{\sigma}) \coloneqq \frac{1}{2\pi i}\int_{(\frac{1}{6})}\frac{L(s,-s,\mu,\delta,t,\pmb{\alpha},\pmb{\sigma})}{(s-\kappa)(-s-\lambda)(\mu-s)(\mu+s)}ds,
\end{align*}
where $\kappa,\lambda,\mu\in\{\sigma_1\alpha_1,\sigma_2\alpha_2,\sigma_3\alpha_3\}$. As $L(s,-s,\mu,\cdot)$ is at least constant-sized in $p^a$ in every vertical strip, shifting the contour $(1/6)$ to the left does not save us anything in terms of $p^a$. Instead, we set 
\begin{align}
    &N(\kappa,\lambda,\mu,\delta,t,\pmb{\alpha},\pmb{\sigma}) \\
    &\coloneqq M(\kappa,\lambda,\mu,\delta,t,\pmb{\alpha},\pmb{\sigma}) + \frac{L(-\lambda,\lambda,\mu,\delta,t,\pmb{\alpha},\pmb{\sigma})}{(-\lambda-\kappa)(\mu+\lambda)(\mu-\lambda)} + \frac{L(\mu,-\mu,\mu,\delta,t,\pmb{\alpha},\pmb{\sigma})}{(\mu-\kappa)(-\mu-\lambda)(2\mu)}.\label{eq:N_def}
\end{align}
Shifting the relevant contours $(1/6)$ and $(1/5)$ in \eqref{eq:integral_A}--\eqref{eq:integral_C} to $(-5/12)$, we obtain the desired result.
\end{proof}

\begin{remark} 
    Note the symmetry relation
    \begin{align}
        N(\kappa,\lambda,\mu,\delta,t,\pmb{\alpha},\pmb{\sigma})=N(\lambda,\kappa,\mu,\delta,t,(\alpha_2,\alpha_1,\alpha_3),(\sigma_2,\sigma_1,\sigma_3)), \label{eq:symmetry_in_N}
    \end{align}
    which follows from
    \begin{align*}
        M(\kappa,\lambda,\mu,\delta,t,\pmb{\alpha},\pmb{\sigma}) &= \frac{1}{2\pi i}\int_{(\frac{1}{6})}\frac{L(-s,s,\mu,\delta,t,(\alpha_2,\alpha_1,\alpha_3),(\sigma_2,\sigma_1,\sigma_3))}{(s-\kappa)(-s-\lambda)(\mu-s)(\mu+s)}ds \\
        &= \frac{1}{2\pi i}\int_{(\frac{1}{6})}\frac{L(s,-s,\mu,\delta,t,(\alpha_2,\alpha_1,\alpha_3),(\sigma_2,\sigma_1,\sigma_3))}{(-s-\kappa)(s-\lambda)(\mu+s)(\mu-s)}ds \\
        &\quad\, + \frac{L(-\kappa,\kappa,\mu,\delta,t,(\alpha_2,\alpha_1,\alpha_3),(\sigma_2,\sigma_1,\sigma_3))}{(-\kappa-\lambda)(\mu-\kappa)(\mu+\kappa)} - \frac{L(-\lambda,\lambda,\mu,\delta,t,\pmb{\alpha},\pmb{\sigma})}{(-\lambda-\kappa)(\mu+\lambda)(\mu-\lambda)} \\
        &\quad\, - \frac{L(\mu,-\mu,\mu,\delta,t,\pmb{\alpha},\pmb{\sigma})}{(\mu-\kappa)(-\mu-\lambda)(2\mu)} + \frac{L(\mu,-\mu,\mu,\delta,t,(\alpha_2,\alpha_1,\alpha_3),(\sigma_2,\sigma_1,\sigma_3))}{(-\mu-\kappa)(\mu-\lambda)(2\mu)},
    \end{align*}
    where in the second step we applied a change of variables $s\mapsto -s$, followed by a contour shift to the right. As a result of \eqref{eq:symmetry_in_N}, the asymptotic formula in Proposition \ref{prop:diag_decomposition} is symmetric in the three pairs $(\alpha_j,\sigma_j)$.
\end{remark}

Now consider the right-hand side of \eqref{eq:f_final_expression}. Regarding the first term $f_{\text{MT}}(\cdot)$, we will show the following in Section \ref{sec:conjecture}.
\begin{lemma}\label{lemma:conjectured_main_term}
    The conjectured main term of $\mathcal{M}^\pm_{nh}(\chi,\pmb{\alpha})$ is given by 
    $$\sum_{\pmb{\sigma}\in\{\pm1\}^3}(\pm1)^{\frac{1-\sigma}{2}}\text{MT}_{nh}^\pm(\chi,\pmb{\alpha},\pmb{\sigma}),$$
    where
    \begin{align}
        &\text{MT}_{nh}^\pm(\chi,\pmb{\alpha},\pmb{\sigma})\coloneqq \frac{1}{2p^{a\alpha}}\cdot\frac{1}{4\pi}\int_\R\tanh(\pi t)th_0(t)f_{\text{MT}}(\delta_\chi^\pm,t,\pmb{\alpha},\pmb{\sigma})\,dt. \label{eq:conjectured_main_term_after_extraction}
    \end{align}
\end{lemma}
As a consequence, the conjectured main term of $\E_{\phi\,(\mo p^b)}[\mathcal{M}_{nh}(\chi\phi,\pmb{\alpha})]$ is given by 
\begin{align*}
    &\frac{1}{\varphi(p^b)}\sum_{\pmb{\sigma}\in\{\pm1\}^3}\sum_{\pm}(\pm1)^{\frac{1-\sigma}{2}}\sum_{\phi\,(\mo p^b)}\text{MT}_{nh}^\pm(\chi\phi,\pmb{\alpha},\pmb{\sigma}) \\
    &= \frac{1}{4p^{a\alpha}}\sum_{\pmb{\sigma}\in\{\pm1\}^3}\sum_\pm(\pm1)^{\frac{1-\sigma}{2}}\frac{1}{4\pi}\int_\R\tanh(\pi t)th_0(t)\sum_{\,\,\,\pm_b}f_{\text{MT}}(\tfrac{1\mp\chi(-1)(\pm_b1)}{2},t,\pmb{\alpha},\pmb{\sigma})\,dt.
\end{align*}
Noting that the two $(\pm)$-terms cancel out when $\sigma=-1$, while they are equal when $\sigma=1$, we obtain the main term on the right-hand side of \eqref{eq:claim_asymptotic}.

We call the remaining four terms on the right-hand side of \eqref{eq:f_final_expression} ``false'' main terms due to the following two types of cancellation:
\begin{enumerate}
    \item[(i)] The second term in the definition of each $N(\cdot)$-term cancels out with the same term for a different $\pmb{\sigma}$, e.g.
    \begin{align*}
        \tfrac{L(-\sigma_2\alpha_2,\sigma_2\alpha_2,\sigma_3\alpha_3,\delta,t,\pmb{\alpha},\pmb{\sigma})}{(-\sigma_2\alpha_2-\sigma_1\alpha_1)(\sigma_3\alpha_3+\sigma_2\alpha_2)(\sigma_3\alpha_3-\sigma_2\alpha_2)} + \tfrac{L(-(-\sigma_2\alpha_2),-\sigma_2\alpha_2,\sigma_3\alpha_3,\delta,t,\pmb{\alpha},(-\sigma_1,-\sigma_2,\sigma_3))}{(-(-\sigma_2\alpha_2)-(-\sigma_1\alpha_1))(\sigma_3\alpha_3+(-\sigma_2\alpha_2))(\sigma_3\alpha_3-(-\sigma_2\alpha_2))} = 0.
    \end{align*}

    \item[(ii)] The remaining two terms in the definition of each $N(\cdot)$-term and the $L(0,0,0,\cdot)$-term cancel out with off-diagonal terms; see Section \ref{sec:cancellation}.
\end{enumerate}

\subsection{Leading-order terms of $\mathcal{M}_{nh}^\pm(\chi,\pmb{0})$}\label{sec:asymptotic_for_diag}
Write $\sigma\alpha\coloneqq\sigma_1\alpha_1+\sigma_2\alpha_2+\sigma_3\alpha_3$. Recalling Lemma \ref{lemma:conjectured_main_term}, we obtain
\begin{align}
    \sum_{\pmb{\sigma}\in\{\pm1\}^3}(\pm1)^{\frac{1-\sigma}{2}}\mt_{nh}^\pm(\chi,\pmb{\alpha},\pmb{\sigma}) &= \frac{1}{2p^{a\alpha}}\sum_{\pmb{\sigma}\in\{\pm1\}^3}(\pm1)^{\frac{1-\sigma}{2}}p^{a\sigma\alpha}\bigg(\prod_{1\leq j<k\leq 3}\zeta^{(p)}(1+\sigma_j\alpha_j+\sigma_k\alpha_k)\bigg) \\
    &\quad\quad\quad\quad\quad\quad\,\,\,\,\cdot \frac{1}{4\pi}\int_\R \tanh(\pi t)t h_0(t)\Gamma_{\rat}(\delta_\chi^\pm,t,\pmb{\alpha},\pmb{\sigma})\,dt.
\end{align}

In particular, the leading-order terms of $\mathcal{M}_{nh}^\pm(\chi,\pmb{0})$ are given by
\begin{align}
    &\lim_{\pmb{\alpha}\rightarrow\,\pmb{0}}\frac{1}{2p^{a\alpha}}\sum_{\pmb{\sigma}\in\{\pm1\}^3}(\dotsc) \\
    &= \lim_{\pmb{\alpha}\rightarrow\,\pmb{0}}\frac{1}{2p^{a\alpha}}\sum_{\pmb{\sigma}\in\{\pm1\}^3}(\pm1)^{\frac{1-\sigma}{2}}\bigg(\sum_{n\geq0}\frac{(\sigma\alpha\log p^a)^n}{n!}\bigg)\prod_{1\leq j<k\leq 3}\left(\frac{1-p^{-1}}{\sigma_j\alpha_j+\sigma_k\alpha_k}+O(1)\right) \\
    &\quad\quad\quad\quad\quad\quad\quad\quad\,\,\,\,\, \cdot\frac{1}{4\pi}\int_\R \tanh(\pi t)t h_0(t)\prod_{\substack{1\leq j\leq 3 \\ \sigma_j=-1}}(1+O_\varepsilon((1+\abs{t})^\varepsilon\abs{\alpha_j}))\,dt \\
    &= \begin{dcases}
        \bigg(\frac{8}{6}(1-p^{-1})^3(\log p^a)^3 + O((\log p^a)^2)\bigg)\cdot\frac{1}{4\pi}\int_{\R}\tanh(\pi t)t h_0(t)\,dt &\text{ in the $(+)$-case} \\
        0 &\text{ in the $(-)$-case}
    \end{dcases}, \label{eq:size_of_main_term}
\end{align}
where in the second step all polar terms in terms of $\sigma_1\alpha_1,\sigma_2\alpha_2,\sigma_3\alpha_3$ cancel out due to the summation over $\pmb{\sigma}$. The leading constants in \eqref{eq:size_of_main_term} confirm the conjecture \eqref{eq:conjecture_conrey_farmer}. The term $O((\log p^a)^2)$ is merely a truncation of a complicated expression in terms of $\log p^a$ and $p$ that can be explicitly computed, and it does not represent the limit of our method. Note also that the integral 
\begin{align}
    \frac{1}{4\pi}\int_{\R}\tanh(\pi t)t h_0(t)\,dt \label{eq:size_of_family}
\end{align}
is the Plancherel volume at $\infty$ of the family of Maass forms in the definition of $\mathcal{M}_{nh}(\chi,\pmb{\alpha})$.
\section{Exponential sums averaged over a family}\label{sec:averaging_kloosterman_sums}
Recall the definition of $X(\cdot)$ from \eqref{eq:X_def}. Given the modulus $c$, let $c_0\coloneqq p^{\nu_p(c)}$ and $c'\coloneqq c/c_0$. Recall that $\nu_p(c)\geq a$. Write $\overline{c'}$ for an inverse of $c'$ modulo $c_0$, and $\overline{c_0}$ for an inverse of $c_0$ modulo $c'$. In this section, we prove the following.
\begin{prop}\label{prop:sum_of_kloosterman_sums}
    If $\nu_p(c)\geq a+b$, then 
    \begin{align*}
        X(\cdot) = \frac{1}{2}\chi(n_1)\overline{\chi}(n_2n_3)S_{\overline{\chi}^2}(\pm_\lambda n_1,n_2n_3;c)\sum_{\,\,\,\pm_b}(\pm_\lambda1)^{\frac{1\mp_b1}{2}}g^{\pm_\lambda}(\cdot,\tfrac{1\mp\chi(-1)(\pm_b1)}{2},\cdot).
    \end{align*}
    If $\nu_p(c)=a$, then 
    \begin{align*}
        X(\cdot) &=  \frac{1}{2(1-p^{-1})}\chi(\mp_\lambda1)S(\pm_\lambda n_1\overline{p^a},n_2n_3\overline{p^a};c')\sum_{\,\,\,\pm_b}(\mp_\lambda1)^{\frac{1\mp_b1}{2}}g^{\pm_\lambda}(\cdot,\tfrac{1\mp\chi(-1)(\pm_b1)}{2},\cdot) \\
        &\quad\, \cdot\sum_{\psi\,(\mo p^{a-b})}\frac{\overline{\psi}(\mp_\lambda n_1n_2n_3\overline{c'}^2)\tau(\chi\psi)}{\tau(\chi\overline{\psi})}.
    \end{align*}
    If $a<\nu_p(c)<a+b$, then $X(\cdot)=0$.
\end{prop}

\begin{remark}
    We emphasize the relevance of Proposition \ref{prop:sum_of_kloosterman_sums}: after summing over $n_1,n_2,n_3,c,d$ as in Lemma \ref{lemma:summarizing_kuznetsov}, some of the $(p^a\pdiv c)$-terms give rise to terms that cancel with false main terms, while the remaining $(p^a\pdiv c)$-terms and the $(p^{a+b}\mid c)$-terms lead to the dual moment in the Motohashi formula. This will be made more explicit in Sections \ref{sec:poisson}--\ref{sec:cancellation}.
\end{remark}

In our proof of Proposition \ref{prop:sum_of_kloosterman_sums}, we will apply the following lemma, which allows us to ``linearize'' the multiplicative character $\chi$.
\begin{lemma}[{\cite[Lemma 2.1]{PeYfourth}}]\label{lemma:postnikov}
    Let $p$ be an odd prime, and $\beta\in\Z_{\geq2}$. There exists a unique group homomorphism 
    $$\ell:\widehat{(\Z/p^{\beta}\Z)^\times}\rightarrow\Z/p^{\beta-1}\Z:\chi\mapsto\ell_\chi$$
    such that the Postnikov formula
    $$\chi(1+pt) = e_{p^\beta}(\ell_\chi\log_p(1+pt))$$
    holds for all $\chi\in\widehat{(\Z/p^{\beta}\Z)^\times}$, $t\in\Z$. Here $\log_p:1+p\Z_p\rightarrow p\Z_p$ denotes the $p$-adic logarithm defined by the convergent power series expansion
    \begin{align*}
        \log_p(1+x) = \sum_{n\geq1}(-1)^{n+1}\frac{x^n}{n}.
    \end{align*} 
    The map $\ell$ is surjective, and for $1\leq\alpha\leq\beta$ we have $\ell_{\chi_1}\equiv\ell_{\chi_2}\,(\mo p^{\beta-\alpha})$ if and only if $\chi_1\overline{\chi_2}$ is a character modulo $p^\alpha$.
\end{lemma}

\begin{proof}[Proof of Proposition \ref{prop:sum_of_kloosterman_sums}]
First note that
\begin{align}
    X(\cdot) = \sum_{\,\,\,\pm_b}g^{\pm_\lambda}(\cdot,\tfrac{1\mp\chi(-1)(\pm_b1)}{2},\cdot)\frac{1}{\varphi(p^b)}\sum_{\phi\,(\mo p^b)}\frac{1\pm_b\phi(-1)}{2}(\chi\phi)(n_1)(\overline{\chi\phi})(n_2n_3)S_{\overline{(\chi\phi)}^2}(\cdot). \label{eq:X_after_parity}
\end{align}

We can then rewrite
\begin{align}
    S_{\overline{\chi}^2}(\cdot) &= S_{\,\overline{\chi}^2}(\pm_\lambda n_1,n_2n_3;c) \\
    &= S(\pm_\lambda n_1 \overline{c_0},n_2n_3\overline{c_0};c')\,\,\sideset{}{^*}\sum_{y\,(\mo c_0)}\chi^2(y)e_{c_0}((\pm_\lambda n_1y+n_2n_3\overline{y})\overline{c'}) \label{eq:twisted_multiplicativity}
\end{align} 
by twisted multiplicativity of Sali\'e sums. From this point onwards, we work similarly to \cite[Section 4]{H}.

Considering the sum over $y$ in \eqref{eq:twisted_multiplicativity}, we decompose $y\in(\Z/c_0\Z)^\times$ as $y=y_0(1+p^{\ceil{\nu_p(c)/2}}y_1)$ with unique $y_0\in(\Z/c_0\Z)^\times/U_p(\ceil{\nu_p(c)/2})$ and $y_1\in\Z/p^{\nu_p(c)-\ceil{\nu_p(c)/2}}\Z$. By Lemma \ref{lemma:postnikov}, there exists an element $\ell_\chi\in\Z/p^{\nu_p(c)-1}\Z$ satisfying $p^{\nu_p(c)-a}\pdiv \ell_\chi$ and for which
\begin{align}
    \chi^2(1+p^{\ceil{\nu_p(c)/2}}y_1)=e_{p^{\nu_p(c)-\ceil{\nu_p(c)/2}}}(2\ell_\chi y_1), \label{eq:postnikov}
\end{align}
and we also know that $\ell_\chi\equiv \ell_{\chi\phi}\,(\mo p^{\nu_p(c)-b})$ for any Dirichlet character $\phi\,(\mo p^b)$. So we obtain
\begin{align}
    &\,\sideset{}{^*}\sum_{y\,(\mo c_0)}\chi^2(y)e_{c_0}((\pm_\lambda n_1y+n_2n_3\overline{y})\overline{c'}) \label{eq:twisted_kloosterman}\\
    &= \sum_{y_0}\chi^2(y_0)e_{c_0}((\pm_\lambda n_1y_0+n_2n_3\overline{y_0})\overline{c'}) \\
    &\quad\quad\,\,\,\, \cdot\sum_{y_1}\chi^2(1+p^{\ceil{\nu_p(c)/2}}y_1)e_{p^{\nu_p(c)-\ceil{\nu_p(c)/2}}}((\pm_\lambda n_1y_0y_1-n_2n_3\overline{y_0}y_1)\overline{c'}) \\
    &= \,\sideset{}{^*}\sum_{\substack{y\,(\mo c_0) \\ 2c'\ell_\chi y\pm_\lambda n_1y^2-n_2n_3\equiv0\,(\mo p^{\nu_p(c)-i})}}\chi^2(y)e_{c_0}((\pm_\lambda n_1y+n_2n_3\overline{y})\overline{c'}) \label{eq:after_postnikov}
\end{align}
for arbitrary $\ceil{\nu_p(c)/2}\leq i\leq \nu_p(c)$; this identity corresponds to \cite[Lemma 4.14]{H}.

Combining \eqref{eq:X_after_parity}, \eqref{eq:twisted_multiplicativity} and \eqref{eq:after_postnikov}, we see that
\begin{align}
    X(\cdot) = \frac{1}{2}S(\pm_\lambda n_1\overline{c_0},n_2n_3\overline{c_0};c')\sum_{\,\,\,\pm_b}g^{\pm_\lambda}(\cdot,\tfrac{1\mp\chi(-1)(\pm_b1)}{2},\cdot)(\Sigma(+,\cdot)\pm_b\chi(-1)\Sigma(-,\cdot)), \label{eq:X_contains_Sigma}
\end{align}
where 
\begin{align}
    \Sigma(\cdot) &\,= \Sigma(\pm_\eta,\pm_\lambda,\chi,n_1,n_2,n_3,c_0,c') \\
    &\coloneqq \frac{1}{\varphi(p^b)}\sum_{\phi\,(\mo p^b)}(\chi\phi)(\pm_\eta n_1)(\overline{\chi\phi})(n_2n_3) \\
    &\quad\quad\quad\quad\quad\quad\quad\,\,\,\, \cdot \,\sideset{}{^*}\sum_{\substack{y\,(\mo c_0) \\ 2c'\ell_\chi y\pm_\lambda n_1y^2-n_2n_3\equiv0\,(\mo p^{\nu_p(c)-i})}}(\chi\phi)^2(y)e_{c_0}((\pm_\lambda n_1y+n_2n_3\overline{y})\overline{c'}) \label{eq:average_of_twisted_kloosterman_sums}
\end{align}
for independent signs $\pm_\eta$ and $\pm_\lambda$. We have thus reduced to evaluating $\Sigma(\cdot)$.

To evaluate $\Sigma(\cdot)$, we follow the same strategy as the proof of \cite[Lemma 4.19]{H} for the principal series representation case. The latter however does not treat the case $a=\nu_p(c)$ (which is crucial in our work) due to the non-vanishing congruence condition
\begin{align} 
    2c'\ell_\chi y\pm_\lambda n_1y^2-n_2n_3\equiv0\,(\mo p^{\nu_p(c)-\ceil{\nu_p(c)/2}}).\label{eq:nonvanishing_congruence}
\end{align}
being possibly degenerate (see \cite[Remark 4.15]{H}). We proceed by case distinction, recalling that $1\leq b<a\leq\nu_p(c)$.

\begin{lemma}\label{lemma:cancellation_irrelevant}
    If $a+b\leq\nu_p(c)$, then 
    \begin{align}
        \Sigma(\cdot) = \begin{dcases}
            \chi(\pm_\eta n_1)\overline{\chi}(n_2n_3)\,\sideset{}{^*}\sum_{y\,(\mo c_0)}\chi^2(y)e_{c_0}((\pm_\lambda n_1y+n_2n_3\overline{y})\overline{c'}) &\text{ if }\pm_\eta1\cdot\pm_\lambda1=1 \\ 
            0 &\text{ if }\pm_\eta1\cdot\pm_\lambda1=-1
        \end{dcases}. \label{eq:Sigma_final_aplusb}
    \end{align}
\end{lemma}

\begin{proof}
    We choose $i=\max\{\ceil{\nu_p(c)/2},a\}$, so that 
    \begin{align*}
        \pm_\lambda n_1y^2\equiv n_2n_3\,(\mo p^{\nu_p(c)-i})
    \end{align*} 
    by \eqref{eq:nonvanishing_congruence}. Since $b\leq\nu_p(c)-i$, we then have
    \begin{align*}
        (\chi\phi)\bigg(\frac{\pm_\lambda n_1y^2}{n_2n_3}\bigg) =\chi\bigg(\frac{\pm_\lambda n_1y^2}{n_2n_3}\bigg)
    \end{align*}
    for all $\phi\,(\mo p^b)$. So we get
    \begin{align}
        \Sigma(\cdot) &= \chi(\pm_\eta n_1)\overline{\chi}(n_2n_3)\,\sideset{}{^*}\sum_{y\,(\mo c_0)}\chi^2(y)e_{c_0}((\pm_\lambda n_1y+n_2n_3\overline{y})\overline{c'})\frac{1}{\varphi(p^b)}\sum_{\phi\,(\mo p^b)}\phi(\pm_\eta1\cdot\pm_\lambda1),
    \end{align}
    from which we conclude.
\end{proof}
    
\begin{lemma}\label{lemma:cancellation_relevant}
    If $a+b>\nu_p(c)$, then 
    \begin{align}
        \Sigma(\cdot) &= \begin{dcases}
            0 &\text{ if }\pm_\eta1\cdot\pm_\lambda1=1 \\
            &\text{ or if }\pm_\eta1\cdot\pm_\lambda1=-1,\,a<\nu_p(c) \\
            \\
            \frac{1}{1-p^{-1}}\sum_{\psi\,(\mo p^{a-b})}\frac{\overline{\psi}(\mp_\lambda n_1n_2n_3\overline{c'}^2)\tau(\chi\psi)}{\tau(\chi\overline{\psi})} &\text{ if }\pm_\eta1\cdot\pm_\lambda1=-1,\,a=\nu_p(c)
        \end{dcases}.
    \end{align}
\end{lemma}

\begin{proof}
    We choose $i=a-1$, so that 
    \begin{align*}
        \pm_\lambda n_1y^2-n_2n_3\equiv-2c'\ell_\chi y\equiv-2c'\ell_{\chi\phi} y\not\equiv0\,(\mo p^{\nu_p(c)-i})
    \end{align*}
    by \eqref{eq:nonvanishing_congruence}. Since $b\geq\nu_p(c)-i$, we then have
    \begin{align*}
        \frac{\pm_\lambda n_1y^2}{n_2n_3}\not\equiv1\,(\mo p^b).
    \end{align*}
    We conclude that
    \begin{align*}
        \Sigma(\cdot) &= \chi(\pm_\eta n_1)\overline{\chi}(n_2n_3)\,\sideset{}{^*}\sum_{y\,(\mo c_0)}\chi^2(y)e_{c_0}((\pm_\lambda n_1y+n_2n_3\overline{y})\overline{c'})\frac{1}{\varphi(p^b)}\sum_{\phi\,(\mo p^b)}\phi\bigg(\frac{\pm_\eta n_1y^2}{n_2n_3}\bigg) \\
        &= \begin{dcases}
            0 &\text{ if }\pm_\eta1\cdot\pm_\lambda1=1 \\
            &\text{ or if }\pm_\eta1\cdot\pm_\lambda1=-1,\,a<\nu_p(c) \\
            \\
            \chi(\pm_\eta n_1)\overline{\chi}(n_2n_3) &\text{ if }\pm_\eta1\cdot\pm_\lambda1=-1,\,a=\nu_p(c) \\
            \cdot\,\sideset{}{^*}\sum_{\substack{y\,(\mo p^a) \\ \pm_\eta n_1\overline{n_2n_3}y^2\equiv1\,(\mo p^b)}}\chi^2(y)e_{p^a}((\pm_\lambda n_1y+n_2n_3\overline{y})\overline{c'})
        \end{dcases}.
    \end{align*}
    Indeed, by orthogonality of Dirichlet characters modulo $p^b$ we have 
    \begin{align*}
        \frac{1}{\varphi(p^b)}\sum_{\phi\,(\mo p^b)}\phi\bigg(\frac{\pm_\eta n_1y^2}{n_2n_3}\bigg) = \mathbbm{1}_{\{\pm_\eta n_1\overline{n_2n_3}y^2\equiv1\,(\mo p^b)\}},
    \end{align*}
    which is clearly $0$ if $\pm_\eta1\cdot\pm_\lambda1=1$. If $\pm_\eta1\cdot\pm_\lambda1=-1$ on the other hand, the congruences
    \begin{align*}
        \frac{\pm_\lambda n_1y^2}{n_2n_3}\equiv1\,(\mo p^{\nu_p(c)-a})\quad\text{ and }\quad\frac{\pm_\eta n_1y^2}{n_2n_3}\equiv1\,(\mo p^b)
    \end{align*}
    contradict each other unless $\nu_p(c)=a$. 

    If $\pm_\eta1\cdot\pm_\lambda1=-1$ and $a=\nu_p(c)$, then it is not too hard to further show that
    $$\Sigma(\cdot) = p^{a-b}\mathbbm{1}_{\{(c'\ell_\chi)^2\equiv\mp_\lambda n_1n_2n_3\,(\mo p^{a-b})\}}$$
    as long as $b\geq a/3$. If $b<a/3$, higher-order terms will appear due to the evaluation of $\overline{y}$ and the Postnikov formula for $\chi(\pm_\eta n_1\overline{n_2n_3}y^2)$, making exact computation a hard combinatorial task. So, to avoid restricting the range of $b$, we do not attempt to compute $\Sigma(\cdot)$ exactly, but instead resort to Fourier inversion as in \cite[Lemma 4.1.13]{K}, which turns out to be sufficient for our purposes. More precisely, writing $n\coloneqq n_1n_2n_3\overline{c'}^2$, we have
    \begin{align}
        \Sigma(\cdot) &= \frac{1}{\varphi(p^b)}\sum_{\phi\,(\mo p^b)}(\chi\phi)(\mp_\lambda n)S_{\overline{\chi\phi}^2}(\pm_\lambda n,1;p^a) \\
        &= \frac{1}{\varphi(p^b)}\sum_{\phi\,(\mo p^b)}(\chi\phi)(\mp_\lambda n)\frac{1}{\varphi(p^a)}\sum_{\eta\,(\mo p^a)}\overline{\eta}(\pm_\lambda n)\tau(\eta)\tau(\overline{\chi\phi}^2\eta) \\
        &= \frac{1}{\varphi(p^a)}\sum_{\psi\,(\mo p^a)}\overline{\psi}(n)\widehat{S}(\psi),
    \end{align}
    where 
    \begin{align*}
        \widehat{S}(\psi) &\coloneqq \sum_{m\,(\mo p^a)}\psi(m)\frac{1}{\varphi(p^b)}\sum_{\phi\,(\mo p^b)}(\chi\phi)(\mp_\lambda m)\frac{1}{\varphi(p^a)}\sum_{\eta\,(\mo p^a)}\overline{\eta}(\pm_\lambda m)\tau(\eta)\tau(\overline{\chi\phi}^2\eta) \\
        &\,= \frac{\psi(\mp_\lambda 1)\tau(\chi\psi)p^a}{\tau(\chi\overline{\psi})}\mathbbm{1}_{\{c(\overline{\chi}\psi)=p^a,\,c(\psi^2)\mid p^{a-b}\}}
    \end{align*}
    as in loc.\ cit. Also note that $\mathbbm{1}_{\{c(\overline{\chi}\psi)=p^a,\,c(\psi^2)\mid p^{a-b}\}}=\mathbbm{1}_{\{c(\psi)\mid p^{a-b}\}}$ since $p$ is odd and $a-b\geq1$. The result now follows.
\end{proof}

Combining \eqref{eq:X_after_parity}, \eqref{eq:twisted_multiplicativity}, \eqref{eq:X_contains_Sigma} and Lemmas \ref{lemma:cancellation_irrelevant}--\ref{lemma:cancellation_relevant}, we conclude the proof of Proposition \ref{prop:sum_of_kloosterman_sums}.
\end{proof}
\section{Poisson summation}\label{sec:poisson}
We recall the shape of the off-diagonal in Lemma \ref{lemma:summarizing_kuznetsov}, followed by the evaluation of $X(\cdot)$ in Proposition \ref{prop:sum_of_kloosterman_sums}. Our next step is to apply Poisson summation to each of the summation variables $n_1,n_2,n_3$ modulo $c$. As our computations closely follow those in \cite[Sections 4.3--4.4]{PeY}, we refer the reader to loc.\ cit.\ for additional details.

To prepare the off-diagonal for Poisson summation, we apply a smooth dyadic partition of unity to each of the summation variables $n_1,n_2,n_3$. Such a partition of unity in particular allows us to extend the summation range from $\Z_{\geq1}^3$ to $\Z^3$. Concretely, consider dyadic numbers $N_1,N_2,N_3,C$, and let 
$$w_0(\cdot)\coloneqq w_0(n_1,n_2,n_3,c)\coloneqq w_{N_1}(n_1)w_{N_2}(n_2)w_{N_3}(n_3)w_C(c)$$
be the component of a partition of unity that localizes the variables by $n_j\asymp N_j$, $c\asymp C$. This leads us to define the localized Bessel transforms
\begin{align*}
    &J_0^+(x,n_1,n_2,n_3,c,d,\delta,\pmb{\alpha},\pmb{\sigma}) \coloneqq \frac{i}{2}w_0(\cdot)\int_{\R}\frac{J_{2it}(x)}{\cosh(\pi t)}t h(t,\cdot)\,dt, \\
    &J_0^-(x,n_1,n_2,n_3,c,d,\delta,\pmb{\alpha},\pmb{\sigma}) \coloneqq\frac{1}{\pi} w_0(\cdot)\int_{\R}K_{2it}(x)\sinh(\pi t)th(t,\cdot)\,dt.
\end{align*}
Note that $w_0(\cdot)$ suppresses the singularity at $y=0$ of the weight functions $V_{1/2+\sigma_j\alpha_j}(y,\cdot)$ appearing in the definition of $h(t,\cdot)$.

Similar to \cite[Proposition 4.3]{PeY}, we formulate ranges for $N_1,N_2,N_3,C$ outside of which the dyadically localized sums are sufficiently small. Concretely, if we restrict to
\begin{align}
    N_1\ll(p^aT)^{1+\varepsilon},\quad N_2,N_3\ll d^{-1}(p^aT)^{1+\varepsilon},\quad p^{a+b}\ll C\ll (p^aT)^{100}, \label{eq:conditions_normal_case}
\end{align}
for terms with $p^{a+b}\mid c$, and
\begin{align}
    N_1\ll(p^aT)^{1+\varepsilon},\quad N_2,N_3\ll d^{-1}(p^aT)^{1+\varepsilon},\quad p^{a}\ll C\ll (p^aT)^{100}, \label{eq:conditions_b_case}
\end{align}
for terms with $p^a\pdiv c$, then the remaining off-diagonal terms account for an error of size $O_\varepsilon(p^{-b+a(-\frac{1}{2}+\varepsilon)})$. This follows from combining the Weil bound \eqref{eq:weil_bound} and Lemma \ref{lemma:bounds_for_bessel_transforms}.

Applying Poisson summation separately to each of the summation variables $n_1,n_2,n_3$ modulo $c$, we obtain sums over dual summation variables $m_1,m_2,m_3\in\Z$. Recalling the notation $c'\coloneqq cp^{-\nu_p(c)}$, we write
\begin{align}
    G^\pm_\chi(m_1,m_2,m_3;c) &\coloneqq c^{-3}\,\sideset{}{^*}\sum_{y\,(\mo c)}\sum_{x_1,x_2,x_3\,(\mo c)}\chi(x_1)\overline{\chi}(x_2x_3)\chi^2(y) \\
    &\quad\quad\quad\quad\quad\quad\quad\quad\quad\quad\quad\,\,\,\,\,\cdot e_c(m_1x_1+m_2x_2+m_3x_3\pm x_1y+x_2x_3\overline{y}), \\
    G^{b,\pm}_\chi(m_1,m_2,m_3;p^{a-b}c') &\coloneqq \frac{\chi(\mp1)}{(p^{a-b}c')^3(1-p^{-1})}\sum_{\psi\,(\mo p^{a-b})}\frac{\tau(\chi\psi)}{\tau(\chi\overline{\psi})}\,\sideset{}{^*}\sum_{y\,(\mo c')}\sum_{\substack{x_1,x_2,x_3\,(\mo p^{a-b}c')}} \\
    &\quad\quad\quad\quad\quad\quad\quad\quad\quad\quad\quad\,\, \overline{\psi}(\mp x_1x_2x_3\overline{c'}^2) e_{p^{a-b}c'}(m_1x_1+m_2x_2+m_3x_3\\
    &\quad\quad\quad\quad\quad\quad\quad\quad\quad\quad\quad\quad\quad\quad\quad\quad\quad\quad\quad + p^{a-b}\overline{p^a}(\pm x_1y+x_2x_3\overline{y})), \label{eq:Gb_def} \\
    K_0^\pm(m_1,m_2,m_3,c,d,\delta,\pmb{\alpha},\pmb{\sigma}) &\coloneqq \int_{\R^3}J_0^\pm\left(\frac{4\pi\sqrt{t_1t_2t_3}}{c},\cdot\right)e_c(-m_1t_1-m_2t_2-m_3t_3)\,dt_1dt_2dt_3, \\
    K_0^{b,\pm}(m_1,m_2,m_3,c,d,\delta,\pmb{\alpha},\pmb{\sigma}) &\coloneqq \int_{\R^3}J_0^\pm\left(\frac{4\pi\sqrt{t_1t_2t_3}}{c},\cdot\right)e_{p^{a-b}c'}(-m_1t_1-m_2t_2-m_3t_3)\,dt_1dt_2dt_3 \\
    &\,= K_0^\pm(m_1p^b,m_2p^b,m_3p^b,c,d,\delta,\pmb{\alpha},\pmb{\sigma}) \label{eq:Kb0pm_def}
\end{align}
for the relevant Gauss-type sums and Fourier transforms of the Bessel transforms, respectively. The superscript $b$ indicates contribution from terms with $p^a\pdiv c$, while the absence of the superscript indicates contribution from terms with $p^{a+b}\mid c$. Note that $G^\pm_\chi(\cdot)$ depends on the sign $\pm$; this dependence was not explicitly included in the definition of the same object $G(\cdot)$ in \cite[Section 4.3]{PeY}.

We separate the terms with $m_1m_2m_3=0$ from those with $m_1m_2m_3\neq0$ by defining 
\begin{align*}
    \mathcal{T}_0(\pm,\pm_\lambda,\chi,d,\pmb{\alpha},\pmb{\sigma}) &\coloneqq \frac{1}{C\sqrt{N_1N_2N_3}}\sum_{\substack{c\geq1 \\ p^{a+b}\mid c}}\sum_{\substack{m_1,m_2,m_3\in\Z \\ m_1m_2m_3=0}}G_\chi^{\pm_\lambda}(m_1,m_2,m_3;c) \\
    &\quad\quad\quad\quad\quad\quad\quad\quad\quad\quad\quad\quad\quad\quad\,\,\, \cdot \sum_{\,\,\,\pm_b}(\pm_\lambda1)^{\frac{1\mp_b1}{2}}K_0^{\pm_\lambda}(\cdot,\tfrac{1\mp\chi(-1)(\pm_b1)}{2},\cdot), \\
    \mathcal{T}_{\epsilon_1,\epsilon_2,\epsilon_3}(\pm,\pm_\lambda,\chi,d,\pmb{\alpha},\pmb{\sigma}) &\coloneqq \frac{1}{C\sqrt{N_1N_2N_3}}\sum_{\substack{c\geq1 \\ p^{a+b}\mid c}}\sum_{\substack{m_1,m_2,m_3\in\Z \\ m_1\epsilon_1,m_2\epsilon_2,m_3\epsilon_3\geq1}}G_\chi^{\pm_\lambda}(m_1,m_2,m_3;c) \\
    &\quad\quad\quad\quad\quad\quad\quad\quad\quad\quad\quad\quad\quad\quad\quad\quad\,\,\, \cdot \sum_{\,\,\,\pm_b}(\pm_\lambda1)^{\frac{1\mp_b1}{2}}K^{\pm_\lambda}(\cdot,\tfrac{1\mp\chi(-1)(\pm_b1)}{2},\cdot)
\end{align*}
for some choice of $\epsilon_1,\epsilon_2,\epsilon_3\in\{\pm1\}$ (i.e. there are 8 different configurations). Similarly define $\mathcal{T}^b_0(\pm,\pm_\lambda,\chi,d,\pmb{\alpha},\pmb{\sigma})$ and $\mathcal{T}^b_{\epsilon_1,\epsilon_2,\epsilon_3}(\pm,\pm_\lambda,\chi,d,\pmb{\alpha},\pmb{\sigma})$. Later we will in fact show that $\mathcal{T}_0(\cdot)=0$; see Lemma \ref{lemma:T0}.

We also apply a smooth dyadic partition of unity to those $m_1,m_2,m_3$ that are all non-zero. That is, we consider dyadic numbers $M_1,M_2,M_3$, and let 
\begin{align*}
    w(\cdot) &\coloneqq w(m_1,m_2,m_3) \coloneqq w_{M_1}(\abs{m_1})w_{M_2}(\abs{m_2})w_{M_3}(\abs{m_3})
\end{align*}
be the component of a partition of unity that localizes the variables by $\abs{m_j}\asymp M_j$. This leads us to define
\begin{align}
    K^\pm(m_1,m_2,m_3,c,d,\delta,\pmb{\alpha},\pmb{\sigma}) &\coloneqq w(\cdot)K^\pm_0(\cdot), \\
    K^{b,\pm}(m_1,m_2,m_3,c,d,\delta,\pmb{\alpha},\pmb{\sigma}) &\coloneqq w(\cdot)K_0^{b,\pm}(\cdot) \\
    &\,= w_{M_1}(\abs{m_1})w_{M_2}(\abs{m_2})w_{M_3}(\abs{m_3})K^\pm_0(m_1p^b,m_2p^b,m_3p^b,c,d,\delta,\pmb{\alpha},\pmb{\sigma}). \label{eq:Kbpm_rewritten}
\end{align}
Under the assumptions \eqref{eq:conditions_normal_case} and \eqref{eq:conditions_b_case}, applying repeated integration by parts to $K^\pm(\cdot),K^{b,\pm}(\cdot)$ shows that we may restrict to $M_j\ll (p^aT)^{A}$ for all $1\leq j\leq 3$ and some (possibly very large) fixed $A>0$.

Executing all of the above, we obtain the following dual expression for the off-diagonal.
\begin{prop}
    Suppose $T=(p^a)^{o(1)}$, and $\re\alpha_j\ll\frac{1}{\log p^a}$ and $\im\alpha_j=O(1)$ for all $1\leq j\leq 3$. Then, one has 
    \begin{align}
        &\sum_{\substack{d\geq1 \\ (d,p)=1}}\sum_{n_1,n_2,n_3\geq1}\sum_{\substack{c\geq1 \\ c\equiv0\,(\mo p^a)}}\frac{X(\cdot)}{cd\sqrt{n_1n_2n_3}} \\
        &= \frac{1}{2}\bigg(\sum_{\substack{d\geq1 \\ (d,p)=1}}\frac{1}{d}\bigg(\sum_{\substack{N_1,N_2,N_3 \\ \text{dyadic, satisfying \eqref{eq:conditions_normal_case}}}}\sum_{\substack{M_1,M_2,M_3\ll(p^aT)^A \\ \text{dyadic}}}\sum_{\epsilon_1,\epsilon_2,\epsilon_3\in\{\pm1\}}\mathcal{T}_{\epsilon_1,\epsilon_2,\epsilon_3}(\cdot) \\
        &\quad\quad\quad\quad\quad\quad\quad\, + \mathcal{T}_0^b(\cdot) + \sum_{\substack{N_1,N_2,N_3 \\ \text{dyadic, satisfying \eqref{eq:conditions_b_case}}}}\sum_{\substack{M_1,M_2,M_3\ll(p^aT)^A \\ \text{dyadic}}}\sum_{\epsilon_1,\epsilon_2,\epsilon_3\in\{\pm1\}}\mathcal{T}^b_{\epsilon_1,\epsilon_2,\epsilon_3}(\cdot)\bigg)\bigg) \\ \label{eq:off_diagonal_after_poisson}
        &\quad\quad\quad\quad\quad\quad\quad\quad\quad\quad\quad\quad\quad\quad\quad\quad\quad\quad\quad\quad\quad\quad\quad\quad\quad\quad\quad\quad\,\, + O_\varepsilon(p^{-b+a(-\frac{1}{2}+\varepsilon)}).
    \end{align}
\end{prop}

We study the sums $G^{\pm}_\chi(\cdot)$, $G^{b,\pm}_\chi(\cdot)$ and related generating functions in Sections \ref{sec:evaluation_of_gauss_sums}--\ref{sec:evaluation_of_generating_functions}. The sum $\mathcal{T}_0^b(\cdot)$ is treated in Section \ref{sec:cancellation}, where we show that it can be matched against the false main terms from Section \ref{sec:diag}. To finish the proof of Theorem \ref{thm:main}, it then suffices to bound $\mathcal{T}_{\epsilon_1,\epsilon_2,\epsilon_3}(\cdot)$ and $\mathcal{T}^b_{\epsilon_1,\epsilon_2,\epsilon_3}(\cdot)$, which we do in Section \ref{sec:completing_the_proof}.
\section{Evaluating $G^\pm_\chi(\cdot)$ and $G^{b,\pm}_\chi(\cdot)$}\label{sec:evaluation_of_gauss_sums}
We proceed similarly to \cite[Section 5]{PeY}.

\subsection{The case of $G^\pm_\chi(\cdot)$}
Write $c=p^{a+b}r$ for some $r\in\Z_{\geq1}$, and $r=r'p^{\nu_p(r)}$ with $(r',p)=1$. By the same calculations as in \cite[(5.3), (5.4)]{PeY}, we obtain the generating functions
\begin{align}
    Z^\pm(s_1,s_2,s_3,s_4) &\coloneqq \sum_{m_1,m_2,m_3\geq1}\sum_{\substack{c\geq1 \\ p^{a+b}\mid c}}\frac{cp^aG^\pm_\chi(m_1,m_2,m_3;c)e_c(\mp m_1m_2m_3)\chi(\mp1)}{m_1^{s_1}m_2^{s_2}m_3^{s_3}(c/p^{a+b})^{s_4}} \\
    &\,= \sum_{\substack{m_1,m_2,m_3,r\geq1 \\ (m_1,pr)=1}}\frac{H_\chi(\pm m_1,m_2,m_3;p^br)}{m_1^{s_1}m_2^{s_2}m_3^{s_3}r^{s_4}}, \label{eq:GH_identity}
\end{align}
where 
\begin{align}
    H_\chi(m_1,m_2,m_3;p^br) &\coloneqq \sum_{t,u\,(\mo p^a)}\chi(t)\overline{\chi}(u)\chi(-m_1+p^bru)\overline{\chi}(-m_2+p^brt) \\
    &\quad\quad\quad\quad\quad\,\,\,\, \cdot e_c(m_3(-m_1+p^bru)(-m_2+p^brt)-m_1m_2m_3). \label{eq:Hchi_def}
\end{align}
Replacing $r$ by $p^br$ everywhere in \cite[Section 5.1]{PeY}, the condition (5.7) in loc.\ cit.\ implies $H_\chi(m_1,m_2,m_3;p^br)=0$ if $p\mid m_1m_2m_3$, so in particular we obtain the following.
\begin{lemma}\label{lemma:T0}
    One has $\mathcal{T}_0(\cdot)=0$.
\end{lemma}

We assume from now on that $(m_1m_2m_3,p)=1$, since $H_\chi(\cdot)$ vanishes otherwise. Departing from (5.5) in loc.\ cit., and substituting consecutively $t\rightarrow \overline{m_3}t$ and $u\rightarrow\overline{m_2m_3}u$, we arrive at
\begin{align*}
    H_\chi(m_1,m_2,m_3;p^br) = \sum_{t,u\,(\mo p^a)}\chi(t+u)\overline{\chi}(u)\chi(-1+\overline{m_1m_2m_3}p^bru)\overline{\chi}(1+\overline{m_1m_2m_3}p^brt)e_{p^a}(t).
\end{align*}
Hence $H_\chi(\pm m_1,m_2,m_3;p^br)=H_\chi(1,1,1;p^{b+\nu_p(r)}\cdot\overline{\pm w})$ in \eqref{eq:GH_identity} for fixed $\nu_p(r)$ can be viewed as a function of $w\coloneqq m_1m_2m_3\overline{r'}$ on $(\Z/p^{a-b}\Z)^\times$. By orthogonality, we then have
\begin{align*}
    H_\chi(\pm m_1,m_2,m_3;p^br) &= \frac{1}{\varphi(p^{a-b})}\sum_{\psi\,(\mo p^{a-b})}\widehat{H}_\chi(\psi;p^{b+\nu_p(r)})\psi(\pm w),
\end{align*}
where
\begin{align}
    \widehat{H}_\chi(\psi;p^{b+\nu_p(r)}) &\coloneqq \sum_{v\,(\mo p^{a-b})}H_\chi(1,1,1;p^{b+\nu_p(r)}\overline{v})\overline{\psi}(v), \label{eq:Hhatplus_def}
\end{align}
similar to \cite[(5.11), (5.12)]{PeY}.
Note that we have dispensed with the case of $c(\psi)=p^a$ (since $b\geq1$), so we do not need to treat the hybrid character sum $g(\chi,\psi)$ as in \cite[Lemma 6.4]{PeY}.

Defining
\begin{align*}
    Z_{\fin}(\psi;s_4) \coloneqq \sum_{k\geq0}\frac{\widehat{H}_\chi(\psi;p^{b+k})}{p^{ks_4}},
\end{align*}
we conclude the following.
\begin{lemma}\label{lemma:Zpm_final}
    One has 
    \begin{align*}
        Z^\pm(s_1,s_2,s_3,s_4) = \frac{1}{\varphi(p^{a-b})}\sum_{\psi\,(\mo p^{a-b})}\psi(\pm1)\frac{L(s_1,\psi)L(s_2,\psi)L(s_3,\psi)L(s_4,\overline{\psi})}{\zeta^{(p)}(s_1+s_4)}Z_{\fin}(\psi;s_4).
    \end{align*}
\end{lemma}

\subsection{The case of $G^{b,\pm}_\chi(\cdot)$}
Recalling the definition of $G^{b,\pm}_\chi(\cdot)$ from \eqref{eq:Gb_def}, the following familiar identity (see e.g. \cite[Lemma 7.1]{PeYfourth} with $q=p^m$, $n=\gamma p^k$) will be helpful to us.
\begin{lemma}\label{lemma:gauss_sum_general}
Let $m\in\Z_{\geq1}$, and $\psi$ a non-trivial Dirichlet character modulo $p^m$ of conductor $p^j$ for some $1\leq j\leq m$. For $k\in\Z_{\geq0}$ and $\gamma\in(\Z/p^m\Z)^\times$, one has
\begin{align*}
    \sum_{v\,(\mo p^m)}e_{p^m}(\gamma vp^k)\psi(v) = \begin{dcases}
        p^{k}\overline{\psi}(\gamma)\tau(\psi') &\text{ if } k=m-j \\
        0 &\text{ otherwise}
    \end{dcases}.
\end{align*}
\end{lemma}

For each $1\leq j\leq 3$, we write $m_j=m_j'm_{j0}$ with $m_{j0}\mid p^\infty$ and $(m_j',p)=1$. Given a Dirichlet character $\psi$ modulo $p^{a-b}$, write its conductor as $p^j$ for some $0\leq j\leq a-b$. We then have by Lemma \ref{lemma:gauss_sum_general} (recalling that $(c',p)=1$)
\begin{align*}
    \sum_{x_1\,(\mo p^{a-b}c')}&\overline{\psi}(x_1)e_{p^{a-b}c'}((m_1\pm p^{a-b}\overline{p^a}y)x_1) \\
    &= \,\sideset{}{^*}\sum_{x_1\,(\mo p^{a-b})}\overline{\psi}(x_1)e_{p^{a-b}}(m_1\overline{c'}x_1)\sum_{x_1'\,(\mo c')}e_r((m_1\pm\overline{p^b}y)x_1'\overline{p^{a-b}}) \\
    &= c'\mathbbm{1}_{\{c'\mid p^bm_1\pm y\}}\cdot\begin{dcases}
        R_{p^{a-b}}(m_{10}) &\text{ if }j=0\\
        m_{10}\psi(m_1'\overline{c'})\tau(\overline{\psi'})\mathbbm{1}_{\{m_{10}=p^{a-b-j}\}} &\text{ if }j\geq1
    \end{dcases}
\end{align*}
and
\begin{align*}
    \sum_{x_2\,(\mo p^{a-b}c')}&\overline{\psi}(x_2)e_{p^{a-b}c'}((m_2+p^{a-b}\overline{p^a}x_3\overline{y})x_2) \\
    &= c'\mathbbm{1}_{\{c'\mid p^bm_2+x_3\overline{y}\}}\cdot\begin{dcases}
        R_{p^{a-b}}(m_{20}) &\text{ if }j=0 \\
        m_{20}\psi(m_2'\overline{c'})\tau(\overline{\psi'})\mathbbm{1}_{\{m_{20}=p^{a-b-j}\}} &\text{ if }j\geq1
    \end{dcases},
\end{align*}
so that 
\begin{align}
    &G^{b,\pm}_\chi(m_1,m_2,m_3;p^{a-b}c') \\
    &= \frac{\chi(\mp1)}{p^{3(a-b)}c'(1-p^{-1})}\,\sideset{}{^*}\sum_{\substack{y\,(\mo c') \\ y\equiv \mp p^bm_1\,(\mo c')}}\sum_{\substack{x_3\,(\mo p^{a-b}c') \\ p^bm_2+x_3\overline{y}\equiv0\,(\mo c')}}e_{p^{a-b}c'}(m_3x_3) \\
    &\quad\quad \cdot \sum_{\psi\,(\mo p^{a-b})}\frac{\overline{\psi}(\mp x_3\overline{c'}^2)\tau(\chi\psi)}{\tau(\chi\overline{\psi})}\cdot\begin{dcases}
        R_{p^{a-b}}(m_{10})R_{p^{a-b}}(m_{20}) &\text{ if }j=0 \\
        p^{2(a-b-j)}\psi(m_1'm_2'\overline{c'}^2)\tau(\overline{\psi'})^2 &\text{ if }j\geq1, \\
        &\quad\, m_{10}=m_{20}=p^{a-b-j} \\
        0 &\text{ otherwise}
    \end{dcases} \\
    &= \frac{\mathbbm{1}_{\{(m_1,c')=1\}}\chi(\mp1)}{p^{3(a-b)}c'(1-p^{-1})}\sum_{\substack{x_3\,(\mo p^{a-b}c') \\ x_3\equiv \pm p^{2b}m_1m_2\,(\mo c')}}e_{p^{a-b}c'}(m_3x_3)\sum_{\psi\,(\mo p^{a-b})}(\dotsc) \\
    &= \frac{\mathbbm{1}_{\{(m_1,c')=1\}}\chi(\mp1)}{p^{3(a-b)}c'(1-p^{-1})}e_{c'}(\pm m_1m_2m_3p^{2b}\overline{p^{a-b}})\,\sideset{}{^*}\sum_{x_3\,(\mo p^{a-b})}e_{p^{a-b}}(m_3\overline{c'}x_3)\sum_{\psi\,(\mo p^{a-b})}(\dotsc) \\
    &= \frac{\mathbbm{1}_{\{(m_1,c')=1\}}\chi(\mp1)}{cp^{2a-3b}(1-p^{-1})}e_c(\pm m_1m_2m_3p^{3b})H_\chi^{b,\pm}(m_1,m_2,m_3;p^{a-b}c'), \label{eq:Gbpm_final}
\end{align}
where
\begin{align}
    &H_\chi^{b,\pm}(m_1,m_2,m_3;p^{a-b}c') \\
    &\coloneqq \sum_{\psi\,(\mo p^{a-b})}\frac{\overline{\psi}(\mp c')\tau(\chi\psi)}{\tau(\chi\overline{\psi})}\cdot\begin{dcases}
        R_{p^{a-b}}(m_{10})R_{p^{a-b}}(m_{20})R_{p^{a-b}}(m_{30}) &\text{ if }j=0 \\
        p^{3(a-b-j)}\psi(m_1'm_2'm_3')\tau(\overline{\psi'})^3 &\text{ if }j\geq1, \\
        &\quad\, m_{10}=m_{20}=m_{30}=p^{a-b-j} \\
        0 &\text{ otherwise}
    \end{dcases}. \label{eq:Hbpm_final}
\end{align}

Define as before the generating functions
\begin{align*}
    &Z^{b,\,\pm}(s_1,s_2,s_3,s_4) \\
    &\coloneqq \sum_{m_1,m_2,m_3\geq1}\sum_{\substack{c\geq1 \\ p^a\pdiv c}}\frac{cp^{2a-3b}(1-p^{-1})G_\chi^{b,\pm}(m_1,m_2,m_3;p^{a-b}c')e_c(\mp m_1m_2m_3p^{3b})\chi(\mp1)}{m_1^{s_1}m_2^{s_2}m_3^{s_3}(c/p^a)^{s_4}} \\
    &\,= \sum_{\substack{m_1,m_2,m_3,c'\geq1 \\ (m_1,c')=1}}\frac{H^{b,\,\pm}_\chi(m_1,m_2,m_3;p^{a-b}c')}{m_1^{s_1}m_2^{s_2}m_3^{s_3}c'^{s_4}}.
\end{align*}
Writing 
\begin{align}
    &Z_{\fin}^{b}(\psi;s_1,s_2,s_3) \\
    &\coloneqq \varphi(p^{a-b})\cdot\begin{dcases}
        \sum_{m_{10},m_{20},m_{30}\mid p^\infty}\frac{R_{p^{a-b}}(m_{10})R_{p^{a-b}}(m_{20})R_{p^{a-b}}(m_{30})}{m_{10}^{s_1}m_{20}^{s_2}m_{30}^{s_3}} &\text{ if }j=0 \\
        \\
        p^{(a-b-j)(3-s_1-s_2-s_3)}\tau(\overline{\psi'})^3 &\text{ if }j\geq1
    \end{dcases}, \label{eq:Zbfin_def}
\end{align}
we conclude the following.
\begin{lemma}
    One has
    \begin{align*}
        &Z^{b,\pm}(s_1,s_2,s_3,s_4) \\
        &= \frac{1}{\varphi(p^{a-b})}\sum_{\psi\,(\mo p^{a-b})}\frac{\psi(\mp1)\tau(\chi\psi)}{\tau(\chi\overline{\psi})}\frac{L(s_1,\psi)L(s_2,\psi)L(s_3,\psi)L(s_4,\overline{\psi})}{\zeta^{(p)}(s_1+s_4)}Z_{\fin}^b(\psi;s_1,s_2,s_3).
    \end{align*}
\end{lemma}
\section{Evaluating $Z^\pm(\cdot)$ and $Z^{b,\,\pm}(\cdot)$}\label{sec:evaluation_of_generating_functions}
We proceed similarly to \cite[Sections 6--8]{PeY}.

\subsection{The case of $Z^\pm(\cdot)$}
We first evaluate $Z_{\fin}(\psi;s_4)$ for an arbitrary Dirichlet character $\psi$ modulo $p^{a-b}$, for which we proceed by case distinction on the conductor of $\psi$. Recall the definition of $\widehat{H}_\chi(\psi;p^{b+k})$ from \eqref{eq:Hhatplus_def}.

\begin{lemma}\label{lemma:Hhat_triv}
    If $\psi=\chi_0$ is the trivial character, then
\begin{align*}
    \widehat{H}_\chi(\psi;p^{b+k}) = \chi(-1)p^aR_{p^{a-b}}(p^k).
\end{align*}
\end{lemma}

The latter statement can be proved in similar fashion to \cite[Lemma 6.5]{PeY}.
So if $\psi=\chi_0$, we obtain
\begin{align}
    Z_{\fin}(\psi;s_4) &= \chi(-1)\bigg(-\frac{p^{2a-b-1}}{p^{(a-b-1)s_4}}+\sum_{k\geq a-b}\frac{p^a\varphi(p^{a-b})}{p^{ks_4}}\bigg) \\
    &= \chi(-1)p^{-a(s_4-2)+(b+1)(s_4-1)}\frac{p-p^{s_4}}{p^{s_4}-1} \\
    &= \begin{dcases}
        O(p^{\frac{3a-b}{2}}) &\text{ if }\re s_4=1/2+\varepsilon \\
        O(p^{a}) &\text{ if } \re s_4=1+\varepsilon
    \end{dcases}. \label{eq:bound_on_Zfin_psi_trivial}
\end{align}

\begin{lemma}\label{lemma:Hhat_nontriv}
    If $\psi$ has conductor $p^j$ for some $1\leq j\leq a-b$, then 
    \begin{align*}
        \widehat{H}_\chi(\psi;p^{b+k}) = p^{2a-b-2j}\tau(\overline{\psi'})\mathbbm{1}_{\{b+k=a-j\}}\sum_{v,t\,(\mo p^j)}\chi(p^{a-j}t+1)\chi(-1+vp^{a-j})\overline{\chi}(1+vp^{2(a-j)}t)\psi(vt).
    \end{align*}
\end{lemma}

\begin{proof}
    Combining \eqref{eq:Hchi_def} and \eqref{eq:Hhatplus_def}, then changing variables $v\rightarrow v\overline{u}$ and $t\rightarrow tu$, we arrive at
    \begin{align}
        \widehat{H}_\chi(\psi;p^{b+k}) &= \sum_{v\,(\mo p^{a-b})}\sum_{t,u\,(\mo p^a)}\chi(t+1)\chi(-1+vp^{b+k})\overline{\chi}(1+vp^{b+k}t)\psi(v\overline{u})e_{p^a}(tu).
    \end{align}
    Applying next Lemma \ref{lemma:gauss_sum_general}, we get
    \begin{align}
        \widehat{H}_\chi(\psi;p^{b+k}) = p^{a-j}\tau(\overline{\psi'})\sum_{v\,(\mo p^{a-b})}\sum_{t\,(\mo p^j)}\chi(p^{a-j}t+1)\chi(-1+vp^{b+k})\overline{\chi}(1+vp^{a+b+k-j}t)\psi(vt).
    \end{align}
    We now argue as in the proof of \cite[Lemma 6.8]{PeY} to show that $\widehat{H}_\chi(\psi;p^{b+k})$ vanishes if $b+k\neq a-j$. Indeed, if $b+k<a-j$ then this follows from the fact that $\overline{\chi}(1+vp^{a+b+k-j}t)\psi(v)$ is $p^j$-periodic in $v$, and $\chi(-1+vp^{b+k})$ is at least $p^{j+1}$-periodic in $v$. If $b+k>a-j$ on the other hand, then the claim follows from the observation that $\chi(-1+vp^{b+k})\overline{\chi}(1+vp^{a+b+k-j}t)$ is $p^{j-1}$-periodic in $v$, and $\psi(v)$ has least period $p^{j}$ in $v$. So we obtain
    \begin{align*}
        &\widehat{H}_\chi(\psi;p^{b+k}) \\
        &= p^{a-j}\tau(\overline{\psi'})\mathbbm{1}_{\{b+k=a-j\}}\sum_{v\,(\mo p^{a-b})}\sum_{t\,(\mo p^j)}\chi(p^{a-j}t+1)\chi(-1+vp^{a-j})\overline{\chi}(1+vp^{2(a-j)}t)\psi(vt),
    \end{align*}
    from which we conclude.
\end{proof}
By \cite[Lemma 2.8]{PeYfourth}, the sum over $v,t\,(\mo p^j)$ in Lemma \ref{lemma:Hhat_nontriv} is $O(p^j)$. So if $c(\psi)=p^j$ with $1\leq j\leq a-b$, we obtain 
\begin{align}
    Z_{\fin}(\psi;s_4) &= O(p^{2a-b-\frac{j}{2}-(a-b-j)\re s_4}) \\
    &= O(p^{\frac{3a-b}{2}}) \quad \text{ if }\re s_4=1/2+\varepsilon. \label{eq:bound_on_Zfin}
\end{align}

\subsection{The case of $Z^{b,\,\pm}(\cdot)$}

Recalling \eqref{eq:Zbfin_def}, we only need to evaluate $Z_{\fin}^{b}(\psi;s_1,s_2,s_3)$ for the trivial character $\psi=\chi_0$. That is, one has
\begin{align}
    Z_{\fin}^{b}(\psi;s_1,s_2,s_3) &= \varphi(p^{a-b})\prod_{1\leq n\leq 3}\bigg(-\frac{p^{a-b-1}}{p^{(a-b-1)s_n}}+\sum_{k\geq a-b}\frac{\varphi(p^{a-b})}{p^{ks_n}}\bigg) \\
    &= \varphi(p^{a-b})p^{3(a-b-1)}\prod_{1\leq n\leq 3}\frac{p^{-(a-b-1)s_n}(p-p^{s_n})}{p^{s_n}-1} \\
    &= \begin{dcases}
        O(p^{\frac{5(a-b)}{2}}) &\text{ if } \re s_n=1/2+\varepsilon\text{ for all }1\leq n\leq 3 \\
        O(p^{a-b}) &\text{ if } \re s_n=1+\varepsilon\text{ for all }1\leq n\leq 3
    \end{dcases}. \label{eq:bound_on_Zfin_psi_trivial_b_case}
\end{align}

If on the other hand $\psi$ has conductor $p^j$ for some $1\leq j\leq a-b$, then note that
\begin{align}
    Z_{\fin}^{b}(\psi;s_1,s_2,s_3) = O(p^{\frac{5(a-b)}{2}}) \quad\text{ if }\re s_n=1/2+\varepsilon\text{ for all }1\leq n\leq 3. \label{eq:bound_on_Zfin_b_case}
\end{align}

\subsection{Large-sieve inequalities}
We state the analogs of \cite[Lemma 8.1]{PeY} for $Z^\pm(\cdot)$ and $Z^{b,\,\pm}(\cdot)$. First observe that we have the decompositions
\begin{align}
    Z^\pm(s_1,s_2,s_3,s_4) &= Z_0^\pm(s_1,s_2,s_3,s_4) + Z_1^\pm(s_1,s_2,s_3,s_4), \label{eq:Zpm_decomposition} \\
    Z^{b,\,\pm}(s_1,s_2,s_3,s_4) &= Z^{b,\,\pm}_0(s_1,s_2,s_3,s_4) + Z_1^{b,\,\pm}(s_1,s_2,s_3,s_4), \label{eq:Zbpm_decomposition}
\end{align}
where $Z_0^\pm(\cdot)$, $Z_0^{b,\,\pm}(\cdot)$ denote the contributions of the trivial character $\psi=\chi_0$, and $Z_1^\pm(\cdot)$, $Z_1^{b,\,\pm}(\cdot)$ the remainders.

\begin{lemma}\label{lemma:hybrid}
    The function $Z_0^\pm(\cdot)$ is meromorphic whenever $\re s_n\geq 1/2$ for all $1\leq n\leq 4$, and analytic whenever all $\re s_n>1$. It has a pole whenever some $s_n=1$ and the other variables are kept fixed. The function $Z_1^\pm(\cdot)$ is analytic whenever $\re s_n\geq1/2$ for all $1\leq n\leq 4$. 
    
    Moreover, if all $\re s_n\geq\sigma>1$, then $Z^\pm_0(\cdot)\ll_\sigma p^b$. If $\re u_n=1/2+\varepsilon$ for all $n$ and some small $\varepsilon>0$, then we have
    \begin{align}
        \int_{-Y}^Y\abs{Z_\delta^\pm(u_1-iy,u_2-iy,u_3-iy,u_4+iy)}\,dy \ll_\varepsilon Y^{1+\varepsilon} \cdot 
        \begin{dcases}
            p^{\frac{a+b}{2}(1+\varepsilon)} &\text{ if }\delta=0 \\
            p^{\frac{3a-b}{2}(1+\varepsilon)} &\text{ if }\delta=1
        \end{dcases} \label{eq:lemma_hybrid}
    \end{align}
    for $Y\gg1$. 
\end{lemma}

\begin{proof}
    The meromorphicity and analyticity statements are straightforward. 
    
    To prove the upper bounds, we first recall the bounds \eqref{eq:bound_on_Zfin_psi_trivial} and \eqref{eq:bound_on_Zfin} on $Z_{\fin}(\psi;s_4)$. Then \eqref{eq:lemma_hybrid} for $\delta=0$ follows from applying a large-sieve inequality to the fourth moment of $\zeta(s)$. To show \eqref{eq:lemma_hybrid} for $\delta=1$, it suffices (due to Lemma \ref{lemma:Zpm_final} and \eqref{eq:bound_on_Zfin}) to bound the fourth moment 
    \begin{align}
        \sum_{\substack{\psi\,(\mo p^{a-b}) \\ \psi\neq\chi_0}}\int_{-Y}^Y\abs{L(u_1-iy,\psi)L(u_2-iy,\psi)L(u_3-iy,\psi)L(u_4+iy,\overline{\psi})}\,dy. \label{eq:lemma_fourth_moment}
    \end{align}
    Since the integrand in \eqref{eq:lemma_fourth_moment} is $\leq \abs{L(u_1-iy,\psi)L(u_2-iy,\psi)}^2+\abs{L(u_3-iy,\psi)L(u_4+iy,\overline{\psi})}^2$, we can without loss of generality proceed by just bounding
    \begin{align}
        \sum_{\substack{\psi\,(\mo p^{a-b}) \\ \psi\neq\chi_0}}\int_{-Y}^Y\abs{L(u_1-iy,\psi)L(u_2-iy,\psi)}^2\,dy. \label{eq:lemma_before_AFE}
    \end{align}
    By the approximate functional equation \cite[Theorem 5.3]{IK}, we have 
    \begin{align}
        &L(u_1-iy,\psi)L(u_2-iy,\psi) \\
        &= \sum_{n\geq1}\sigma_{u_1,u_2}(n)\psi(n)n^{iy}V_{u_1-iy,u_2-iy}\bigg(\frac{n}{c(\psi)}\bigg) \label{eq:AFE_lemma} \\
        &\quad\, + \epsilon(u_1-iy,u_2-iy,\psi)\sum_{n\geq1}\sigma_{-u_1,-u_2}(n)\overline{\psi}(n)n^{-1-iy}V_{1-u_1+iy,1-u_2+iy}\bigg(\frac{n}{c(\psi)}\bigg), 
    \end{align}
    where 
    \begin{align*}
        &\sigma_{s_1,s_2}(n) \coloneqq \sum_{\substack{n_1,n_2\geq1 \\ n=n_1n_2}}n_1^{-s_1}n_2^{-s_2}, \\
        &V_{s_1,s_2}(t) \coloneqq \frac{1}{2\pi i}\int_{(\sigma)}t^{-v}\frac{\gamma(s_1+v,\psi)\gamma(s_2+v,\psi)}{\gamma(s_1,\psi)\gamma(s_2,\psi)}\frac{G(v)}{v}\,dv, \\
        &\gamma(s,\psi) \coloneqq \pi^{-\frac{s}{2}}\Gamma\bigg(\frac{s+\delta_\psi}{2}\bigg), \\
        &\epsilon(s_1,s_2,\psi) \coloneqq i^{-2\delta_\psi}\tau(\psi')^2c(\psi)^{-s_1-s_2}\frac{\gamma(1-s_1,\psi)\gamma(1-s_2,\psi)}{\gamma(s_1,\psi)\gamma(s_2,\psi)},
    \end{align*}
    and $\sigma\in\R$ lies to the right of all poles of the integrand defining $V_{s_1,s_2}(t)$. Both sums in \eqref{eq:AFE_lemma} have ``effective length'' $c(\psi)(1+\abs{y})$. So it suffices to bound \eqref{eq:lemma_before_AFE} with $L(u_1-iy,\psi)L(u_2-iy,\psi)$ replaced by (without loss of generality)
    \begin{align}
        &Q(\psi,u_1,u_2,y) \coloneqq \sum_{1\leq n\ll p^{a-b}Y}\sigma_{u_1,u_2}(n)\psi(n)n^{iy}V_{u_1-iy,u_2-iy}\bigg(\frac{n}{c(\psi)}\bigg) \\
        &= \frac{1}{2\pi i}\int_{(\varepsilon)}c(\psi)^v\frac{\gamma(u_1-iy+v)\gamma(u_2-iy+v)}{\gamma(u_1-iy)\gamma(u_2-iy)}\frac{G(v)}{v}\bigg(\sum_{1\leq n\ll p^{a-b}Y}\sigma_{u_1+v,u_2+v}(n)\psi(n)n^{iy}\bigg)\,dv
    \end{align}
    for $\varepsilon>0$ small. Note that we should distinguish between even and odd characters $\psi$ (since $\gamma(s)$ depends on $\delta_\psi$), but we keep this implicit for the sake of brevity. Applying Cauchy--Schwarz, we obtain 
    \begin{align*}
        \abs{Q(\cdot)}^2 &\leq \int_{(\varepsilon)}\bigg\lvert c(\psi)^v\frac{\gamma(u_1-iy+v)\gamma(u_2-iy+v)}{\gamma(u_1-iy)\gamma(u_2-iy)}\frac{G(v)^{1/2}}{v}\bigg\rvert^2\,dv \\
        &\quad\, \cdot \int_{(\varepsilon)}\bigg\lvert G(v)^{1/2}\sum_{1\leq n\ll p^{a-b}Y}\sigma_{u_1+v,u_2+v}(n)\psi(n)n^{iy}\bigg\rvert^2\,dv.
    \end{align*}
    The first factor is $\ll_\varepsilon (p^{a-b}(1+\abs{y}))^{\varepsilon}$ by the third bound in the proof of \cite[Proposition 5.4]{IK}. Swapping the orders of $\sum_\psi$, $\int_{-Y}^Y$ and $\int_{(\varepsilon)}$, it now suffices to apply the hybrid large sieve \cite[Theorem 2]{G} in the form
    \begin{align*}
        \sum_{\psi\,(\mo p^{a-b})}\int_{-Y}^Y\bigg\lvert\sum_{1\leq n \ll p^{a-b}Y}\sigma_{u_1+v,u_2+v}(n)\psi(n)n^{iy}\bigg\rvert^2\,dy &\ll p^{a-b}Y\sum_{1\leq n\ll p^{a-b}Y}\abs{\sigma_{u_1+v,u_2+v}(n)}^2 \\
        &\ll_\varepsilon p^{a-b}Y\sum_{1\leq n\ll p^{a-b}Y}n^{-1+\varepsilon} \\
        &\ll_\varepsilon (p^{a-b}Y)^{1+\varepsilon}.
    \end{align*}
    Combining everything, we obtain the desired result.
\end{proof}

\begin{lemma}\label{lemma:hybrid_b_case}
    The function $Z_0^{b,\,\pm}(\cdot)$ is meromorphic whenever $\re s_n\geq 1/2$ for all $1\leq n\leq 4$, and analytic whenever all $\re s_n>1$. It has a pole whenever some $s_n=1$ and the other variables are kept fixed. The function $Z_1^{b,\,\pm}(\cdot)$ is analytic whenever $\re s_n\geq1/2$ for all $1\leq n\leq 4$. 

    Moreover, if all $\re s_n\geq\sigma>1$, then $Z^{b,\,\pm}_0(\cdot)\ll_\sigma1$. If $\re u_n=1/2+\varepsilon$ for all $n$ and some small $\varepsilon>0$, then we have
    \begin{align*}
        \int_{-Y}^Y\abs{Z_\delta^{b,\,\pm}(u_1-iy,u_2-iy,u_3-iy,u_4+iy)}\,dy \ll_\varepsilon Y^{1+\varepsilon}\cdot\begin{dcases}
            p^{\frac{3(a-b)}{2}(1+\varepsilon)} &\text{ if }\delta=0 \\
            p^{\frac{5(a-b)}{2}(1+\varepsilon)} &\text{ if }\delta=1
        \end{dcases}
    \end{align*}
    for $Y\gg1$. 
\end{lemma}

\begin{proof}
    The proof of Lemma \ref{lemma:hybrid} carries over almost verbatim, the only modification being that \eqref{eq:bound_on_Zfin_psi_trivial} and \eqref{eq:bound_on_Zfin} are now replaced by \eqref{eq:bound_on_Zfin_psi_trivial_b_case} and \eqref{eq:bound_on_Zfin_b_case}.
\end{proof}
\section{Cancellation between certain diagonal and off-diagonal terms}\label{sec:cancellation}
Recall the decompositions \eqref{eq:avg_setup}, \eqref{eq:f_final_expression} and \eqref{eq:off_diagonal_after_poisson}. In analogy with the holomorphic case in \cite[Section 3]{Pe}, we assert the following.

\begin{prop}[Cancellation of false main terms]\label{prop:cancellation}
If $T=(p^a)^{o(1)}$, and $\re\alpha_j\ll \frac{1}{\log p^a}$ and $\im\alpha_j=O(1)$ for all $1\leq j\leq 3$, we have
\begin{align}
    &\sum_{\pmb{\sigma}\in\{\pm1\}^3}(\pm1)^{\frac{1-\sigma}{2}}\bigg(\frac{1}{\varphi(p^b)}\sum_{\phi\,(\mo p^b)}\bigg(\frac{1}{2p^{a\alpha}}\mathcal{D}(\pm,\chi\phi,\pmb{\alpha},\pmb{\sigma}) - \text{MT}^\pm_{nh}(\chi\phi,\pmb{\alpha},\pmb{\sigma})\bigg) \label{eq:diag_before_cancellation} \\
    &\quad\quad\quad\quad\quad\quad\quad\,\,\, + \frac{1}{2p^{a\alpha}}\sum_{\,\,\,\pm_\lambda}(\pm_\lambda1)^{\frac{1\mp1}{2}}\sum_{\substack{d\geq1 \\ (d,p)=1}}\frac{1}{d}\sum_{\substack{N_1,N_2,N_3,C \\ \text{dyadic}}}\mathcal{T}_0^b(\pm,\pm_\lambda,\chi,d,\pmb{\alpha},\pmb{\sigma})\bigg) \label{eq:second_match} \\
    &\ll_\varepsilon p^{a(-\frac{1}{3}+\varepsilon)}. \label{eq:after_canceling}\\
\end{align}
\end{prop}
To prove the latter result, we express $G^{b,\pm}_\chi(\cdot)$ in terms of Ramanujan sums in Section \ref{sec:gauss_sums_ramanujan}, next we obtain Mellin inversion formulas for $K^{b,\,\pm}_0(\cdot)$ in Sections \ref{sec:Kbp_mellin_inverted}--\ref{sec:Kbm_mellin_inverted}, then we combine these results in Section \ref{sec:combining_all_info_offdiag} to obtain an expression for the sum of $\mathcal{T}_0^b(\cdot)$'s in \eqref{eq:second_match} that matches with the diagonal terms in \eqref{eq:diag_before_cancellation}.

\subsection{$G^{b,\,\pm}_\chi(\cdot)$ with $m_1m_2m_3=0$}\label{sec:gauss_sums_ramanujan}
Suppose throughout that $m_1m_2m_3=0$. Assuming that the written $m_j$'s are $\neq0$, we have by \eqref{eq:Gbpm_final} and \eqref{eq:Hbpm_final} the simplified expressions
\begin{align}
    &G^{b,\pm}_\chi(0,0,0;p^{a-b}c') = \mathbbm{1}_{\{c'=1\}}\chi(\mp1)(1-p^{-1})^2, \label{eq:gauss_sum_all0} \\
    &G^{b,\pm}_\chi(m_1,0,0;p^{a-b}c') = \mathbbm{1}_{\{(m_1,c')=1\}}\chi(\mp1)\frac{1-p^{-1}}{p^{a-b}c'}R_{p^{a-b}}(m_1), \label{eq:gauss_sum_m1nonzero} \\
    &G^{b,\pm}_\chi(0,m_2,0;p^{a-b}c') = \mathbbm{1}_{\{c'=1\}}\chi(\mp1)\frac{1-p^{-1}}{p^{a-b}}R_{p^{a-b}}(m_2), \label{eq:gauss_sum_m2nonzero} \\
    &G^{b,\pm}_\chi(0,0,m_3;p^{a-b}c') = \mathbbm{1}_{\{c'=1\}}\chi(\mp1)\frac{1-p^{-1}}{p^{a-b}}R_{p^{a-b}}(m_3), \label{eq:gauss_sum_m3nonzero} \\
    &G^{b,\pm}_\chi(m_1,m_2,0;p^{a-b}c') = \mathbbm{1}_{\{(m_1,c')=1\}}\chi(\mp1)\frac{1}{p^{2(a-b)}c'}R_{p^{a-b}}(m_1)R_{p^{a-b}}(m_2), \label{eq:gauss_sum_m1m2nonzero} \\
    &G^{b,\pm}_\chi(m_1,0,m_3;p^{a-b}c') = \mathbbm{1}_{\{(m_1,c')=1\}}\chi(\mp1)\frac{1}{p^{2(a-b)}c'}R_{p^{a-b}}(m_1)R_{p^{a-b}}(m_3), \label{eq:gauss_sum_m1m3nonzero} \\
    &G^{b,\pm}_\chi(0,m_2,m_3;p^{a-b}c') = \mathbbm{1}_{\{c'=1\}}\chi(\mp1)\frac{1}{p^{2(a-b)}}R_{p^{a-b}}(m_2)R_{p^{a-b}}(m_3).\label{eq:gauss_sum_m2m3nonzero}
\end{align}
Note that these formulas match \cite[(14)--(18)]{Pe} with different normalizations.

\subsection{$K^{b,+}_0(\cdot)$ with $m_1m_2m_3=0$}\label{sec:Kbp_mellin_inverted}
We now compute an integral representation for $K^{b,+}_0(\cdot)$ in terms of $\Gamma(\cdot)$. First recall the Mellin--Barnes integral representation
\begin{align}
    x^{-1}J_\nu(x) &= \frac{1}{2\pi i}\int_{-i\infty}^{+i\infty}\frac{\Gamma(-s)}{\Gamma(1+s+\nu)}2^{-2s-\nu}x^{-1+2s+\nu}\,ds \label{eq:mellin_bessel_J} \\
    &= \frac{1}{2\pi i}\int_{\frac{1+\re\nu}{2}-i\infty}^{\frac{1+\re\nu}{2}+i\infty}\frac{\Gamma\left(\frac{1-2s+\nu}{2}\right)}{\Gamma\left(\frac{1+2s+\nu}{2}\right)}2^{1-2s}x^{2(s-1)}\,ds
\end{align}
for $x,\re \nu\in\R_{>0}$, where the integration path in the first integral can be chosen as the vertical line $(0)$ with a small indentation to the left near $s=0$ as this is a pole of $\Gamma(-s)$; see e.g. \cite[(3.4.21)]{PaK}. We would like to take $\nu=2it$ with $t\in\R$, and $x=4\pi\sqrt{t_1t_2t_3}/c$. By Stirling's formula, the modulus of the integrand in \eqref{eq:mellin_bessel_J} is of size $O(\abs{\im s}^{-2\re s-\re\nu-1})$ as $\im s\rightarrow\pm\infty$, so the corresponding integral is not absolutely convergent for $\re \nu=0$. However, shifting the integration contour in \eqref{eq:mellin_bessel_J} slightly to the right while keeping an indentation on the left of $s=0$, we obtain by analytic continuation that \eqref{eq:mellin_bessel_J} with integration over this new contour (call it $\mathcal{C}$) also holds for $\re\nu=0$. In particular, we get
\begin{align*}
    &\frac{c}{4\pi\sqrt{t_1t_2t_3}}J_{2it}\bigg(\frac{4\pi\sqrt{t_1t_2t_3}}{c}\bigg) = \frac{1}{2\pi i}\int_{\mathcal{C}+\frac{1}{2}}\frac{\Gamma(1/2-s+it)}{\Gamma(1/2+s+it)}2^{1-2s}\left(\frac{4\pi}{c}\right)^{2(s-1)}(t_1t_2t_3)^{s-1}\,ds,
\end{align*}
where $\mathcal{C}+\frac{1}{2}$ denotes the contour $\mathcal{C}$ horizontally translated to the right by a shift of $\frac{1}{2}$.

We also have the Mellin inversion formula
\begin{align}
    e(-mt_j) &= \frac{1}{2\pi i}\int_{(\theta)}\frac{\Gamma(s)}{(2\pi im)^s}t_j^{-s}\,ds \label{eq:e_mellin_inverted}
\end{align}
for $m\in\R_{\neq0}$ and $\theta\in(0,1/2]$, the integral being convergent as argued in \cite[Section 2.4, Example 1]{PaK}. So by definition of $V_{1/2+\alpha_j}(\cdot)$ and Mellin convolution/Parseval's formula (see e.g. \cite[(3.1.14)]{PaK}), we obtain for $m\in\R$ and $n\in\R_{>0}$ that
\begin{align*}
    &\int_{\R_{>0}}V_{1/2+\alpha_j}(nt_j,t,\delta)e(-mt_j)t_j^s\frac{dt_j}{t_j} \\
    &= 
    \begin{dcases}
        \frac{1}{2\pi i}\int_{(\theta_j)}n^{-u}\gamma_{\rat}(\delta,t,u,\alpha_j)\frac{G(u-\alpha_j)}{u-\alpha_j}\frac{\Gamma(s-u)}{(2\pi im)^{s-u}}du &\text{ if }m\neq0,\,\theta_j\in(\re\alpha_j,+\infty), \\
        &\quad\, \re s\in(\theta_j,1/2+\theta_j] \\
        \\
        n^{-s}\gamma_{\rat}(\delta,t,s,\alpha_j)\frac{G(s-\alpha_j)}{s-\alpha_j} &\text{ if }m=0,\,\re s\in(\re\alpha_j,+\infty)
    \end{dcases}.
\end{align*}

Recalling the definition of $K_0^{b,+}(\cdot)$ from \eqref{eq:Kb0pm_def}, we now have the Mellin inversion formula
\begin{align*}
    &\sum_{\substack{N_1,N_2,N_3 \\ \text{dyadic}}}\frac{1}{\sqrt{N_1N_2N_3}}K_0^{b,+}(m_1,m_2,m_3,c,d,\delta,\pmb{\alpha},\pmb{\sigma}) \\
    &=\frac{ic}{4\pi}w_{C}(c)\int_{\R}\frac{th_0(t)}{\cosh(\pi t)}\Gamma_{\rat}(\delta,t,\pmb{\alpha},\pmb{\sigma})\bigg(\frac{1}{2\pi i}\int_{\mathcal{C}+\frac{1}{2}}\frac{\Gamma(1/2-s+it)}{\Gamma(1/2+s+it)}\left(\frac{2\pi}{c}\right)^{2s} \\
    &\quad\,\cdot\bigg(\int_{\R_{>0}}V_{1/2+\sigma_1\alpha_1}\bigg(\frac{t_1}{p^a},t,\delta\bigg)e_{p^{a-b}c'}(-m_1t_1)t_1^s\frac{dt_1}{t_1} \\
    &\quad\quad\,\,\,\, \cdot\prod_{j\in\{2,3\}}\int_{\R_{>0}}V_{1/2+\sigma_j\alpha_j}\bigg(\frac{dt_j}{p^a},t,\delta\bigg)e_{p^{a-b}c'}(-m_jt_j)t_j^s\frac{dt_j}{t_j}\bigg)ds\bigg)dt.
\end{align*}
If $m_j=0$ for some $1\leq j\leq 3$, then the Gaussian $G(s-\sigma_j\alpha_j)$ resulting from the $t_j$-integral has rapid decay on vertical lines, resolving the earlier lack of absolute convergence of the $s$-integral for $\re\nu=0$, so it is now justified to deform the contour $\mathcal{C}+\frac{1}{2}$ back to the left to a vertical line $(\eta)$ with $\eta\in(\re\alpha_j,1/2)$.

We now specify to the relevant cases with $m_1m_2m_3=0$. To simplify our formulas, we use the notation
\begin{align*}
    \{(m_1,m_2,m_3)\in\Z^3\,:\,m_1m_2m_3=0\} = \{(0,0,0)\} \sqcup \bigsqcup_{1\leq i\leq 3}A_i \sqcup \bigsqcup_{\substack{1\leq i<j\leq 3}}P_{ij},
\end{align*}
where the $A_i\coloneqq\{(m_1,m_2,m_3)\in\Z^3\,:\,m_i\neq0,\,m_j=0\text{ for }i\neq j\}$ describe coordinate axes and the $P_{ij}\coloneqq\{(m_1,m_2,m_3)\in\Z^3\,:\,m_im_j\neq0,\,m_k=0\text{ for }k\neq i,j\}$ describe coordinate planes. We then obtain
\begin{align}
    &\sum_{\substack{N_1,N_2,N_3 \\ \text{dyadic}}}\frac{1}{\sqrt{N_1N_2N_3}}K_0^{b,+}(m_1,m_2,m_3,p^ac',\cdot) \\
    &= \frac{ip^ac'}{4\pi}w_C(p^ac')\int_{\R}\frac{th_0(t)}{\cosh(\pi t)}\bigg(\frac{1}{2\pi i}\int_{(\frac{1}{4})}\frac{\Gamma(1/2-s+it)}{\Gamma(1/2+s+it)}\bigg(\frac{2\pi}{c'}\bigg)^{2s}Q(\cdot)\,ds\bigg)\,dt, \label{eq:Kbp0_general}
\end{align}
where $Q(\cdot) = Q(m_1,m_2,m_3,c',d,\delta,t,s,\pmb{\alpha},\pmb{\sigma})$ is defined as follows. If $(m_1,m_2,m_3)=(0,0,0)$ and $c'=1$, then
\begin{align}
    Q(\cdot) \coloneqq d^{-2s}p^{as}\prod_{1\leq j\leq 3}\gamma_{\rat}(\delta,t,s,\alpha_j)\frac{G(s-\sigma_j\alpha_j)}{s-\sigma_j\alpha_j}.
\end{align}
If $(m_1,m_2,m_3)\in A_1$, then 
\begin{align}
    Q(\cdot) &\coloneqq d^{-2s}\prod_{j\in\{2,3\}}\gamma_{\rat}(\delta,t,s,\alpha_j)\frac{G(s-\sigma_j\alpha_j)}{s-\sigma_j\alpha_j}\\
    &\quad\,\, \cdot \frac{1}{2\pi i}\int_{(\frac{1}{8})}p^{au}\gamma_{\rat}(\delta,t,u,\alpha_1)\frac{G(u-\sigma_1\alpha_1)}{u-\sigma_1\alpha_1}\frac{\Gamma(s-u)}{(2\pi i\frac{m_1}{p^{a-b}c'})^{s-u}}\,du.
\end{align}
If $(m_1,m_2,m_3)\in A_2$ and $c'=1$, then 
\begin{align}
    Q(\cdot) &\coloneqq d^{-s}\prod_{j\in\{1,3\}}\gamma_{\rat}(\delta,t,s,\alpha_j)\frac{G(s-\sigma_j\alpha_j)}{s-\sigma_j\alpha_j} \\
    &\quad\,\, \cdot \frac{1}{2\pi i}\int_{(\frac{1}{8})}d^{-u}p^{au}\gamma_{\rat}(\delta,t,u,\alpha_2)\frac{G(u-\sigma_2\alpha_2)}{u-\sigma_2\alpha_2}\frac{\Gamma(s-u)}{(2\pi i\frac{m_2}{p^{a-b}})^{s-u}}\,du.
\end{align}
If $(m_1,m_2,m_3)\in P_{12}$, then 
\begin{align}
    Q(\cdot) &\coloneqq d^{-s}p^{-as}\gamma_{\rat}(\delta,t,s,\alpha_3)\frac{G(s-\sigma_3\alpha_3)}{s-\sigma_3\alpha_3} \\
    &\quad\,\, \cdot \frac{1}{2\pi i}\int_{(\frac{1}{8})}p^{au}\gamma_{\rat}(\delta,t,u,\alpha_1)\frac{G(u-\sigma_1\alpha_1)}{u-\sigma_1\alpha_1}\frac{\Gamma(s-u)}{(2\pi i\frac{m_1}{p^{a-b}c'})^{s-u}}\,du \\
    &\quad\,\, \cdot \frac{1}{2\pi i}\int_{(\frac{1}{8})}d^{-v}p^{av}\gamma_{\rat}(\delta,t,v,\alpha_2)\frac{G(v-\sigma_2\alpha_2)}{v-\sigma_2\alpha_2}\frac{\Gamma(s-v)}{(2\pi i\frac{m_2}{p^{a-b}c'})^{s-v}}\,dv.
\end{align}
If $(m_1,m_2,m_3)\in P_{23}$ and $c'=1$, then 
\begin{align}
    Q(\cdot) &\coloneqq p^{-as}\gamma_{\rat}(\delta,t,s,\alpha_1)\frac{G(s-\sigma_1\alpha_1)}{s-\sigma_1\alpha_1} \\
    &\quad\,\, \cdot \frac{1}{2\pi i}\int_{(\frac{1}{8})}d^{-u}p^{au}\gamma_{\rat}(\delta,t,u,\alpha_2)\frac{G(u-\sigma_2\alpha_2)}{u-\sigma_2\alpha_2}\frac{\Gamma(s-u)}{(2\pi i\frac{m_2}{p^{a-b}})^{s-u}}\,du \\
    &\quad\,\, \cdot \frac{1}{2\pi i}\int_{(\frac{1}{8})}d^{-v}p^{av}\gamma_{\rat}(\delta,t,v,\alpha_3)\frac{G(v-\sigma_3\alpha_3)}{v-\sigma_3\alpha_3}\frac{\Gamma(s-v)}{(2\pi i\frac{m_3}{p^{a-b}})^{s-v}}\,dv.
\end{align}

The remaining formulas with $(m_1,m_2,m_3)\in A_3\sqcup P_{13}$ are obtained by symmetry. The above identities are quite lenghty, but keeping such detailed track of parameters turns out to be crucial when carefully proving the cancellation result Proposition \ref{prop:cancellation}. The most important features to keep track of are the following.
\begin{enumerate}
    \item[-] we have $\alpha_j\ll\frac{1}{\log p^a}$ in practice,
    \item[-] the function $u\mapsto\gamma_{\rat}(\cdot,u,\cdot)$ is holomorphic to the right of the $(-1/2-\delta)$-line, on which it has one double pole or two simple poles, 
    \item[-] the function $u\mapsto \gamma_{\rat}(\delta,t,u,v)\frac{G(u-\sigma_j\alpha_j)}{u-\sigma_j\alpha_j}$ is holomorphic to the right of the simple pole $u=\sigma_j\alpha_j$, and exponentially decays in terms of $\abs{\im u}$ as $\im u\rightarrow\pm\infty$ to the right of $u=-1/2-\delta$, 
    \item[-] the $u$- and $v$-integrals in the definition of $Q(\cdot)$ are absolutely convergent because of the aforementioned properties of their integrands.
\end{enumerate}

\subsection{$K^{b,-}_0(\cdot)$ with $m_1m_2m_3=0$}\label{sec:Kbm_mellin_inverted}
We now compute an integral representation for $K^{b,-}_0(\cdot)$ in terms of $\Gamma(\cdot)$. By \cite[(3.4.18)]{PaK}, we have
\begin{align*}
    x^{-1}K_\nu(x) &= \frac{1}{2\pi i}\int_{(\beta)}\Gamma(-s)\Gamma(-s-\nu)2^{-1-2s-\nu}x^{-1+2s+\nu}ds \\
    &= \frac{1}{2\pi i}\int_{(\frac{1+2\beta+\re\nu}{2})}\Gamma\bigg(\frac{1-2s+\nu}{2}\bigg)\Gamma\bigg(\frac{1-2s-\nu}{2}\bigg)2^{-2s}x^{2(s-1)}ds
\end{align*}
for $\beta<\min\{0,\re\nu\}$. In particular, we get 
\begin{align*}
    &\frac{c}{4\pi\sqrt{t_1t_2t_3}}K_{2it}\bigg(\frac{4\pi\sqrt{t_1t_2t_3}}{c}\bigg) \\
    &= \frac{1}{2\pi i}\int_{(\beta)}\Gamma(1/2-s+it)\Gamma(1/2-s-it)2^{-2s}\left(\frac{4\pi}{c}\right)^{2(s-1)}(t_1t_2t_3)^{s-1}ds
\end{align*}
for $\beta\in\R_{<\frac{1}{2}}$. 

Recalling the definition of $K_0^{b,-}(\cdot)$ from \eqref{eq:Kb0pm_def}, we now have the Mellin inversion formula
\begin{align*}
    &\sum_{\substack{N_1,N_2,N_3 \\ \text{dyadic}}}\frac{1}{\sqrt{N_1N_2N_3}}K^{b,-}_0(m_1,m_2,m_3,c,d,\delta,\pmb{\alpha},\pmb{\sigma}) \\
    &= \frac{c}{4\pi^2}w_C(c)\int_{\R}\sinh(\pi t)th_0(t)\Gamma_{\rat}(\delta,t,\pmb{\alpha},\pmb{\sigma})\bigg(\frac{1}{2\pi i}\int_{(\frac{1}{4})}\Gamma(1/2-s+it)\Gamma(1/2-s-it)\left(\frac{2\pi}{c}\right)^{2s} \\
    &\quad\, \cdot\bigg(\int_{\R_{>0}}V_{1/2+\sigma_1\alpha_1}\bigg(\frac{t_1}{p^a},t,\delta\bigg)e_{p^{a-b}c'}(-m_1t_1)t_1^s\frac{dt_1}{t_1} \\
    &\quad\quad\,\,\,\, \cdot\prod_{j\in\{2,3\}}\int_{\R_{>0}}V_{1/2+\sigma_j\alpha_j}\bigg(\frac{dt_j}{p^{a}},t,\delta\bigg)e_{p^{a-b}c'}(-m_jt_j)t_j^s\frac{dt_j}{t_j}\bigg)ds\bigg)dt.
\end{align*}
Specifying to the case $m_1m_2m_3=0$, we obtain
\begin{align}
    &\sum_{\substack{N_1,N_2,N_3 \\ \text{dyadic}}}\frac{1}{\sqrt{N_1N_2N_3}}K_0^{b,-}(m_1,m_2,m_3,p^ac',\cdot) \\
    &= \frac{p^ac'}{4\pi^2}w_C(p^ac')\int_{\R}\sinh(\pi t)th_0(t) \\
    &\quad\quad\quad\quad\quad\quad\quad\quad\,\, \cdot\bigg(\frac{1}{2\pi i}\int_{(\frac{1}{4})}\Gamma(1/2-s+it)\Gamma(1/2-s-it)\left(\frac{2\pi}{c'}\right)^{2s}Q(\cdot)\,ds\bigg)\,dt, \label{eq:Kbm0_general}
\end{align}
where $Q(\cdot)$ has exactly the same definition as in Section \ref{sec:Kbp_mellin_inverted}.
The remaining formulas with $(m_1,m_2,$ $m_3)\in A_3\sqcup P_{13}$ are obtained by symmetry.

\subsection{Combining all info}\label{sec:combining_all_info_offdiag}

We now obtain the analogs to \cite[Lemmas 1-3]{Pe}:

\begin{lemma}[Matching the $(0,0,0)$-terms]\label{lemma:all0}
If $T=(p^a)^{o(1)}$, and $\re\alpha_j\ll \frac{1}{\log p^a}$ and $\im\alpha_j=O(1)$ for all $1\leq j\leq 3$, we have
\begin{align*}
    &\sum_{\pmb{\sigma}\in\{\pm1\}^3}(\pm1)^{\frac{1-\sigma}{2}}\sum_{\,\,\,\pm_\lambda}(\pm_\lambda1)^{\frac{1\mp1}{2}}\sum_{\substack{d\geq1 \\ (d,p)=1}}\frac{1}{d}\sum_{\substack{N_1,N_2,N_3,C \\ \text{dyadic}}}\frac{1}{C\sqrt{N_1N_2N_3}}\sum_{\substack{c'\geq1 \\ (c',p)=1}}G_\chi^{b,\pm_\lambda}(0,0,0;p^{a-b}c') \\
    &\quad\quad\quad\quad\quad\quad\quad\quad\quad\quad\quad\quad\quad\quad\quad\quad\quad\quad\quad\quad\,\, \cdot\sum_{\,\,\,\pm_b}(\mp_\lambda1)^{\frac{1\mp_b1}{2}}K^{b,\pm_\lambda}_0(0,0,0,\cdot,\tfrac{1\mp\chi(-1)(\pm_b1)}{2},\cdot) \\
    &= -\sum_{\pmb{\sigma}\in\{\pm1\}^3}(\pm1)^{\frac{1-\sigma}{2}}\frac{1}{4\pi}\int_{\R}\tanh(\pi t)th_0(t) \\
    &\quad\quad\quad\quad\quad\quad\quad\quad\quad\quad\quad\,\,\, \cdot\sum_{\delta\in\{0,1\}}\bigg(\frac{L(\sigma_1\alpha_1,-\sigma_1\alpha_1,\sigma_1\alpha_1,\delta,t,(\alpha_2,\alpha_3,\alpha_1),(\sigma_2,\sigma_3,\sigma_1))}{(\sigma_1\alpha_1-\sigma_2\alpha_2)(-\sigma_1\alpha_1-\sigma_3\alpha_3)(2\sigma_1\alpha_1)} \\
    &\quad\quad\quad\quad\quad\quad\quad\quad\quad\quad\quad\quad\quad\quad\quad\,\,\,\, + \frac{L(\sigma_2\alpha_2,-\sigma_2\alpha_2,\sigma_2\alpha_2,\delta,t,(\alpha_1,\alpha_3,\alpha_2),(\sigma_1,\sigma_3,\sigma_2))}{(\sigma_2\alpha_2-\sigma_1\alpha_1)(-\sigma_2\alpha_2-\sigma_3\alpha_3)(2\sigma_2\alpha_2)} \\
    &\quad\quad\quad\quad\quad\quad\quad\quad\quad\quad\quad\quad\quad\quad\quad\,\,\,\, + \frac{L(\sigma_3\alpha_3,-\sigma_3\alpha_3,\sigma_3\alpha_3,\delta,t,\pmb{\alpha},\pmb{\sigma})}{(\sigma_3\alpha_3-\sigma_1\alpha_1)(-\sigma_3\alpha_3-\sigma_2\alpha_2)(2\sigma_3\alpha_3)} \\
    &\quad\quad\quad\quad\quad\quad\quad\quad\quad\quad\quad\quad\quad\quad\quad\,\,\,\, + \frac{L(0,0,0,\delta,t,\pmb{\alpha},\pmb{\sigma})}{2(-\sigma_1\alpha_1)(-\sigma_2\alpha_2)(-\sigma_3\alpha_3)}\bigg)\,dt + O_\varepsilon(p^{a(-\frac{1}{3}+\varepsilon)}).
\end{align*}
\end{lemma}

\begin{proof}
    By \eqref{eq:gauss_sum_all0}, the left-hand side in the Lemma statement becomes 
    \begin{align} 
        &(1-p^{-1})^2\sum_{\pmb{\sigma}\in\{\pm1\}^3}(\pm1)^{\frac{1-\sigma}{2}}\sum_{\,\,\,\pm_\lambda}(\pm_\lambda1)^{\frac{1\mp1}{2}}\chi(\mp_\lambda1)\sum_{\substack{d\geq1 \\ (d,p)=1}}\frac{1}{d}\sum_{\substack{N_1,N_2,N_3,C \\ \text{dyadic}}}\frac{1}{C\sqrt{N_1N_2N_3}} \label{eq:LHS_lemma_all0} \\
        &\quad\quad\quad\quad\quad\quad\quad\quad\quad\quad\quad\quad\quad\quad\quad \cdot\sum_{\,\,\,\pm_b}(\mp_\lambda1)^{\frac{1\mp_b1}{2}}K^{\pm_\lambda}_0(0,0,0,\cdot,\tfrac{1\mp\chi(-1)(\pm_b1)}{2},\cdot). 
    \end{align}
    For the $(+_\lambda)$-terms, we obtain by \eqref{eq:Kbp0_general} that
    \begin{align}
        &\sum_{\substack{d\geq1 \\ (d,p)=1}}\frac{1}{d}\sum_{\substack{N_1,N_2,N_3,C \\ \text{dyadic}}}\frac{1}{C\sqrt{N_1N_2N_3}}K_0^{b,+}(0,0,0,p^a,\cdot,\delta,\cdot) \\
        &= \frac{i}{4\pi(1-p^{-1})^2}\int_\R\frac{th_0(t)}{\cosh(\pi t)}\bigg(\frac{1}{2\pi i}\int_{(\frac{1}{4})}\frac{\Gamma(1/2-s+it)}{\Gamma(1/2+s+it)}(2\pi)^{2s}\frac{L(s,-s,s,\delta,t,\pmb{\alpha},(\sigma_1,-\sigma_2,\sigma_3))}{2s(s-\sigma_1\alpha_1)(s-\sigma_2\alpha_2)(s-\sigma_3\alpha_3)} \\
        &\quad\quad\quad\quad\quad\quad\quad\quad\quad\quad\quad\quad\quad\quad\quad\quad\quad \cdot \gamma_{\rat}(\delta,t,s,-s)\,ds\bigg)dt.  \label{eq:t_integral_to_study}
    \end{align}
    The $t$-integrand is ``almost odd'' in $t$, in the sense that the factor
    \begin{align*}
        \frac{\Gamma(1/2-s+it)}{\Gamma(1/2+s+it)}
    \end{align*}
    ``prevents'' the integrand from being odd. As the integral of an odd function over $\R$ vanishes, the equality \eqref{eq:t_integral_to_study} still holds if we replace
    \begin{align*}
        \frac{\Gamma(1/2-s+it)}{\Gamma(1/2+s+it)} \quad \text{ by } \quad \frac{\Gamma(1/2-s+it)}{\Gamma(1/2+s+it)}+\text{(any even function)},
    \end{align*}
    for example we can replace the fraction by its odd part 
    \begin{align*}
        \frac{1}{2}\bigg(\frac{\Gamma(1/2-s+it)}{\Gamma(1/2+s+it)}-\frac{\Gamma(1/2-s-it)}{\Gamma(1/2+s-it)}\bigg).
    \end{align*}
    Now note that
    \begin{align*}
        &\frac{1}{2}\bigg(\frac{\Gamma(1/2-s+it)}{\Gamma(1/2+s+it)}-\frac{\Gamma(1/2-s-it)}{\Gamma(1/2+s-it)}\bigg)\gamma_{\rat}(\delta,t,s,-s) \\
        &= \frac{1}{2(2\pi)^{2s}}\bigg(\frac{\Gamma_\C(1/2-s+it)}{\Gamma_\C(1/2+s+it)}-\frac{\Gamma_\C(1/2-s-it)}{\Gamma_\C(1/2+s-it)}\bigg)\gamma_{\rat}(\delta,t,s,-s) \\
        &= \frac{1}{2(2\pi)^{2s}}\bigg(\frac{\Gamma_\R(3/2-\delta-s+it)\Gamma_\R(1/2+\delta+s-it)}{\Gamma_\R(3/2-\delta+s+it)\Gamma_\R(1/2+\delta-s-it)}-\frac{\Gamma_\R(3/2-\delta-s-it)\Gamma_\R(1/2+\delta+s+it)}{\Gamma_\R(3/2-\delta+s-it)\Gamma_\R(1/2+\delta-s+it)}\bigg) \\
        &= \frac{1}{2(2\pi)^{2s}}\bigg(\frac{\cos(\pi(1/2-\delta+s+it)/2)}{\cos(\pi(1/2-\delta-s+it)/2)}-\frac{\cos(\pi(1/2-\delta+s-it)/2)}{\cos(\pi(1/2-\delta-s-it)/2)}\bigg) \\
        &= \frac{1}{2(2\pi)^{2s}}\frac{\cos(\pi(s+it))-\cos(\pi(s-it))}{\cos(\pi(1/2-\delta-s))+\cos(\pi it)} \\
        &= \frac{1}{(2\pi)^{2s}}\frac{-i\sin(\pi s)\sinh(\pi t)}{(-1)^\delta\sin(\pi s)+\cosh(\pi t)},
    \end{align*}
    where in the second step we applied Legendre's duplication formula, and in the third step Euler's reflection formula. Shifting the vertical line $(1/4)$ in the inner integral in \eqref{eq:t_integral_to_study} to the left to $(-1/3)$, we pass simple poles at $s\in\{\sigma_1\alpha_1,\sigma_2\alpha_2,\sigma_3\alpha_3\}$ (the pole at $s=0$ cancels out with the zero of $\sin(\pi s)$ at the same point). So, regarding the $(+_\lambda)$-terms in \eqref{eq:LHS_lemma_all0}, we end up with 
    \begin{align*}
        &\sum_{\substack{d\geq1 \\ (d,p)=1}}\frac{1}{d}\sum_{\substack{N_1,N_2,N_3,C \\ \text{dyadic}}}\frac{1}{C\sqrt{N_1N_2N_3}}K_0^{b,+}(0,0,0,p^a,\cdot,\delta,\cdot) \\
        &= \frac{1}{4\pi(1-p^{-1})^2}\int_\R\tanh(\pi t)th_0(t) \\
        &\quad\quad\quad\quad\quad\,\,\, \cdot\bigg(\frac{\sin(\pi \sigma_1\alpha_1)}{(-1)^{\delta}\sin(\pi \sigma_1\alpha_1)+\cosh(\pi t)}\frac{L(\sigma_1\alpha_1,-\sigma_1\alpha_1,\sigma_1\alpha_1,\delta,t,(\alpha_2,\alpha_3,\alpha_1),(\sigma_2,-\sigma_3,\sigma_1))}{2\sigma_1\alpha_1(\sigma_1\alpha_1-\sigma_2\sigma_2)(\sigma_1\alpha_1-\sigma_3\alpha_3)} \\
        &\quad\quad\quad\quad\quad\quad\quad + \frac{\sin(\pi\sigma_2\alpha_2)}{(-1)^{\delta}\sin(\pi\sigma_2\alpha_2)+\cosh(\pi t)}\frac{L(\sigma_2\alpha_2,-\sigma_2\alpha_2,\sigma_2\alpha_2,\delta,t,(\alpha_1,\alpha_3,\alpha_2),(\sigma_1,-\sigma_3,\sigma_2))}{2\sigma_2\alpha_2(\sigma_2\alpha_2-\sigma_1\alpha_1)(\sigma_2\alpha_2-\sigma_3\alpha_3)} \\
        &\quad\quad\quad\quad\quad\quad\quad + \frac{\sin(\pi\sigma_3\alpha_3)}{(-1)^{\delta}\sin(\pi\sigma_3\alpha_3)+\cosh(\pi t)}\frac{L(\sigma_3\alpha_3,-\sigma_3\alpha_3,\sigma_3\alpha_3,\delta,t,\pmb{\alpha},(\sigma_1,-\sigma_2,\sigma_3))}{2\sigma_3\alpha_3(\sigma_3\alpha_3-\sigma_1\alpha_1)(\sigma_3\alpha_3-\sigma_2\alpha_2)}\bigg)\,dt \\
        &\quad\quad\quad\quad\quad\quad\quad\quad\quad\quad\quad\quad\quad\quad\quad\quad\quad\quad\quad\quad\quad\quad\quad\quad\quad\quad\quad\quad\quad\quad\quad\quad\quad\quad\quad\quad + O_\varepsilon(p^{a(-\frac{1}{3}+\varepsilon)}). 
    \end{align*}
    
    Regarding the $(-_\lambda)$-terms, we have by \eqref{eq:Kbm0_general} that
    \begin{align}
        &\sum_{\substack{d\geq1 \\ (d,p)=1}}\frac{1}{d}\sum_{\substack{N_1,N_2,N_3,C \\ \text{dyadic}}}\frac{1}{C\sqrt{N_1N_2N_3}}K^{b,-}_0(0,0,0,p^a,\cdot,\delta,\cdot) \\
        &= \frac{1}{4\pi^2(1-p^{-1})^2}\int_{\R}\sinh(\pi t)th_0(t)\bigg(\frac{1}{2\pi i}\int_{(\frac{1}{4})}\Gamma(1/2-s+it)\Gamma(1/2-s-it)(2\pi)^{2s} \\
        &\quad\quad\quad\quad\quad\quad\quad\quad\quad \cdot\frac{L(s,-s,s,\delta,t,\pmb{\alpha},(\sigma_1,-\sigma_2,\sigma_3))}{2s(s-\sigma_1\alpha_1)(s-\sigma_2\alpha_2)(s-\sigma_3\alpha_3)}\gamma_{\rat}(\delta,t,s,-s)\,ds\bigg)dt. \label{eq:t_integral_to_study_minus}
    \end{align}
    Now note that
    \begin{align*}
        &\Gamma(1/2-s+it)\Gamma(1/2-s-it)\gamma_{\rat}(\delta,t,s,-s) \\
        &= \frac{\Gamma_\C(1/2-s-it)\Gamma_\C(1/2-s+it)}{2^2(2\pi)^{-(1-2s)}}\gamma_{\rat}(\delta,t,s,-s) \\
        &= \frac{\Gamma_\R(3/2-\delta-s-it)\Gamma_\R(3/2-\delta-s+it)\Gamma_\R(1/2+\delta+s+it)\Gamma_\R(1/2+\delta+s-it)}{2^2(2\pi)^{-(1-2s)}} \\
        &= \frac{1}{2^2(2\pi)^{-(1-2s)}\cos(\pi(1/2-\delta-s-it)/2)\cos(\pi(1/2-\delta-s+it)/2)} \\
        &= \frac{1}{2(2\pi)^{-(1-2s)}(\cos(\pi(1/2-\delta-s))+\cos(\pi it))} \\
        &= \frac{\pi}{(2\pi)^{2s}((-1)^\delta\sin(\pi s)+\cosh(\pi t))},
    \end{align*}
    where again in the second step we applied Legendre's duplication formula, and in the third step Euler's reflection formula. Shifting the vertical line $(1/4)$ in the inner integral in \eqref{eq:t_integral_to_study_minus} to the left to $(-1/3)$, we pass simple poles at $s\in\{0,\sigma_1\alpha_1,\sigma_2\alpha_2,\sigma_3\alpha_3\}$. So, regarding the $(-_\lambda)$-terms in \eqref{eq:LHS_lemma_all0}, we end up with
    \begin{align*}
        &\sum_{\substack{d\geq1 \\ (d,p)=1}}\frac{1}{d}\sum_{\substack{N_1,N_2,N_3,C \\ \text{dyadic}}}\frac{1}{C\sqrt{N_1N_2N_3}}K^{b,-}_0(0,0,0,p^a,\cdot,\delta,\cdot) \\
        &= \frac{1}{4\pi(1-p^{-1})^2}\int_{\R}\tanh(\pi t)th_0(t) \\
        &\quad\quad\quad\quad\quad\,\,\, \cdot\bigg(\frac{\cosh(\pi t)}{(-1)^{\delta}\sin(\pi\sigma_1\alpha_1)+\cosh(\pi t)}\frac{L(\sigma_1\alpha_1,-\sigma_1\alpha_1,\sigma_1\alpha_1,\delta,t,(\alpha_2,\alpha_3,\alpha_1),(\sigma_2,-\sigma_3,\sigma_1))}{2\sigma_1\alpha_1(\sigma_1\alpha_1-\sigma_2\alpha_2)(\sigma_1\alpha_1-\sigma_3\alpha_3)} \\
        &\quad\quad\quad\quad\quad\quad\quad + \frac{\cosh(\pi t)}{(-1)^{\delta}\sin(\pi\sigma_2\alpha_2)+\cosh(\pi t)}\frac{L(\sigma_2\alpha_2,-\sigma_2\alpha_2,\sigma_2\alpha_2,\delta,t,(\alpha_1,\alpha_3,\alpha_2),(\sigma_1,-\sigma_3,\sigma_2))}{2\sigma_2\alpha_2(\sigma_2\alpha_2-\sigma_1\alpha_1)(\sigma_2\alpha_2-\sigma_3\alpha_3)} \\
        &\quad\quad\quad\quad\quad\quad\quad + \frac{\cosh(\pi t)}{(-1)^{\delta}\sin(\pi\sigma_3\alpha_3)+\cosh(\pi t)}\frac{L(\sigma_3\alpha_3,-\sigma_3\alpha_3,\sigma_3\alpha_3,\delta,t,\pmb{\alpha},(\sigma_1,-\sigma_2,\sigma_3))}{2\sigma_3\alpha_3(\sigma_3\alpha_3-\sigma_1\alpha_1)(\sigma_3\alpha_3-\sigma_2\alpha_2)} \\
        &\quad\quad\quad\quad\quad\quad\quad + \frac{L(0,0,0,\delta,t,\pmb{\alpha},(\sigma_1,-\sigma_2,\sigma_3))}{2(-\sigma_1\alpha_1)(-\sigma_2\alpha_2)(-\sigma_3\alpha_3)}\bigg)\,dt + O_\varepsilon(p^{a(-\frac{1}{3}+\varepsilon)}).
    \end{align*} 

    Combining everything, one obtains after careful matching\footnote{using identities such as $\pm\chi(-1)=(-1)^{\frac{1\mp\chi(-1)}{2}}$ and $-(-1)^{\frac{1-\sigma}{2}}=(-1)^{\frac{1+\sigma}{2}}$} that \eqref{eq:LHS_lemma_all0} equals
    \begin{align*}
        &\sum_{\pmb{\sigma}\in\{\pm1\}^3}(\pm1)^{\frac{1-(-\sigma)}{2}}\frac{1}{4\pi}\int_\R\tanh(\pi t)th_0(t) \\
        &\quad\quad\quad\quad\quad\quad\quad\quad\quad\quad\,\,\,\, \cdot \sum_{\,\,\,\pm_b}\bigg(\frac{L(\sigma_1\alpha_1,-\sigma_1\alpha_1,\sigma_1\alpha_1,\tfrac{1\mp\chi(-1)(\pm_b1)}{2},t,(\alpha_2,\alpha_3,\alpha_1),(\sigma_2,-\sigma_3,\sigma_1))}{2\sigma_1\alpha_1(\sigma_1\alpha_1-\sigma_2\sigma_2)(\sigma_1\alpha_1-\sigma_3\alpha_3)} \\
        &\quad\quad\quad\quad\quad\quad\quad\quad\quad\quad\quad\quad\quad\,\,\,\, + \frac{L(\sigma_2\alpha_2,-\sigma_2\alpha_2,\sigma_2\alpha_2,\tfrac{1\mp\chi(-1)(\pm_b1)}{2},t,(\alpha_1,\alpha_3,\alpha_2),(\sigma_1,-\sigma_3,\sigma_2))}{2\sigma_2\alpha_2(\sigma_2\alpha_2-\sigma_1\alpha_1)(\sigma_2\alpha_2-\sigma_3\alpha_3)} \\
        &\quad\quad\quad\quad\quad\quad\quad\quad\quad\quad\quad\quad\quad\,\,\,\, + \frac{L(\sigma_3\alpha_3,-\sigma_3\alpha_3,\sigma_3\alpha_3,\tfrac{1\mp\chi(-1)(\pm_b1)}{2},t,\pmb{\alpha},(\sigma_1,-\sigma_2,\sigma_3))}{2\sigma_3\alpha_3(\sigma_3\alpha_3-\sigma_1\alpha_1)(\sigma_3\alpha_3-\sigma_2\alpha_2)} \\
        &\quad\quad\quad\quad\quad\quad\quad\quad\quad\quad\quad\quad\quad\,\,\,\, + \frac{L(0,0,0,\tfrac{1\mp\chi(-1)(\pm_b1)}{2},t,\pmb{\alpha},\pmb{\sigma})}{2(-\sigma_1\alpha_1)(-\sigma_2\alpha_2)(-\sigma_3\alpha_3)}\bigg)\,dt + O_\varepsilon(p^{a(-\frac{1}{3}+\varepsilon)}).
    \end{align*}
    Rearranging the permutations $\pmb{\sigma}$ in the latter expression, one obtains the right-hand side in the Lemma statement.
\end{proof}

\begin{lemma}[Matching the $A_1$-terms]\label{lemma:m1nonzero}
If $T=(p^a)^{o(1)}$, and $\re\alpha_j\ll \frac{1}{\log p^a}$ and $\im\alpha_j=O(1)$ for all $1\leq j\leq 3$, we have
\begin{align*}
    &\sum_{\pmb{\sigma}\in\{\pm1\}^3}(\pm1)^{\frac{1-\sigma}{2}}\sum_{\,\,\,\pm_\lambda}(\pm_\lambda1)^{\frac{1\mp1}{2}}\sum_{\substack{d\geq1 \\ (d,p)=1}}\frac{1}{d}\sum_{\substack{N_1,N_2,N_3,C \\ \text{dyadic}}}\frac{1}{C\sqrt{N_1N_2N_3}}\sum_{\substack{c'\geq1 \\ (c',p)=1}}\sum_{\substack{m_1\in\Z_{\neq0}}}G_\chi^{b,\pm_\lambda}(m_1,0,0;p^{a-b}c') \\
    &\quad\quad\quad\quad\quad\quad\quad\quad\quad\quad\quad\quad\quad\quad\quad\quad\quad\quad\quad\quad\,\, \cdot\sum_{\,\,\,\pm_b}(\mp_\lambda1)^{\frac{1\mp_b1}{2}}K_0^{b,\pm_\lambda}(m_1,0,0,\cdot,\tfrac{1\mp\chi(-1)(\pm_b1)}{2},\cdot) \\
    &= -\sum_{\pmb{\sigma}\in\{\pm1\}^3}(\pm1)^{\frac{1-\sigma}{2}}\frac{1}{4\pi}\int_{\R}\tanh(\pi t)th_0(t)\sum_{\delta\in\{0,1\}}M(\sigma_2\alpha_2,\sigma_3\alpha_3,\sigma_1\alpha_1,\delta,t,(\alpha_2,\alpha_3,\alpha_1),(\sigma_2,\sigma_3,\sigma_1))\,dt \\
    &\quad\quad\quad\quad\quad\quad\quad\quad\quad\quad\quad\quad\quad\quad\quad\quad\quad\quad\quad\quad\quad\quad\quad\quad\quad\quad\quad\quad\quad\quad\quad\quad\quad\quad\quad\,\,\, + O_\varepsilon(p^{a(-\frac{1}{3}+\varepsilon)}).
\end{align*}
\end{lemma}

\begin{proof}
    By \eqref{eq:gauss_sum_m1nonzero}, the left-hand side in the Lemma statement becomes
    \begin{align} 
        &\frac{1-p^{-1}}{p^{a-b}}\sum_{\pmb{\sigma}\in\{\pm1\}^3}(\pm1)^{\frac{1-\sigma}{2}}\sum_{\,\,\,\pm_\lambda}(\pm_\lambda1)^{\frac{1\mp1}{2}}\chi(\mp_\lambda1)\sum_{\substack{d\geq1 \\ (d,p)=1}}\frac{1}{d}\sum_{\substack{N_1,N_2,N_3,C \\ \text{dyadic}}}\frac{1}{C\sqrt{N_1N_2N_3}}\sum_{\substack{c'\geq1 \\ (c',p)=1}}\frac{1}{c'}  \\
        &\quad\quad\quad\quad\quad\quad\, \cdot\sum_{\substack{m_1\in\Z_{\neq0} \\ (m_1,c')=1}}R_{p^{a-b}}(m_1)\sum_{\,\,\,\pm_b}(\mp_\lambda1)^{\frac{1\mp_b1}{2}}K^{b,\pm_\lambda}_0(m_1,0,0,p^ac',\cdot,\tfrac{1\mp\chi(-1)(\pm_b1)}{2},\cdot). \label{eq:LHS_lemma_m1nonzero}
    \end{align}
    Due to the expressions \eqref{eq:Kbp0_general} and \eqref{eq:Kbm0_general} for $K_0^{b,\pm}(\cdot)$, we are led to consider the Dirichlet series (recalling that $(c',p)=1$)
    \begin{align*}
        \sum_{\substack{m_1\geq1 \\ (m_1,c')=1}}\frac{R_{p^{a-b}}(m_1)}{m_1^s} &= \sum_{\substack{m_1\geq1 \\ (m_1,c')=1}}\frac{\varphi(p^{a-b})}{(p^{a-b}m_1)^s}-\sum_{\substack{m_1\geq1 \\ (m_1,pc')=1}}\frac{p^{a-b-1}}{(p^{a-b-1}m_1)^s} \\
        &= p^{(a-b)(1-s)}(1-p^{-(1-s)})\zeta^{(c')}(s),
    \end{align*}
    which has abscissa of convergence 0 (as can be verified by summation by parts). In particular, our sums under consideration are of the form
    \begin{align*}
        &\frac{\Gamma(s-u)}{(2\pi i p^{-a+b})^{s-u}}\sum_{\substack{m_1\geq1 \\ (m_1,c')=1}}\frac{R_{p^{a-b}}(m_1)}{m_1^{s-u}} + \frac{\Gamma(s-u)}{(-2\pi ip^{-a+b})^{s-u}}\sum_{\substack{m_1\leq-1 \\ (m_1,c')=1}}\frac{R_{p^{a-b}}(m_1)}{\abs{m_1}^{s-u}} \\
        &= p^{a-b}\zeta^{(p)}(1-s+u)\prod_{\widetilde{p}\,\mid c'}(1-\widetilde{p}^{\,-(s-u)})
    \end{align*}
    for $\re s>\re u$. Here we used the asymmetric functional equation $\zeta(s)=2^s\pi^{s-1}\sin(\pi s/2)\Gamma(1-s)\zeta(1-s)$ and the fact that $\frac{1}{(\pm i)^{s-u}} = e^{\mp i(s-u)\frac{\pi}{2}}$. As for the summation over $c'$, we have
    \begin{align*}
        \sum_{\substack{c'\geq1 \\ (c',p)=1}}\frac{1}{c'^{1+s+u}}\prod_{\widetilde{p}\,\mid c'}(1-\widetilde{p}^{\,-(s-u)}) = \sum_{\substack{c'\geq1 \\ (c',p)=1}}\frac{(\mu*(\cdot)^{s-u})(c')}{c'^{1+2s}} = \frac{\zeta^{(p)}(1+s+u)}{\zeta^{(p)}(1+2s)}.
    \end{align*}

    For the $(+_\lambda)$-terms in \eqref{eq:LHS_lemma_m1nonzero}, we now obtain by \eqref{eq:Kbp0_general} that
    \begin{align*}
        &\sum_{\substack{d\geq1 \\ (d,p)=1}}\frac{1}{d}\sum_{\substack{N_1,N_2,N_3,C \\ \text{dyadic}}}\frac{1}{C\sqrt{N_1N_2N_3}}\sum_{\substack{c'\geq1 \\ (c',p)=1}}\frac{1}{c'}\sum_{\substack{m_1\in\Z_{\neq0} \\ (m_1,c')=1}}R_{p^{a-b}}(m_1)K_0^{b,+}(m_1,0,0,p^ac',\cdot,\delta,\cdot) \\
        &= \frac{ip^{a-b}}{4\pi(1-p^{-1})}\int_\R\frac{th_0(t)}{\cosh(\pi t)}\bigg(\frac{1}{2\pi i}\int_{(\frac{5}{12})}\frac{\Gamma(1/2-s+it)}{\Gamma(1/2+s+it)}\frac{(2\pi)^{2s}}{(s-\sigma_2\alpha_2)(s-\sigma_3\alpha_3)}\gamma_{\rat}(\delta,t,s,-s) \\
        &\quad\quad\quad\quad\quad\quad\quad\quad\quad\quad\quad\quad\quad\quad\quad\quad\quad \cdot\bigg(\frac{1}{2\pi i}\int_{(\frac{1}{8})}\frac{L(u,s,-s,\delta,t,\pmb{\alpha},(\sigma_1,\sigma_2,-\sigma_3))}{(u-\sigma_1\alpha_1)(u+s)(u-s)}\,du\bigg)\,ds\bigg)\,dt.
    \end{align*}
    Shifting the vertical line $(1/8)$ in the latter inner-most integral to the left to $(-1/3)$, we pass a simple pole at $u=\sigma_1\alpha_1$. So, similar to the proof of Lemma \ref{lemma:all0}, we end up with
    \begin{align*}
        &\frac{p^{a-b}}{4\pi(1-p^{-1})}\int_\R\tanh(\pi t)th_0(t) \\
        &\quad\quad\quad\quad\quad\,\, \cdot \bigg(\frac{1}{2\pi i}\int_{(\frac{5}{12})}\frac{\sin(\pi s)}{(-1)^{\delta}\sin(\pi s)+\cosh(\pi t)}\frac{L(\sigma_1\alpha_1,s,-s,\delta,t,\pmb{\alpha},(\sigma_1,\sigma_2,-\sigma_3))}{(\sigma_1\alpha_1+s)(\sigma_1\alpha_1-s)(s-\sigma_2\alpha_2)(s-\sigma_3\alpha_3)}\,ds\bigg)\,dt \\
        &\quad\quad\quad\quad\quad\quad\quad\quad\quad\quad\quad\quad\quad\quad\quad\quad\quad\quad\quad\quad\quad\quad\quad\quad\quad\quad\quad\quad\quad\quad\quad\quad\quad\quad+ O_\varepsilon(p^{a-b+a(-\frac{1}{3}+\varepsilon)}).
    \end{align*}

    Regarding the $(-_\lambda)$-terms in \eqref{eq:LHS_lemma_m1nonzero}, we have by \eqref{eq:Kbm0_general} that 
    \begin{align*}
        &\sum_{\substack{d\geq1 \\ (d,p)=1}}\frac{1}{d}\sum_{\substack{N_1,N_2,N_3,C \\ \text{dyadic}}}\frac{1}{C\sqrt{N_1N_2N_3}}\sum_{\substack{c'\geq1 \\ (c',p)=1}}\frac{1}{c'}\sum_{\substack{m_1\in\Z_{\neq0} \\ (m_1,c')=1}}R_{p^{a-b}}(m_1)K_0^{b,-}(m_1,0,0,p^ac',\cdot,\delta,\cdot) \\
        &= \frac{p^{a-b}}{4\pi^2(1-p^{-1})}\int_{\R}\sinh(\pi t)th_0(t)\bigg(\frac{1}{2\pi i}\int_{(\frac{5}{12})}\Gamma(1/2-s+it)\Gamma(1/2-s-it)\frac{(2\pi)^{2s}}{(s-\sigma_2\alpha_2)(s-\sigma_3\alpha_3)} \\
        &\quad\quad\quad\quad\quad\quad\quad\quad\quad\quad\quad\quad\,\, \cdot \gamma_{\rat}(\delta,t,s,-s)\bigg(\frac{1}{2\pi i}\int_{(\frac{1}{8})}\frac{L(u,s,-s,\delta,t,\pmb{\alpha},(\sigma_1,\sigma_2,-\sigma_3))}{(u-\sigma_1\alpha_1)(u+s)(u-s)}\,du\bigg)\,ds\bigg)\,dt.
    \end{align*}
    Shifting the vertical line $(1/8)$ in the latter inner-most integral to the left to $(-1/3)$, we pass a simple pole at $u=\sigma_1\alpha_1$, so we end up with 
    \begin{align*}
        &\frac{p^{a-b}}{4\pi(1-p^{-1})}\int_\R\tanh(\pi t)th_0(t) \\
        &\quad\quad\quad\quad\,\,\,\, \cdot\bigg(\frac{1}{2\pi i}\int_{(\frac{5}{12})}\frac{\cosh(\pi t)}{(-1)^{\delta}\sin(\pi s)+\cosh(\pi t)}\frac{L(\sigma_1\alpha_1,s,-s,\delta,t,\pmb{\alpha},(\sigma_1,\sigma_2,-\sigma_3))}{(\sigma_1\alpha_1+s)(\sigma_1\alpha_1-s)(s-\sigma_2\alpha_2)(s-\sigma_3\alpha_3)}\,ds\bigg)\,dt \\
        &\quad\quad\quad\quad\quad\quad\quad\quad\quad\quad\quad\quad\quad\quad\quad\quad\quad\quad\quad\quad\quad\quad\quad\quad\quad\quad\quad\quad\quad\quad\quad\quad\quad\,\,\, + O_\varepsilon(p^{a-b+a(-\frac{1}{3}+\varepsilon)}).
    \end{align*}

    Combining everything, one obtains after careful matching that \eqref{eq:LHS_lemma_m1nonzero} equals
    \begin{align*}
        &\sum_{\pmb{\sigma}\in\{\pm1\}^3}(\pm1)^{\frac{1-(-\sigma)}{2}}\frac{1}{4\pi}\int_\R\tanh(\pi t)th_0(t) \\
        &\quad\quad\quad\quad\quad\quad\quad\quad\quad\quad\,\,\,\,\cdot \bigg(\frac{1}{2\pi i}\int_{(\frac{5}{12})}\sum_{\pm_b}\frac{L(\sigma_1\alpha_1,s,-s,\tfrac{1\mp\chi(-1)(\pm_b1)}{2},t,\pmb{\alpha},(\sigma_1,\sigma_2,-\sigma_3))}{(\sigma_1\alpha_1+s)(\sigma_1\alpha_1-s)(s-\sigma_2\alpha_2)(s-\sigma_3\alpha_3)}\,ds\bigg)\,dt \\
        &\quad\quad\quad\quad\quad\quad\quad\quad\quad\quad\quad\quad\quad\quad\quad\quad\quad\quad\quad\quad\quad\quad\quad\quad\quad\quad\quad\quad\quad\quad\quad\quad\quad\quad+ O_\varepsilon(p^{a(-\frac{1}{3}+\varepsilon)}).
    \end{align*}
    Rearranging the permutations $\pmb{\sigma}$ in the latter expression, one obtains the right-hand side in the Lemma statement.
\end{proof}

\begin{lemma}[Matching the $A_2$-terms]\label{lemma:m2nonzero} 
If $T=(p^a)^{o(1)}$, and $\re\alpha_j\ll \frac{1}{\log p^a}$ and $\im\alpha_j=O(1)$ for all $1\leq j\leq 3$, we have
\begin{align}
    &\sum_{\pmb{\sigma}\in\{\pm1\}^3}(\pm1)^{\frac{1-\sigma}{2}}\sum_{\,\,\,\pm_\lambda}(\pm_\lambda1)^{\frac{1\mp1}{2}}\sum_{\substack{d\geq1 \\ (d,p)=1}}\frac{1}{d}\sum_{\substack{N_1,N_2,N_3,C \\ \text{dyadic}}}\frac{1}{C\sqrt{N_1N_2N_3}}\sum_{\substack{c'\geq1 \\ (c',p)=1}}\sum_{\substack{m_2\in\Z_{\neq0}}}G_\chi^{b,\pm_\lambda}(0,m_2,0;p^{a-b}c') \label{eq:m2nonzero} \\
    &\quad\quad\quad\quad\quad\quad\quad\quad\quad\quad\quad\quad\quad\quad\quad\quad\quad\quad\quad\quad\,\, \cdot \sum_{\,\,\,\pm_b}(\mp_\lambda1)^{\frac{1\mp_b1}{2}}K_0^{b,\pm_\lambda}(0,m_2,0,\cdot,\tfrac{1\mp\chi(-1)(\pm_b1)}{2},\cdot) \\
    &= -\sum_{\pmb{\sigma}\in\{\pm1\}^3}(\pm1)^{\frac{1-\sigma}{2}}\frac{1}{4\pi}\int_{\R}\tanh(\pi t)th_0(t)\sum_{\delta\in\{0,1\}}M(\sigma_1\alpha_1,\sigma_3\alpha_3,\sigma_2\alpha_2,\delta,t,(\alpha_1,\alpha_3,\alpha_2),(\sigma_1,\sigma_3,\sigma_2))\,dt \\
    &\quad\quad\quad\quad\quad\quad\quad\quad\quad\quad\quad\quad\quad\quad\quad\quad\quad\quad\quad\quad\quad\quad\quad\quad\quad\quad\quad\quad\quad\quad\quad\quad\quad\quad\quad + O_\varepsilon(p^{a(-\frac{1}{3}+\varepsilon)}).
\end{align}
\end{lemma}

\begin{proof}
    By \eqref{eq:gauss_sum_m2nonzero}, the left-hand side in the Lemma statement becomes 
    \begin{align}
        &\frac{1-p^{-1}}{p^{a-b}}\sum_{\pmb{\sigma}\in\{\pm1\}^3}(\pm1)^{\frac{1-\sigma}{2}}\sum_{\,\,\,\pm_\lambda}(\pm_\lambda1)^{\frac{1\mp1}{2}}\chi(\mp_\lambda1)\sum_{\substack{d\geq1 \\ (d,p)=1}}\frac{1}{d}\sum_{\substack{N_1,N_2,N_3,C \\ \text{dyadic}}}\frac{1}{C\sqrt{N_1N_2N_3}} \\
        &\quad\quad\quad\quad\quad\quad\quad \cdot\sum_{\substack{m_2\in\Z_{\neq0}}}R_{p^{a-b}}(m_2)\sum_{\,\,\,\pm_b}(\mp_\lambda1)^{\frac{1\mp_b1}{2}}K_0^{b,\pm_\lambda}(0,m_2,0,p^a,\cdot,\tfrac{1\mp\chi(-1)(\pm_b1)}{2},\cdot). \label{eq:LHS_lemma_m2nonzero}
    \end{align}

    Due to the expressions \eqref{eq:Kbp0_general} and \eqref{eq:Kbm0_general} for $K^{b,\pm}_0(\cdot)$, we are led to consider the Dirichlet series
    \begin{align*}
        \sum_{m_2\geq1}\frac{R_{p^{a-b}}(m_2)}{m_2^s} = \sum_{m_2\geq1}\frac{\varphi(p^{a-b})}{(p^{a-b}m_2)^s} - \sum_{\substack{m_2\geq1 \\ (m_2,p)=1}}\frac{p^{a-b-1}}{(p^{a-b-1}m_2)^s} = p^{(a-b)(1-s)}(1-p^{-(1-s)})\zeta(s),
    \end{align*}
    which has abscissa of convergence $0$. In particular, our four sums under consideration are of the form
    \begin{align*}
        \frac{\Gamma(s-u)}{(2\pi ip^{-a+b})^{s-u}}\sum_{m_2\geq1}\frac{R_{p^{a-b}}(m_2)}{m_2^{s-u}} + \frac{\Gamma(s-u)}{(-2\pi ip^{-a+b})^{s-u}}\sum_{m_2\leq-1}\frac{R_{p^{a-b}}(m_2)}{\abs{m_2}^{s-u}} = p^{a-b}\zeta^{(p)}(1-s+u)
    \end{align*}
    with $\re s>\re u$.

    For the $(+_\lambda)$-terms in \eqref{eq:LHS_lemma_m2nonzero}, we now obtain by \eqref{eq:Kbp0_general} that
    \begin{align*}
        &\sum_{\substack{d\geq1 \\ (d,p)=1}}\frac{1}{d}\sum_{\substack{N_1,N_2,N_3,C \\ \text{dyadic}}}\frac{1}{C\sqrt{N_1N_2N_3}}\sum_{m_2\in\Z_{\neq0}}R_{p^{a-b}}(m_2)K_0^{b,+}(0,m_2,0,p^a,\cdot,\delta,\cdot) \\
        &= \frac{ip^{a-b}}{4\pi(1-p^{-1})}\int_{\R}\frac{th_0(t)}{\cosh(\pi t)}\bigg(\frac{1}{2\pi i}\int_{(\frac{5}{12})}\frac{\Gamma(1/2-s+it)}{\Gamma(1/2+s+it)}\frac{(2\pi)^{2s}}{(s-\sigma_1\alpha_1)(s-\sigma_3\alpha_3)}\gamma_{\rat}(\delta,t,s,-s) \\
        &\quad\quad\quad\quad\quad\quad\quad\quad\quad\quad\quad\quad\quad\quad\quad\quad\quad \cdot\bigg(\frac{1}{2\pi i}\int_{(\frac{1}{8})}\frac{L(s,u,-s,\delta,t,\pmb{\alpha},(\sigma_1,\sigma_2,-\sigma_3))}{(u-\sigma_2\alpha_2)(u+s)(u-s)}\,du\bigg)\,ds\bigg)\,dt.
    \end{align*}
    Regarding the $(-_\lambda)$-terms in \eqref{eq:LHS_lemma_m2nonzero}, we have by \eqref{eq:Kbm0_general} that
    \begin{align*}
        &\sum_{\substack{d\geq1 \\ (d,p)=1}}\frac{1}{d}\sum_{\substack{N_1,N_2,N_3,C \\ \text{dyadic}}}\frac{1}{C\sqrt{N_1N_2N_3}}\sum_{m_2\in\Z_{\neq0}}K_0^{b,-}(0,m_2,0,p^a,\cdot,\delta,\cdot)R_{p^{a-b}}(m_2) \\
        &= \frac{p^{a-b}}{4\pi^2(1-p^{-1})}\int_{\R}\sinh(\pi t)th_0(t)\bigg(\frac{1}{2\pi i}\int_{(\frac{5}{12})}\Gamma(1/2-s+it)\Gamma(1/2-s-it)\frac{(2\pi)^{2s}}{(s-\sigma_1\alpha_1)(s-\sigma_3\alpha_3)} \\
        &\quad\quad\quad\quad\quad\quad\quad\quad\quad\quad\quad\quad\quad \cdot\gamma_{\rat}(\delta,t,s,-s)\bigg(\frac{1}{2\pi i}\int_{(\frac{1}{8})}\frac{L(s,u,-s,\delta,t,\pmb{\alpha},(\sigma_1,\sigma_2,-\sigma_3))}{(u-\sigma_2\alpha_2)(u+s)(u-s)}\,du\bigg)\,ds\bigg)\,dt.
    \end{align*}
    Proceeding exactly the same way as in the proof of Lemma \ref{lemma:m1nonzero}, we obtain that \eqref{eq:LHS_lemma_m2nonzero} equals the right-hand side in the Lemma statement.
\end{proof}

\begin{lemma}[Matching the $A_3$-terms]\label{lemma:m3nonzero}
If $T=(p^a)^{o(1)}$, and $\re\alpha_j\ll \frac{1}{\log p^a}$ and $\im\alpha_j=O(1)$ for all $1\leq j\leq 3$, we have
\begin{align*}
    &\sum_{\pmb{\sigma}\in\{\pm1\}^3}(\pm1)^{\frac{1-\sigma}{2}}\sum_{\,\,\,\pm_\lambda}(\pm_\lambda1)^{\frac{1\mp1}{2}}\sum_{\substack{d\geq1 \\ (d,p)=1}}\frac{1}{d}\sum_{\substack{N_1,N_2,N_3,C \\ \text{dyadic}}}\frac{1}{C\sqrt{N_1N_2N_3}}\sum_{\substack{c'\geq1 \\ (c',p)=1}}\sum_{\substack{m_3\in\Z_{\neq0}}}G_\chi^{b,\pm_\lambda}(0,0,m_3;p^{a-b}c') \\
    &\quad\quad\quad\quad\quad\quad\quad\quad\quad\quad\quad\quad\quad\quad\quad\quad\quad\quad\quad\quad\,\,\, \cdot \sum_{\,\,\,\pm_b}(\mp_\lambda1)^{\frac{1\mp_b1}{2}}K_0^{b,\pm_\lambda}(0,0,m_3,p^ac',\cdot,\tfrac{1\mp\chi(-1)(\pm_b1)}{2},\cdot) \\
    &= -\sum_{\pmb{\sigma}\in\{\pm1\}^3}(\pm1)^{\frac{1-\sigma}{2}}\frac{1}{4\pi}\int_{\R}\tanh(\pi t)th_0(t)\sum_{\delta\in\{0,1\}}M(\sigma_1\alpha_1,\sigma_2\alpha_2,\sigma_3\alpha_3,\delta,t,\pmb{\alpha},\pmb{\sigma})\,dt + O_\varepsilon(p^{a(-\frac{1}{3}+\varepsilon)}).
\end{align*}
\end{lemma}

\begin{proof}
    Analogous to the proof of Lemma \ref{lemma:m2nonzero}.
\end{proof}

It remains to treat the terms with $(m_1,m_2,m_3)\in \sqcup_{1\leq i<j\leq3}P_{ij}$. If $(m_1,m_2,m_3)\in P_{12}$, then by \eqref{eq:gauss_sum_m1m2nonzero} and \eqref{eq:Kbm0_general} we have 
\begin{align}
    &\sum_{\substack{d\geq1 \\ (d,p)=1}}\frac{1}{d}\sum_{\substack{N_1,N_2,N_3,C \\ \text{dyadic}}}\frac{1}{C\sqrt{N_1N_2N_3}}\sum_{\substack{c'\geq1 \\ (c',p)=1}}\sum_{\substack{m_1,m_2\in\Z_{\neq0}}}G_\chi^{b,-}(m_1,m_2,0;p^{a-b}c')K_0^{b,-}(m_1,m_2,0,p^ac',\cdot,\delta,\cdot) \\ 
    &=\frac{1}{4\pi^2}\int_{\R}\sinh(\pi t)th_0(t)\bigg(\frac{1}{2\pi i}\int_{(\frac{5}{12})}\Gamma(1/2-s+it)\Gamma(1/2-s-it)(2\pi)^{2s}\gamma_{\rat}(\delta,t,s,-s) \\
    &\quad\quad\quad\quad \cdot \bigg(\bigg(\frac{1}{2\pi i}\bigg)^2\int_{(\frac{1}{24})}\int_{(\frac{1}{24})}\frac{L(-s,u,v,\delta,t,\pmb{\alpha},(-\sigma_1,\sigma_2,\sigma_3))}{(u-s)(v-s)(u+v)(s-\sigma_1\alpha_1)(u-\sigma_2\alpha_2)(v-\sigma_3\alpha_3)}\,dudv\bigg)\,ds\bigg)\,dt  \label{eq:bound_coord_plane} \\
    &\ll_\varepsilon p^{a(-\frac{1}{3}+\varepsilon)}, 
\end{align}
and the same bound holds for the $(+_\lambda)$-case. By \eqref{eq:gauss_sum_m1m3nonzero}, \eqref{eq:gauss_sum_m2m3nonzero}, \eqref{eq:Kbp0_general}, \eqref{eq:Kbm0_general}, the coordinate planes $P_{13}$, $P_{23}$ can be treated analogously.

Combining Lemmas \ref{lemma:all0}--\ref{lemma:m3nonzero} and bounds such as \eqref{eq:bound_coord_plane} for coordinate planes, we finally end up proving Proposition \ref{prop:cancellation}.
\section{Completing the proof of Theorem \ref{thm:main}}\label{sec:completing_the_proof}
We complete the proof of Theorem \ref{thm:main} by showing the following.
\begin{prop} 
If $T=(p^a)^{o(1)}$, and $\re\alpha_j\ll \frac{1}{\log p^a}$ and $\im\alpha_j=O(1)$ for all $1\leq j\leq 3$, we have
\begin{align*}
    \mathcal{T}_{\epsilon_1,\epsilon_2,\epsilon_3}(\pm,\pm_\lambda,\chi,d,\delta,\pmb{\alpha},\pmb{\sigma}),\,\mathcal{T}^b_{\epsilon_1,\epsilon_2,\epsilon_3}(\pm,\pm_\lambda,\chi,d,\delta,\pmb{\alpha},\pmb{\sigma}) \ll_\varepsilon p^{-b+a\varepsilon}.
\end{align*}
\end{prop}
We prove this using the stationary phase results from \cite[Section 11]{PeY} (the proofs of which are based on those in \cite[Section 10]{PeYpetersson}). The bounds in \cite[Section 12]{PeY} then imply that the cancellation due to summation over $\phi\,(\mo p^b)$ allows us to save a factor $p^{b}$ on off-diagonal terms.

Throughout, we write $N\coloneqq N_1N_2N_3$ and $M\coloneqq M_1M_2M_3$. 

\subsection{Modifications to stationary phase results}\label{sec:modifications}
The stationary phase results from \cite[Section 11]{PeY} apply to the specific case of $\{\alpha_1=\alpha_2=\alpha_3=0\}$. With some minor modifications, the results remain valid in our setting.

For the oscillatory case we may adopt \cite[Lemma 11.1]{PeY}, since \cite[Lemmas 10.3 and 10.5]{PeY} still apply. So in this case we can safely ignore the terms involving $K^{-}(\cdot)$, $K^{b,-}(\cdot)$ and focus solely on the terms involving $K^{+}(\cdot)$, $K^{b,+}(\cdot)$. Moreover, we suppress everywhere factors that are powers of $T$ as we are assuming $T\ll_\varepsilon p^{a\varepsilon}$.

For the non-oscillatory case we may adopt \cite[Lemma 11.2]{PeY}, as the bounds in \cite[Lemmas 10.2 and 10.4]{PeY} still hold due to Lemma \ref{lemma:bounds_for_bessel_transforms}. We also note the following modifications:
\begin{itemize}
    \item[-] We suppress everywhere factors that are powers of $T$, for the same reason as before.
    \item[-] For $K^-(\cdot)$, $K^{b,-}(\cdot)$, the $(-)$-sign in the phase factor $e_c(-m_1m_2m_3)$ should be a $(+)$; see also \cite[Lemma 8.1]{Y} for the correct sign. This correction only has an effect on the definition of $f(y)$ (though its stated properties remain the same), so we denote it instead by $f^+(y)$ or $f^-(y)$, respectively. (We do not make this distinction in the oscillatory case, as $K^-(\cdot)$, $K^{b,-}(\cdot)$ are then negligibly small.)
\end{itemize}

\subsection{Bounding $\mathcal{T}_{\epsilon_1,\epsilon_2,\epsilon_3}(\cdot)$}
Without loss of generality, we only treat the case $\{\epsilon_1=\epsilon_2=\epsilon_3=1\}$. Recalling \eqref{eq:GH_identity}, we have 
\begin{align}
    \mathcal{T}_{1,1,1}(\cdot) &\ll \frac{1}{C\sqrt{N}}\sum_{\,\,\,\pm_b}\bigg\lvert\sum_{\substack{m_1,m_3,m_3,r\geq1}}G_\chi^{\pm_\lambda}(m_1,m_2,m_3;p^{a+b}r)K^{\pm_\lambda}(\cdot,\tfrac{1\mp\chi(-1)(\pm_b1)}{2},\cdot)\bigg\rvert \\
    &\ll \frac{1}{C^2p^a\sqrt{N}}\sum_{\,\,\,\pm_b}\bigg\lvert\sum_{\substack{m_1,m_2,m_3,r\geq1 \\ (m_1,pr)=1}}e_{p^{a+b}r}(\pm_\lambda m_1m_2m_3)H_\chi(\pm_\lambda m_1,m_2,m_3;p^br) \label{eq:sum_of_H} \\
    &\quad\quad\quad\quad\quad\quad\quad\quad\quad\quad\quad\quad\quad\quad\quad\quad\quad\quad\quad\quad\quad\,\,\,\cdot K^{\pm_\lambda}(\cdot,\tfrac{1\mp\chi(-1)(\pm_b1)}{2},\cdot)\bigg\rvert.
\end{align}

The behaviour of $K^\pm(\cdot)$ depends on whether or not 
\begin{align}
    \frac{\sqrt{N}}{C}\gg p^{a\eta}\label{eq:oscillatory_condition}
\end{align}
holds for some $\eta>0$.

\subsubsection{Oscillatory case} 
We assume \eqref{eq:oscillatory_condition} holds. Then by \cite[Lemma 11.1]{PeY} (with modifications as stated in Section \ref{sec:modifications}) we obtain that \eqref{eq:sum_of_H} is
\begin{align*}
    &\ll_\varepsilon \frac{p^{a\varepsilon}}{p^aM}\bigg\lvert \int_{\abs{\pmb{u}}\ll p^{a\varepsilon}}\int_{\abs{y}\ll p^{a\varepsilon}}F(\pmb{u};y)p^{-(a+b)(iy+u_4)}M_1^{u_1}M_2^{u_2}M_3^{u_3}C^{u_4} \\
    &\quad\quad\quad\quad\quad\quad\quad\quad\quad\quad\quad\,\,\,\, \cdot Z^{+}(u_1-iy,u_2-iy,u_3-iy,u_4+iy)\,dyd\pmb{u} \bigg\rvert,
\end{align*}
where the contribution comes from terms with $m_j\asymp M_j$ for all $j$, $c\asymp C$, and
\begin{align}
    M_j\asymp\frac{\sqrt{N}}{N_j} \quad \text{ for all }j. \label{eq:stationary_condition}
\end{align}
Here $F(\pmb{u};y)$ is entire in terms of $\pmb{u}$, and satisfies $F(\pmb{u};y)\ll_{\re\pmb{u},A}(1+\abs{\pmb{u}})^{-A}(1+\abs{y})^{-A}$ for any $A>0$. We initially take $\re u_j=1+\varepsilon$ for all $j$. Recalling \eqref{eq:Zpm_decomposition}, we consider the contribution to the above integral by $Z_0^+(\cdot)$ and $Z_1^+(\cdot)$ separately. For $Z_0^+(\cdot)$, we keep the lines at $(1+\varepsilon)$, while for $Z_1^+(\cdot)$ we can move them to $(1/2+\varepsilon)$ without having to pick up any poles. We then obtain by the bounds in Lemma \ref{lemma:hybrid} that
\begin{align*}
    \mathcal{T}_{1,1,1}(\cdot) \ll_\varepsilon \frac{p^{a\varepsilon}}{p^aM} \bigg( \frac{MC}{p^{a+b}} p^b + \sqrt{\frac{MC}{p^{a+b}}} p^{\frac{3a-b}{2}}\bigg) \asymp p^{a\varepsilon}\bigg(\frac{C}{p^{2a}} + \frac{\sqrt{C}}{p^b N^{1/4}}\bigg) \ll_\varepsilon p^{-b+a\varepsilon},
\end{align*}
where in the second step we applied \eqref{eq:stationary_condition} and in the third step we used the fact that $C\ll \sqrt{N}\ll_\varepsilon p^{a(\frac{3}{2}+\varepsilon)}$. 

\subsubsection{Non-oscillatory case} 
We now assume that $\sqrt{N}/C\ll p^{a\eta}$ for some very small $\eta>0$. We can then apply \cite[Lemma 11.2]{PeY} (with modifications as stated in Section \ref{sec:modifications}) to obtain that
\begin{align*}
    \mathcal{T}_{1,1,1}(\cdot) &\ll_\varepsilon \frac{Np^{a\varepsilon}}{C^3p^a}\sum_{\,\,\,\pm_\lambda}\bigg\lvert\int_{\abs{\pmb{u}}\ll p^{a(\eta+\varepsilon)}}F(\pmb{u})\int_{\abs{y}\ll p^{a\varepsilon}+\frac{M}{C}}f^{\pm_\lambda}(y)p^{-(a+b)(iy+u_4)}M_1^{u_1}M_2^{u_2}M_3^{u_3}C^{u_4} \\
    &\quad\quad\quad\quad\quad\quad\quad\quad\quad\quad\quad\quad\quad\quad\quad\quad\quad\quad\,\,\,\,\, \cdot Z^{\pm_\lambda}(u_1-iy,u_2-iy,u_3-iy,u_4+iy)\,dyd\pmb{u}\bigg\rvert,
\end{align*}
where the contribution comes from terms with $m_j\asymp M_j$ for all $j$, $c\asymp C$, and
\begin{align}
    \frac{M_jN_j}{C}\ll_\varepsilon p^{a\varepsilon}\quad\text{ for all }j. \label{eq:non_oscillatory_condition}
\end{align}
Here $f^\pm(y)\ll(1+\abs{y})^{-1/2}$ and $F(\pmb{u})\ll_{A,\re\pmb{u},\varepsilon}p^{a\varepsilon}\prod_{l=1}^4(1+\abs{u_l}/p^{a\varepsilon})^{-A}$ for all $A>0$. If there exists $\varepsilon>0$ such that $\frac{M}{C}\gg p^{a\varepsilon}$, then $f^\pm(y)$ may be chosen to have support on $\abs{y}\asymp \frac{M}{C}$. We initially take $\re u_j=1+\varepsilon$ for all $j$. Recalling \eqref{eq:Zpm_decomposition}, we consider the contribution to the above integral by $Z^\pm_0(\cdot)$ and $Z^\pm_1(\cdot)$ separately, calling these parts $\mathcal{T}_{00}$ and $\mathcal{T}_1$, respectively.

In the case of $\mathcal{T}_1$, we can shift the contours to $\re u_j=1/2+\varepsilon$ for all $j$. By the large sieve-like bound \eqref{eq:lemma_hybrid}, we then have
\begin{align*}
    \mathcal{T}_1 \ll_\varepsilon \frac{Np^{a\varepsilon}}{C^3p^a}\sqrt{\frac{MC}{p^{a+b}}}\bigg(1+\sqrt{\frac{M}{C}}\bigg)p^{\frac{3a-b}{2}}.
\end{align*}
Applying first \eqref{eq:non_oscillatory_condition} and then the non-oscillatory condition, the latter bound further leads to
\begin{align*}
    \mathcal{T}_1 \ll_\varepsilon p^{-b+a\varepsilon}\bigg(\frac{\sqrt{N}}{C}+1\bigg) \ll_\varepsilon p^{-b+a\varepsilon}.
\end{align*}

Now for $\mathcal{T}_{00}$. If $\frac{M}{C}\gg p^{a\varepsilon}$ for some $\varepsilon>0$ then we shift the $u_j$-contours to $(1/2+\varepsilon)$, which is allowed as the poles of $Z_0^\pm(u_1-iy,u_2-iy,u_3-iy,u_4+iy)$ occur at height $\im u_j\asymp \abs{y}\asymp\frac{M}{C}$ but we are truncating the $u_j$-integrals at $\ll p^{a\varepsilon}$. The large sieve-like bound \eqref{eq:lemma_hybrid} and \eqref{eq:non_oscillatory_condition} then give
\begin{align}
    \mathcal{T}_{00} \ll_\varepsilon \frac{Np^{a\varepsilon}}{C^3p^a}\sqrt{\frac{MC}{p^{a+b}}}\sqrt{\frac{M}{C}}p^{\frac{a+b}{2}} \ll_\varepsilon p^{a(-1+\varepsilon)}.
\end{align}
Suppose now on the other hand that $\frac{M}{C}\ll_\varepsilon p^{a\varepsilon}$ for all $\varepsilon>0$. In this case we keep the $u_j$-contours at $(1+\varepsilon)$, which gives 
\begin{align*}
    \mathcal{T}_{00} \ll_\varepsilon \frac{Np^{a\varepsilon}}{C^3p^a}\frac{MC}{p^{a+b}}p^b \ll_\varepsilon \frac{Np^{a\varepsilon}}{Cp^{2a}} \ll_\varepsilon \frac{p^{3a}p^{a\varepsilon}}{p^{a+b}p^{2a}} =p^{-b+a\varepsilon},
\end{align*}
where in the last bound we applied \eqref{eq:conditions_normal_case}.

\subsection{Bounding $\mathcal{T}^b_{\epsilon_1,\epsilon_2,\epsilon_3}(\cdot)$}
Without loss of generality again we only treat the case $\{\epsilon_1=\epsilon_2=\epsilon_3=1\}$. We have
\begin{align}
    \mathcal{T}^b_{1,1,1}(\cdot) &\ll \frac{1}{C\sqrt{N}}\sum_{\,\,\,\pm_b}\bigg\lvert\sum_{\substack{m_1,m_2,m_3,c'\geq1 \\ (p,c')=1}}G_\chi^{b,\pm_\lambda}(m_1,m_2,m_3;p^{a-b}c')K^{b,\pm_\lambda}(\cdot,\tfrac{1\mp\chi(-1)(\pm_b1)}{2},\cdot)\bigg\rvert \\
    &\ll \frac{1}{C^2p^{2a-3b}\sqrt{N}}\sum_{\,\,\,\pm_b}\bigg\lvert\sum_{\substack{m_1,m_2,m_3,c'\geq1 \\ (m_1p,c')=1}}e_{p^ac'}(\pm_\lambda m_1m_2m_3p^{3b})H_\chi^{b,+}(\pm_\lambda m_1,m_2,m_3;p^{a-b}c') \label{eq:sum_of_Hb} \\
    &\quad\quad\quad\quad\quad\quad\quad\quad\quad\quad\quad\quad\quad\quad\quad\quad\quad\quad\quad\quad\quad\quad\quad\quad\quad\quad \cdot K^{b,\pm}(\cdot,\tfrac{1\mp\chi(-1)(\pm_b1)}{2},\cdot)\bigg\rvert, 
\end{align}
where we recall \eqref{eq:Kbpm_rewritten}. The ensuing bounds are obtained similarly to those for $\mathcal{T}_{1,1,1}(\cdot)$; we still give a bit of detail in the oscillatory case.

\subsubsection{Oscillatory case}
Suppose $\eqref{eq:oscillatory_condition}$ holds. Then we obtain as before that
\begin{align*}
    \mathcal{T}_{1,1,1}^b(\cdot) &\ll_\varepsilon \frac{p^{a\varepsilon}}{p^{2a}M}\bigg\lvert\int_{\abs{\pmb{u}}\ll p^{a\varepsilon}}\int_{\abs{y}\ll p^{a\varepsilon}}F(\pmb{u};y)p^{-a(iy+u_4)+3biy}M_1^{u_1}M_2^{u_2}M_3^{u_3}C^{u_4} \\
    &\quad\quad\quad\quad\quad\quad\quad\quad\quad\quad\quad\quad\, \cdot Z^{b,+}(u_1-iy,u_2-iy,u_3-iy,u_4+iy)\,dyd\pmb{u}\bigg\rvert,
\end{align*}
where now the contribution comes from terms with $m_j\asymp M_j$ for all $j$, $c\asymp C$, and 
\begin{align}
    p^bM_j\asymp\frac{\sqrt{N}}{N_j}\quad\text{ for all }j. \label{eq:stationary_condition_caseb}
\end{align}
By the bounds in Lemma \ref{lemma:hybrid_b_case}, we then have
\begin{align*}
    \mathcal{T}^b_{1,1,1}(\cdot) \ll_\varepsilon \frac{p^{a\varepsilon}}{p^{2a}M}\bigg(\frac{MC}{p^a}+\sqrt{\frac{MC}{p^a}}p^{\frac{5(a-b)}{2}}\bigg) \asymp p^{a\varepsilon}\bigg(\frac{C}{p^{3a}}+\frac{\sqrt{C}}{p^bN^{1/4}}\bigg) \ll_\varepsilon p^{-b+a\varepsilon},
\end{align*}
where in the second step we now applied \eqref{eq:stationary_condition_caseb}.

\subsubsection{Non-oscillatory case}
Now assume $\sqrt{N}/C\ll p^{a\eta}$ for some very small $\eta>0$. The contribution to $\mathcal{T}_{1,1,1}^b(\cdot)$ from $Z_1^{b,\,\pm}(\cdot)$ is $\ll_\varepsilon p^{-b+a\varepsilon}$, and from $Z_0^{b,\,\pm}(\cdot)$ is no worse than $p^{a(-1+\varepsilon)}$.
\section{Conjectured asymptotic formula for the cubic moment}
\label{sec:conjecture}
\subsection{Background}
Katz and Sarnak \cite{KS} introduced the seminal idea that certain symmetry types from random matrix theory determine the zero distributions of families of $L$-functions, a result that they rigorously proved for function fields \cite{KSbook}. Extensive evidence suggests this relationship extends to low-lying zeros for families of $L$-functions over number fields; see e.g. \cite{ILS}. This paradigm was further expanded by Conrey and Farmer \cite{CF}, and Keating and Snaith \cite{KSn0,KSn}. The first two authors applied the concept of symmetry types to analyze the mean values of families $\mathcal{F}$ of $L$-functions at the central point $s=1/2$. They conjecture that
\begin{align}
    \frac{1}{\mathcal{Q}^*}\sum_{\substack{f\in\mathcal{F} \\ c(f)\leq \mathcal{Q}}} V(L_f(1/2))^k \sim g_k\frac{a(k)}{\Gamma(1+B(k))}(\log \mathcal{Q}^A)^{B(k)}\quad\text{ as }\mathcal{Q}\rightarrow+\infty, \label{eq:conjecture_conrey_farmer}
\end{align}
where 
\begin{itemize}
    \item[-] the family $\mathcal{F}$ is partially ordered by the conductor $c(f)$ of each $L$-function,
    \item[-] the sum ranges over the $\mathcal{Q}^*$ elements with $c(f)\leq\mathcal{Q}$, which is 
    $$\mathcal{Q}^*=\frac{1}{4\pi}\int_\R\tanh(\pi t)th_0(t)\,dt$$
    in our case,\footnote{At this point, we have not yet separated $\mathcal{F}$ into odd and even $f$; see \cite[p.886]{CF}.}
    \item[-] $V(z)$ is a symmetry-dependent function, e.g. $V(z)=z$ for orthogonal symmetry,
    \item[-] $g_k$ is symmetry-dependent, e.g. $g_3=8$ for orthogonal symmetry, 
    \item[-] $a(k)$ is family-dependent, and is given by the Euler product $A_k(\pmb{0})$ in \cite{CFKRS}, which is $(1-p^{-1})^3$ in our case,
    \item[-] $B(k)$ is symmetry-dependent, e.g. $B(3)=3$ for orthogonal symmetry, so $\Gamma(1+B(3))=6$,
    \item[-] $A$ is the degree to which the parameter $\mathcal{Q}$ occurs in the functional equation, and it is $1$ in our case.
\end{itemize}

\subsection{Setting up the conjectured asymptotic for the cubic moment}
A heuristic recipe to obtain the conjectured asymptotic \eqref{eq:conjecture_conrey_farmer} (with shifts) along with a full main term and strong error term is provided in \cite{CFKRS}. We execute the recipe step by step for our particular case.

First, we determine the relevant parameters appearing in the conjecture. Defining
\begin{align*}
    \mathcal{F}^\pm(p^a,\chi)\coloneqq \bigg(\bigsqcup_{t\in\R\cup i[-\frac{1}{2},\frac{1}{2}]}\bigsqcup_{\ell m=p^a}\mathcal{H}^\pm_{it}(m,\overline{\chi}^2)\bigg) \bigsqcup \bigg(\bigcup_{t\in\R}\bigsqcup_{\ell m=p^a}\mathcal{H}^\pm_{it,\text{Eis}}(m,\overline{\chi}^2)\bigg),
\end{align*}
write $\mathcal{F}^\pm(p^a,\chi)\otimes\chi$ for the family under consideration. Here 
$$\mathcal{H}_{it}^+(\cdot),\mathcal{H}_{it}^-(\cdot)\subset\mathcal{H}_{it}(\cdot) \quad \text{ and } \quad \mathcal{H}_{it,\text{Eis}}^+(\cdot),\mathcal{H}_{it,\text{Eis}}^-(\cdot)\subset\mathcal{H}_{it,\text{Eis}}(\cdot)$$
denote the subsets of even and odd forms. Let $u\in\mathcal{F}^\pm(p^a,\chi)$ be arbitrary. The root number $\epsilon(u\otimes\chi)=\lambda_u(-1)$ and the function
\begin{align*}
    \mathcal{X}(s,u\otimes\chi) \coloneqq
        p^{a(1-2s)}\frac{L_\infty(1-s,u\otimes\chi)}{L_\infty(s,u\otimes\chi)}
\end{align*}
appear in the asymmetric functional equation 
$$L(s,u\otimes\chi)=\epsilon(u\otimes\chi)\mathcal{X}(s,u\otimes\chi)\overline{L}(1-s,u\otimes\chi).$$
Also define the $\log$-conductor 
\begin{align*}
    c(u\otimes\chi) &\coloneqq \bigg\lvert\bigg[\epsilon(u\otimes\chi)\frac{\partial\mathcal{X}(s,u\otimes\chi)}{\partial s}\bigg]_{s=\frac{1}{2}}\bigg\rvert \\
    &= 2\log(p^a/\pi) + 2\re\frac{\Gamma'}{\Gamma}\bigg(\frac{1/2+\delta_{u\otimes\chi}+it}{2}\bigg)\in\R,
\end{align*}
which scales as the $\log$ of the ``usual'' conductor of $u\otimes\chi$. As we are dealing with an infinite family, the recipe requires $p^a$, $\delta_{u\otimes\chi}$, $t$ to be monotonic functions of $c(u\otimes\chi)$ (see \cite[Section 3.1--2]{CFKRS}). In practice, this means that these parameters should each be either constant or tend to $+\infty$ with $c(u\otimes\chi)$. This can only be the case if $\delta_{u\otimes\chi}$ is kept constant; that is, one obtains separate asymptotic formulas for the case of even $u$ and odd $u$.

For a function $G:\mathcal{F}^\pm(p^a,\chi)\rightarrow\C$, we consider the expected value
\begin{align*}
    \langle G(u)\rangle \coloneqq \lim_{a\rightarrow\infty}\bigg(\quad\sideset{}{^K}\sum_{u\in\mathcal{F}^\pm(p^a,\chi)}1\bigg)^{-1}\quad\sideset{}{^K}\sum_{u\in\mathcal{F}^\pm(p^a,\chi)}G(u),
\end{align*}
where the weighted sum
\begin{align*}
    \,\sideset{}{^K}\sum_{u\in\mathcal{F}^\pm(p^a,\chi)}G(u)
    &\coloneqq \sum_{t}h_0(t)\sum_{\ell m=p^a}\sum_{f\in\mathcal{H}^\pm_{it}(m,\overline{\chi}^2)}w_{f,\ell}G(f) + \frac{1}{4\pi}\int_{\R}h_0(t)\sum_{\ell m=p^a}\sum_{E\in\mathcal{H}^\pm_{it,\text{Eis}}(m,\overline{\chi}^2)}w_{E,\ell}G(E)\,dt
\end{align*}
includes the test function and harmonic weights appearing in the Kuznetsov trace formula \eqref{eq:kuznetsov_rhs}.

We now proceed with the recipe in \cite[Section 4.1]{CFKRS}:
\begin{enumerate}
    \item[(i)] We start with a product of three shifted $L$-functions
    \begin{align*}
        \mathcal{L}_{u}(s,\pmb{\alpha}) \coloneqq \prod_{1\leq j\leq 3}L(s+\alpha_j,u\otimes\chi)
    \end{align*}
    for all $u\in\mathcal{F}^\pm(p^a,\chi)$, where $\pmb{\alpha}=(\alpha_1,\alpha_2,\alpha_3)\in\C^3$. 

    \item[(ii)] For each $u\in\mathcal{F}^\pm(p^a,\chi)$, apply to each factor $L(s+\alpha_j,u\otimes\chi)$ its approximate functional equation, which is roughly of the form
    \begin{align}
        \sum_{n_j}\frac{\lambda_u(n_j)\chi(n_j)}{n_j^{s+\alpha_j}}+\epsilon(u\otimes\chi)\mathcal{X}(s+\alpha_j,u\otimes\chi)\sum_{n_j}\frac{\overline{\lambda_u}(n_j)\overline{\chi}(n_j)}{n_j^{1-s-\alpha_j}}. \label{eq:AFE_conjectured_moment}
    \end{align}
    Recall that $\lambda_u(n)\chi(u)\in\R$ for all $n$. Multiplying out \eqref{eq:AFE_conjectured_moment} for $1\leq j\leq 3$ gives us $8=2^3$ terms in total, each of which we write as 
    \begin{align}
        &\text{(product of $\epsilon(u\otimes\chi)$-factors)}\cdot\text{(product of $\mathcal{X}(s+\alpha_j,u\otimes\chi)$-factors)} \label{eq:crude_terms} \\
        &\cdot\sum_{n_1,n_2,n_3}\text{(summand)}. 
    \end{align}

    \item[(iii)] Replace each product of $\epsilon(u\otimes\chi)$-factors in \eqref{eq:crude_terms} by its expected value 
    \begin{align*}
        \langle\epsilon(\cdot\otimes\chi)^k\rangle = (\pm1)^{k}
    \end{align*}
    for the relevant $0\leq k\leq 3$; see \eqref{eq:ktf_one_eigenvalue} below. So we will have $8=2^{3}$ terms in our final answer.

    \item[(iv)] Replace each summand in \eqref{eq:crude_terms} by its expected value, which is
    \begin{align*}
        \frac{\langle\prod_{1\leq j\leq 3}\lambda_u(n_j)\chi(n_j)\rangle}{\prod_{1\leq j\leq k}n_j^{s+\alpha_j}\prod_{k+1\leq j\leq 3}n_j^{1-s-\alpha_j}}
    \end{align*}
    for some $0\leq k\leq 3$. Since $\prod_{j=1}^3\lambda_u(n_j)=\sum_{m\in\Z}b_{m}\lambda_u(m)$ for some $b_m\in\C$ by Hecke multiplicativity, it suffices to compute 
    \begin{itemize}
        \item[-] $\langle\lambda_u(m)\rangle$ for $m\in\Z$ with $(m,p)=1$, and
        \item[-] $b_m$ whenever $\langle\lambda_u(m)\rangle\neq0$ for such $m$.
    \end{itemize}
    By the trace formula \eqref{eq:kuznetsov_rhs} with $n_1=m$, $n_2=1$ and test function $h(t,\cdot)\coloneqq h_0(t)$, followed by the bound \eqref{eq:weil_bound} and Lemma \ref{lemma:bounds_for_bessel_transforms}, we obtain for $m\in\Z$ with $(m,p)=1$ that
    \begin{align}
        \,\sideset{}{^K}\sum_{u\in\mathcal{F}^\pm(p^a,\chi)}\lambda_u(m) &= \,\,\,\sideset{}{^K}\sum_{u\in\mathcal{F}(p^a,\chi)}\frac{1\pm\lambda_u(-1)}{2}\lambda_u(m) \\
        &= \frac{\mathbbm{1}_{\{m=1\}}g_0(\cdot)}{2} + O_\varepsilon(p^{a(-1+\varepsilon)}\abs{m}^{1/2}T^{1+\varepsilon}). \label{eq:ktf_one_eigenvalue}
    \end{align}
    So $\langle\lambda_u(m)\rangle = \mathbbm{1}_{\{m=1\}}$, and one verifies that
    \begin{align*}
        b_1 = 
        \begin{cases}
            \overline{\chi}(n_1n_2n_3) &\text{ if } \frac{n_1n_2}{d^2}=n_3\text{ for some }d\mid(n_1,n_2) \\
            0 &\text{ otherwise}
        \end{cases},
    \end{align*}
    from which we conclude that
    \begin{align}
        \langle\prod_{1\leq j\leq 3}\lambda_u(n_j)\chi(n_j)\rangle = \mathbbm{1}_{\{n_1n_2/d^2=n_3\text{ for some }d\mid(n_1,n_2)\}}. \label{eq:average_factor_conj}
    \end{align}
    In particular, we have the multiplicativity relation
    \begin{align}
        \langle\prod_{1\leq j\leq 3}\lambda_u(m_jn_j)\chi(m_jn_j)\rangle = \langle\prod_{1\leq j\leq 3}\lambda_u(m_j)\chi(m_j)\rangle\cdot\langle\prod_{1\leq j\leq 3}\lambda_u(n_j)\chi(n_j)\rangle \label{eq:multiplicativity_conj}
    \end{align}
    for $(m_1m_2m_3,n_1n_2n_3)=1$.

    \item[(v)] Complete the resulting sums in \eqref{eq:crude_terms} (recalling that they appear ``truncated'' in the approximate functional equation), and denote the result by $M_u(s,\pmb{\alpha})$. Note that this quantity depends on $u$ only through the products of $\mathcal{X}(s+\alpha_j,u\otimes\chi)$-factors. The sums over $n_1,n_2,n_3$ now become factorizable Euler products
    \begin{align}
        &\sum_{n_1,n_2,n_3\geq1}\frac{\langle\prod_{1\leq j\leq 3}\lambda_u(n_j)\chi(n_j)\rangle}{\prod_{1\leq j\leq k}n_j^{s+\alpha_j}\prod_{k+1\leq j\leq 3}n_j^{1-s-\alpha_j}} \\
        &= \prod_{\widetilde{p}\neq p}\sum_{\substack{e_1,e_2,e_3\geq0 \\ 0\leq d\leq\min\{e_1,e_2\} \\ e_1+e_2-2d=e_3}}\frac{1}{\prod_{1\leq j\leq k}\widetilde{p}^{\,e_j(s+\alpha_j)}\prod_{k+1\leq j\leq 3}\widetilde{p}^{\,e_j(1-s-\alpha_j)}} \label{eq:euler_product}
    \end{align}
    for some $0\leq k\leq 3$, due to \eqref{eq:average_factor_conj} and \eqref{eq:multiplicativity_conj}. We replace the exponents $1-s-\alpha_j$ in the $3-k$ factors by $s-\alpha_j$ (as we are evaluating at $s=1/2$ in the end anyway), so that the Euler product converges for $\re s$ large enough. Since $\mathcal{X}(s,u\otimes\chi)=\mathcal{X}(1-s,u\otimes\chi)^{-1}$, we also replace the $\mathcal{X}(s+\alpha_j,u\otimes\chi)$-factors in \eqref{eq:AFE_conjectured_moment} by $\mathcal{X}(s-\alpha_j,u\otimes\chi)^{-1}$.

    We now identify the leading-order poles of \eqref{eq:euler_product}. Evaluating the geometric series defining each Euler product factor gives, for example for $k=3$,
    \begin{align}
        R(s,\pmb{\alpha}) &\coloneqq \prod_{\widetilde{p}\neq p}\sum_{\substack{e_1,e_2,e_3\geq0 \\ 0\leq d\leq\min\{e_1,e_2\} \\ e_1+e_2-2d=e_3}}\frac{1}{\prod_{1\leq j\leq 3}\widetilde{p}^{\,e_j(s+\alpha_j)}} \\
        &\,= \zeta^{(p)}(2s+\alpha_1+\alpha_2)\zeta^{(p)}(2s+\alpha_1+\alpha_3)\zeta^{(p)}(2s+\alpha_2+\alpha_3),
    \end{align}
    which only has simple poles at $s=\frac{1}{2}(1-\alpha_j-\alpha_k)$ for $1\leq j<k\leq 3$, and is holomorphic everywhere else. 

    We end up with
    \begin{align*}
        M_u(s,\pmb{\alpha}) &= R(s,\pmb{\alpha}) \\
        &\quad\, \pm \mathcal{X}(s-\alpha_1,u\otimes\chi)^{-1}R(s,-\alpha_1,\alpha_2,\alpha_3) \\
        &\quad\, \pm \mathcal{X}(s-\alpha_2,u\otimes\chi)^{-1}R(s,\alpha_1,-\alpha_2,\alpha_3) \\
        &\quad\, \pm \mathcal{X}(s-\alpha_3,u\otimes\chi)^{-1}R(s,\alpha_1,\alpha_2,-\alpha_3) \\
        &\quad\, + (\mathcal{X}(s-\alpha_1,u\otimes\chi)\mathcal{X}(s-\alpha_2,u\otimes\chi))^{-1}R(s,-\alpha_1,-\alpha_2,\alpha_3) \\
        &\quad\, + (\mathcal{X}(s-\alpha_1,u\otimes\chi)\mathcal{X}(s-\alpha_3,u\otimes\chi))^{-1}R(s,-\alpha_1,\alpha_2,-\alpha_3) \\
        &\quad\, + (\mathcal{X}(s-\alpha_2,u\otimes\chi)\mathcal{X}(s-\alpha_3,u\otimes\chi))^{-1}R(s,\alpha_1,-\alpha_2,-\alpha_3) \\
        &\quad\, \pm (\mathcal{X}(s-\alpha_1,u\otimes\chi)\mathcal{X}(s-\alpha_2,u\otimes\chi)\mathcal{X}(s-\alpha_3,u\otimes\chi))^{-1}R(s,-\alpha_1,-\alpha_2,-\alpha_3).
    \end{align*}
    As expected, we observe one $0$-swap term, three $1$-swap terms, three $2$-swap terms, and one $3$-swap term.
\end{enumerate}

The conjecture of \cite{CFKRS} now states the following.
\begin{conjecture}\label{conj:cfkrs}
    One has
    \begin{align}
        \mathcal{M}_{nh}^\pm(\chi,\pmb{\alpha}) \coloneqq \,\,\,\sideset{}{^K}\sum_{u\in\mathcal{F}^\pm(p^a,\chi)}\mathcal{L}_u(1/2,\pmb{\alpha}) = \,\,\,\sideset{}{^K}\sum_{u\in\mathcal{F}^\pm(p^a,\chi)}M_{u}(1/2,\pmb{\alpha})(1+O_\varepsilon(e^{(-1/2+\varepsilon)c(u\otimes\chi)})). \label{eq:final_boss_conjecture}
    \end{align}
\end{conjecture}

In the setting of Theorem \ref{thm:main}, we have $e^{(-1/2+\varepsilon)c(u\otimes\chi)}= p^{a(-1+\varepsilon)}$. Note however that in general an error term of this strength may be overly optimistic; for example, Diaconu and Whitehead \cite{DW} showed existence of a secondary main term for a cubic moment of quadratic Dirichlet $L$-functions.

We would like the expression for the main term 
\begin{align}
    \,\,\,\sideset{}{^K}\sum_{u\in\mathcal{F}^\pm(p^a,\chi)}M_u(1/2,\pmb{\alpha}) \label{eq:conj_main_term}
\end{align}
in \eqref{eq:final_boss_conjecture} to match our diagonal computations in Section \ref{sec:diag}. To show this, we remark that each
\begin{align*}
    \mathcal{X}(1/2-\alpha_j,u\otimes\chi)^{-1} = p^{-2a\alpha_j}\gamma_{\rat}(\delta_{u\otimes\chi},t,-\alpha_j,\alpha_j)
\end{align*}
depends on $u$ only through the parity $\delta_{u\otimes\chi}$ and the spectral parameter $t$. So we obtain by \eqref{eq:kuznetsov_rhs} with $n_1=1$ and $n_2\in\{\pm1\}$ that
\begin{align}
    &\,\,\,\sideset{}{^K}\sum_{u\in\mathcal{F}^\pm(p^a,\chi)}M_u(1/2,\pmb{\alpha}) = \,\,\,\sideset{}{^K}\sum_{u\in\mathcal{F}(p^a,\chi)}\frac{1\pm\lambda_u(-1)}{2}M_u(1/2,\pmb{\alpha}) \\
    &= \frac{1}{2p^{a\alpha}}\sum_{\pmb{\sigma}\in\{\pm1\}^3}(\pm1)^{\frac{1-\sigma}{2}}\bigg(\frac{1}{4\pi}\int_{\R}\tanh(\pi t)th_0(t)f_{\mt}(\tfrac{1\mp\chi(-1)}{2},t,\pmb{\alpha},\pmb{\sigma})\,dt \\
    &\quad\quad\quad\quad\quad\quad\quad\quad\quad\quad\quad\, + \sum_{\substack{c\geq1 \\ p^a\mid c}}\frac{S_{\overline{\chi}^2}(1,1;c)}{c}\frac{i}{2}\int_\R\frac{J_{2it}(4\pi/c)}{\cosh(\pi t)}th_0(t)f_{\mt}(\tfrac{1\mp\chi(-1)}{2},t,\pmb{\alpha},\pmb{\sigma})\,dt \\
    &\quad\quad\quad\quad\quad\quad\quad\quad\quad\quad\quad\, \pm\sum_{\substack{c\geq1 \\ p^a\mid c}}\frac{S_{\overline{\chi}^2}(1,-1;c)}{c}\frac{1}{\pi}\int_\R K_{2it}(x)\sinh(\pi t)th_0(t)f_{\mt}(\tfrac{1\mp\chi(-1)}{2},t,\pmb{\alpha},\pmb{\sigma})\,dt\bigg). \label{eq:conjectured_asymptotic}
\end{align}
Similar to the computations in Section \ref{sec:asymptotic_for_diag}, one can show that
\begin{align*}
    \frac{1}{2p^{a\alpha}}\sum_{\pmb{\sigma}\in\{\pm1\}^3}(\pm1)^{\frac{1-\sigma}{2}}\sum_{\substack{c\geq1 \\ p^a\mid c}}(\dotsc) \ll_\varepsilon p^{a(-1+\varepsilon)},
\end{align*}
i.e. the off-diagonal contribution in \eqref{eq:conjectured_asymptotic} gets absorbed into the error term in \eqref{eq:final_boss_conjecture}. So we see that 
\begin{align*}
    \sideset{}{^K}\sum_{u\in\mathcal{F}^\pm(p^a,\chi)}M_u(1/2,\pmb{\alpha}) = \sum_{\pmb{\sigma}\in\{\pm1\}^3}(\pm1)^{\frac{1-\sigma}{2}}\text{MT}^\pm_{nh}(\chi,\pmb{\alpha},\pmb{\sigma}) + O_\varepsilon(p^{a(-1+\varepsilon)}),
\end{align*}
with $\text{MT}_{nh}^\pm(\cdot)$ as defined in \eqref{eq:conjectured_main_term_after_extraction}. In particular, this proves Lemma \ref{lemma:conjectured_main_term}.

\appendix
\section{Proof of the holomorphic case}\label{sec:proof_hol}
Let $q\in\Z_{\geq1}$, $\kappa\in2\Z_{\geq1}$, $\chi$ primitive modulo $q$. For $m\mid q$, write $\mathcal{H}_\kappa(m,\overline{\chi}^2)$ for the set of normalized Hecke newforms of weight $\kappa$, level $m$, and central character $\overline{\chi}^2$. Given $f\in\mathcal{H}_{\kappa}(m,\overline{\chi}^2)$, one has $f\otimes\chi\in\mathcal{H}_\kappa(q^2,1)$. For $\re s>1$, we define the $L$-series
\begin{align*}
    L(s,f\otimes\chi) \coloneqq \sum_{n\geq1}\frac{\lambda_f(n)\chi(n)}{n^s},
\end{align*}
as well as the local $L$-function at the infinite place
\begin{align*}
    L_\infty(s,f\otimes\chi) \coloneqq (2\pi)^{-s}\Gamma\left(s+\frac{\kappa-1}{2}\right),
\end{align*}
and the completed $L$-function $\Lambda(s,f\otimes\chi)\coloneqq q^sL_\infty(s,f\otimes\chi)L(s,f\otimes\chi)$. The functional equation $\Lambda(s,f\otimes\chi)=\epsilon(f\otimes\chi)\Lambda(1-s,f\otimes\chi)$ then holds, with root number $\epsilon(f\otimes\chi)=\chi(-1)i^{-\kappa}$.

Write $w_{f,\ell}$ for some harmonic weights facilitating a Petersson trace formula for Hecke newforms, as opposed to the ``standard'' Petersson formula applied in \cite[Section 2]{Pe}. The holomorphic analog of \eqref{eq:moment} is 
\begin{align*}
    \mathcal{M}_h(\chi,\pmb{\alpha}) \coloneqq \sum_{\ell m=p^a}\sum_{f\in\mathcal{F}_\kappa(m,\overline{\chi}^2)}w_{f,\ell}\prod_{1\leq j\leq 3}L(1/2+\alpha_j,f\otimes\chi).
\end{align*}

We claim the following result.
\begin{theorem}
    Let $p$ be an odd prime, $a,b\in\Z_{\geq1}$ with $b\leq a/3$, and $\re\alpha_j\ll \frac{1}{\log p^a}$ for all $1\leq j\leq3$. One has
    \begin{align}
        \E_{\phi\,(\mo p^b)}[\mathcal{M}_{h}(\chi\phi,\pmb{\alpha})] = \sum_{\pmb{\sigma}\in\{\pm1\}^3}\frac{1+\sigma}{2}\text{MT}_h(\pmb{\alpha},\pmb{\sigma})+O_\varepsilon(p^{-b+a\varepsilon}), \label{eq:claim_asymptotic_hol}
    \end{align}
    where
    \begin{align}
        \text{MT}_h(\pmb{\alpha},\pmb{\sigma}) &\coloneqq \frac{p^{a\sigma\alpha}}{p^{a\alpha}}\zeta^{(p)}(1+\sigma_1\alpha_1+\sigma_2\alpha_2)\zeta^{(p)}(1+\sigma_1\alpha_1+\sigma_3\alpha_3)\zeta^{(p)}(1+\sigma_2\alpha_2+\sigma_3\alpha_3) \\
        &\quad\,\, \cdot\prod_{1\leq j\leq 3}\frac{\Gamma_\C(\sigma_j\alpha_j+\kappa/2)}{\Gamma_\C(\alpha_j+\kappa/2)}.
    \end{align}
\end{theorem}

\begin{proof}
    We specify the changes needed to the proof of Theorem \ref{thm:main}.

    Write
    \begin{align*}
        V_{1/2+\alpha_j}(y,\kappa)\coloneqq\frac{1}{2\pi i}\int_{(\sigma)}y^{-s}\frac{L_\infty(1/2+s,f\otimes\chi)}{L_\infty(1/2+\alpha_j,f\otimes\chi)}\frac{G(s-\alpha_j)}{s-\alpha_j}\,ds
    \end{align*}
    for the weight function in the approximate functional equation, and let 
    \begin{align*}
        &h(\kappa,n_1,n_2,n_3,d,\pmb{\alpha},\pmb{\sigma}) \\
        &\coloneqq V_{1/2+\sigma_1\alpha_1}\bigg(\frac{n_1}{p^a},\kappa\bigg)V_{1/2+\sigma_2\alpha_2}\bigg(\frac{dn_2}{p^a},\kappa\bigg)V_{1/2+\sigma_3\alpha_3}\bigg(\frac{dn_3}{p^a},\kappa\bigg)\prod_{1\leq j\leq 3}\frac{L_\infty(1/2+\sigma_j\alpha_j,f\otimes\chi)}{L_\infty(1/2+\alpha_j,f\otimes\chi)}.
    \end{align*}
    Then one obtains for $f\in\mathcal{H}_\kappa(m,\overline{\chi}^2)$, $m\mid q$, that
    \begin{align}
        &\prod_{1\leq j\leq 3}L(1/2+\alpha_j,f\otimes\chi) \\
        &= q^{-\alpha}\sum_{\pmb{\sigma}\in\{\pm1\}^3}(\chi(-1)i^{-\kappa})^\sigma\sum_{\substack{d\geq1 \\ (d,q)=1}}\frac{1}{d}\sum_{n_1,n_2,n_3\geq1}\frac{\lambda_f(n_1)\overline{\lambda_f}(n_2n_3)\chi(n_1)\overline{\chi}(n_2n_3)}{\sqrt{n_1n_2n_3}}h(\cdot),
    \end{align}
    where $\alpha\coloneqq\alpha_1+\alpha_2+\alpha_3$, $\sigma\coloneqq\sum_{1\leq j\leq 3}(1-\sigma_j)/2$. Similar to Proposition \ref{prop:kuznetsov}, one can write down a newform Petersson trace formula
    \begin{align}
        \sum_{\ell m=p^a}\sum_{f\in\mathcal{H}_\kappa(m,\overline{\chi}^2)}w_{f,\ell}\lambda_f(n_1)\overline{\lambda_f}(n_2) = \mathbbm{1}_{\{n_1=n_2\}} + 2\pi i^{-\kappa}\sum_{\substack{c\geq1 \\ p^a\mid c}}\frac{S_{\overline{\chi}^2}(n_1,n_2;c)}{c}J_{\kappa-1}\bigg(\frac{4\pi\sqrt{n_1n_2}}{c}\bigg) \label{eq:petersson_newforms}
    \end{align}
    for $n_1,n_2\in\Z_{\geq1}$ with $(n_1n_2,p)=1$. Applying \eqref{eq:petersson_newforms} with indices $n_1$ and $n_2n_3$, we arrive at 
    \begin{align}
        \mathcal{M}_h(\chi,\pmb{\alpha}) = p^{-a\alpha}\sum_{\pmb{\sigma}\in\{\pm1\}^3}(\chi(-1)i^{-\kappa})^\sigma (\mathcal{D}(\pmb{\alpha},\pmb{\sigma})+2\pi i^{-\kappa}\mathcal{S}(\chi,\pmb{\alpha},\pmb{\sigma})), \label{eq:moment_after_ptf_hol}
    \end{align}
    where 
    \begin{align*}
        &\mathcal{D}(\pmb{\alpha},\pmb{\sigma}) \coloneqq \sum_{\substack{d\geq1 \\ (d,p)=1}}\frac{1}{d}\sum_{\substack{n_1,n_2,n_3\geq1 \\ n_1=n_2n_3}}\frac{\chi(n_1)\overline{\chi}(n_2n_3)}{\sqrt{n_1n_2n_3}}h(\cdot) = \sum_{\substack{d,n_2,n_3\geq1 \\ (dn_2n_3,p)=1}}\frac{h(\cdot,n_2n_3,n_2,n_3,d,\cdot)}{dn_2n_3}, \\
        &\mathcal{S}(\chi,\pmb{\alpha},\pmb{\sigma}) \coloneqq \sum_{\substack{d\geq1 \\ (d,p)=1}}\frac{1}{d}\sum_{n_1,n_2,n_3\geq1}\frac{\chi(n_1)\overline{\chi}(n_2n_3)}{\sqrt{n_1n_2n_3}}\sum_{\substack{c\geq1 \\ p^a\mid c}}\frac{S_{\overline{\chi}^2}(n_1,n_2n_3;c)}{c}J_{\kappa-1}\bigg(\frac{4\pi\sqrt{n_1n_2n_3}}{c}\bigg)h(\cdot).
    \end{align*}
    So, in order to evaluate $\E_{\phi\,(\mo p^b)}[\mathcal{M}_h(\chi\phi,\pmb{\alpha})]$, it suffices to compute
    \begin{align*}
        \frac{1}{\varphi(p^b)}\sum_{\phi\,(\mo p^b)}((\chi\phi)(-1))^\sigma = \begin{dcases}
            1 &\text{ if }\sigma\equiv0\,(\mo 2) \\
            0 &\text{ otherwise}
        \end{dcases}
    \end{align*}
    and
    \begin{align}
        &\frac{1}{\varphi(p^b)}\sum_{\phi\,(\mo p^b)}(\chi\phi)((-1)^\sigma n_1)(\overline{\chi\phi})(n_2n_3)S_{\overline{\chi\phi}^2}(n_1,n_2n_3;c) \\
        &= 
        \begin{dcases}
            \chi(n_1)\overline{\chi}(n_2n_3)S_{\overline{\chi}^2}(n_1,n_2n_3;c) &\text{ if } \sigma\equiv0\,(\mo 2),\, p^{a+b}\mid c \\
            \\
            \frac{1}{1-p^{-1}}\sum_{\psi\,(\mo p^{a-b})}\overline{\psi}(-n_1n_2n_3\overline{c'}^2)\frac{\tau(\chi\psi)}{\tau(\chi\overline{\psi})}S(n_1\overline{p^a},n_2n_3\overline{p^a};c') &\text{ if }\sigma\equiv1\,(\mo 2),\,p^a\pdiv c \\ 
            \\
            0 &\text{ otherwise}
        \end{dcases},
    \end{align}
    which follows from the computations in Section \ref{sec:averaging_kloosterman_sums}.

    To evaluate $\mathcal{D}(\pmb{\alpha},\pmb{\sigma})$, one repeats the computations of Section \ref{sec:diag} with 
    \begin{align*}
        &L(s_1,s_2,s_3,\kappa,\pmb{\alpha},\pmb{\sigma}) \\
        &\coloneqq p^{a(s_1+s_2+s_3)}(s_1+s_2)(s_1+s_3)(s_2+s_3)\zeta^{(p)}(1+s_1+s_2)\zeta^{(p)}(1+s_1+s_3)\zeta^{(p)}(1+s_2+s_3) \\
        &\quad\, \cdot\prod_{1\leq j\leq 3}\frac{\Gamma_{\C}(s_j+\frac{\kappa}{2})}{\Gamma_{\C}(\alpha_j+\frac{\kappa}{2})}G(s_j-\sigma_j\alpha_j),
    \end{align*}
    which is now holomorphic on $\{(s_1,s_2,s_3)\in\C^3\,:\,\re s_j>-\kappa/2\text{ for } 1\leq j\leq 3\}$. We end up with 
    \begin{align}
        \mathcal{D}(\pmb{\alpha},\pmb{\sigma}) &= \frac{L(\sigma_1\alpha_1,\sigma_2\alpha_2,\sigma_3\alpha_3,\kappa,\pmb{\alpha},\pmb{\sigma})}{(\sigma_1\alpha_1+\sigma_2\alpha_2)(\sigma_2\alpha_2+\sigma_3\alpha_3)(\sigma_3\alpha_3+\sigma_1\alpha_1)} \label{eq:D_final_expression_hol} \\
        &\quad\, + N(\sigma_1\alpha_1,\sigma_2\alpha_2,\sigma_3\alpha_3,\kappa,\pmb{\alpha},\pmb{\sigma}) \\
        &\quad\, + N(\sigma_1\alpha_1,\sigma_3\alpha_3,\sigma_2\alpha_2,\kappa,(\alpha_1,\alpha_3,\alpha_2),(\sigma_1,\sigma_3,\sigma_2)) \\ 
        &\quad\, + N(\sigma_2\alpha_2,\sigma_3\alpha_3,\sigma_1\alpha_1,\kappa,(\alpha_2,\alpha_3,\alpha_1),(\sigma_2,\sigma_3,\sigma_1)) \\
        &\quad\, + \frac{1}{2}\frac{L(0,0,0,\kappa,\pmb{\alpha},\pmb{\sigma})}{(-\sigma_1\alpha_1)(-\sigma_2\alpha_2)(-\sigma_3\alpha_3)} + O_\varepsilon(p^{a(-\frac{1}{3}+\varepsilon)}),
    \end{align}
    where 
    \begin{align*}
        &N(\iota,\lambda,\mu,\kappa,\pmb{\alpha},\pmb{\sigma}) \coloneqq M(\iota,\lambda,\mu,\kappa,\pmb{\alpha},\pmb{\sigma}) + \frac{L(-\lambda,\lambda,\mu,\kappa,\pmb{\alpha},\pmb{\sigma})}{(-\lambda-\iota)(\mu+\lambda)(\mu-\lambda)} + \frac{L(\mu,-\mu,\mu,\kappa,\pmb{\alpha},\pmb{\sigma})}{(\mu-\iota)(-\mu-\lambda)(2\mu)}, \\
        &M(\iota,\lambda,\mu,\kappa,\pmb{\alpha},\pmb{\sigma}) \coloneqq \frac{1}{2\pi i}\int_{(\frac{1}{6})}\frac{L(s,-s,\mu,\kappa,\pmb{\alpha},\pmb{\sigma})}{(s-\iota)(-s-\lambda)(\mu-s)(\mu+s)}ds.
    \end{align*}
    As in the non-holomorphic case, the first term in \eqref{eq:D_final_expression_hol} contributes to the conjectured main term in \eqref{eq:claim_asymptotic_hol}, while the other terms cancel out in a way similar to the results in Section \ref{sec:cancellation} (after Poisson summation). It then remains to bound the off-diagonal contribution of terms with $m_1m_2m_3\neq0$ resulting from Poisson summation, which can be done analogously to Section \ref{sec:completing_the_proof}, this time using the stationary phase results from \cite[Section 10.4]{PeYpetersson}.
\end{proof}

\section*{References}
\printbibliography

\end{document}